\documentclass[12pt, twoside]{article}
\usepackage{amsmath,amsthm,amssymb}
\usepackage{times}
\usepackage{enumerate}

\def\titlerunning#1{\gdef\titrun{#1}}
\makeatletter
\def\author#1{\gdef\autrun{\def\and{\unskip, }#1}\gdef\@author{#1}}
\def\address#1{{\def\and{\\\hspace*{18pt}}\renewcommand{\thefootnote}{}%
\footnote {#1}}%
\markboth{\autrun}{\titrun}} \makeatother
\def\email#1{e-mail: #1}
\def\subjclass#1{{\renewcommand{\thefootnote}{}%
\footnote{\emph{Mathematics Subject Classification (2010):} #1}}}

\usepackage{amssymb,latexsym}
\usepackage{mathrsfs}
\usepackage{amssymb}
\usepackage{amsmath,latexsym}
\usepackage{amscd,latexsym} 
\usepackage[all]{xy}

\newcommand{\K}{{\mathbb K}}
\newcommand{\F}{{\mathbb F}}
\newcommand{\N}{{\mathbb N}}

\newcommand{\R}{{\mathbb R}}

\newcommand{\Z}{{\mathbb Z}}

\newtheorem{theorem}{Theorem}[section]

\newtheorem{corollary}[theorem]{Corollary}

\newtheorem{remark}[theorem]{Remark}

\newtheorem{lemma}[theorem]{Lemma}

\newtheorem{proposition}[theorem]{Proposition}
\newtheorem{claim}[theorem]{Claim}

\numberwithin{equation}{section}

\begin{document}
\baselineskip=17pt

\titlerunning{***}

\title{The splitting  lemmas for nonsmooth functionals
on Hilbert spaces I\footnote{IN MEMORY OF PROFESSOR SHUZHONG SHI
(1939--2008)}\;\footnote{The original version (arXiv:1102.2062 on Feb 2011)
  was split into 3 parts, this is the first one.}}


\author{Guangcun Lu}

\date{June 10, 2014}

\maketitle

\address{F1. Lu: School of Mathematical Sciences, Beijing Normal University,
Laboratory of Mathematics  and Complex Systems,  Ministry of
  Education,    Beijing 100875, The People's Republic
 of China; \email{gclu@bnu.edu.cn}}

\subjclass{Primary~58E05, 49J52, 49J45}

\begin{abstract}
The Gromoll-Meyer's generalized Morse lemma (so called splitting
lemma) near degenerate critical points on Hilbert spaces,   which is
one of key results in infinite dimensional Morse theory, is usually
stated for at least $C^2$-smooth functionals. It obstructs one using
Morse theory to study most of variational problems of form
$F(u)=\int_\Omega f(x, u,\cdots, D^mu)dx$ as in (\ref{e:1.1}). In
this paper we establish a splitting theorem and a shifting theorem
for a class of continuously directional differentiable functionals
(lower than $C^1$) on a Hilbert space $H$ which have higher
smoothness (but lower than $C^2$) on a densely and continuously
imbedded Banach space $X\subset H$ near a critical point lying in
$X$. (This splitting theorem generalize almost all previous ones to
my knowledge). Moreover, a new theorem of Poincar\'e-Hopf type and a
relation between critical groups of the functional on $H$ and $X$
are given.  Different from the usual implicit function theorem
method and dynamical system one our proof is to combine the ideas
of the Morse-Palais lemma due to Duc-Hung-Khai \cite{DHK} with some
techniques from \cite{JM, Skr, Va1}.  Our theory is applicable to
the Lagrangian systems on compact manifolds and boundary value
problems for a large class of nonlinear higher order elliptic
equations.

\end{abstract}

\tableofcontents

\section{Introduction }

\subsection{Motivation}
Morse theory is an important tool in critical point theory. Morse
inequalities, which provide the appropriate relations between global
topological notions and the critical groups of the critical points,
had been generalized to very general frameworks, see \cite{Ch, MaWi}
(for $C^1$-functionals on manifolds of infinite dimension) and
\cite{Cor} (for continuous functionals on complete metric spaces)
and the references therein. These inequalities and precise
computations of critical groups are extremely useful in
distinguishing different types of critical points and  obtaining
multiple critical points of a functional (cf. \cite{BaSzWi, Ch,
MaWi, PerAO}). However, the calculation of critical groups in
applications is a complex problem. Gromoll-Meyer's generalization of
Morse lemma to an isolated degenerate critical point in \cite{GrM},
also called the splitting theorem, provides a basic tool for the
effective computation of critical groups. Since then many authors
made their effort to improve the splitting theorem, see \cite{Ch,
Ho, MaWi, IoSch, JM, LiLiLiu, DHK, DHK1, LiLiLiuP} and related
historical and bibliographical notes in \cite[Remark 5.1]{Ch} and
\cite[page 202]{MaWi}. Probably, the  most convenient formulations
in the present applications are ones given in \cite[Th. 2.1]{Ch0}
(see also \cite[Th. 5.1]{Ch}) and
 \cite[Th.8.3]{MaWi} (see also \cite{MaWi0}). It was only assumed therein that $f$ is a
$C^2$-functional on a neighborhood $U$ of the origin $\theta$ in a
Hilbert space $H$ and that $\theta$ is an isolated critical point of
$f$ such that $0$ is either an isolated point of the spectrum
$\sigma(d^2f(\theta))$ or not in $\sigma(d^2f(\theta))$. This can be
used to deal with many elliptic boundary value problems of form
$\triangle u=f(x,u)$ on bounded smooth domains in $\R^n$ with
Dirichlet boundary condition.

However,  the action functionals in many important variational
problems are at most $C^{2-0}$ on spaces where the functionals can
satisfy the (PS) condition.  Let $\Omega\subset\R^n$ be
  a bounded domain with smooth boundary
$\partial\Omega$, $x=(x_1,\cdots,x_n)\in\R^n$, and let
$\alpha=(\alpha_1,\cdots,\alpha_n)$ be a multi-index of nonnegative
integer components $\alpha_i$, and $|\alpha|=\alpha_1+\cdots
+\alpha_n$ be its length. Denote by $M(m)$ the number of such
$\alpha$ of length $|\alpha|\le m$, and by
$\xi=\{\xi_\alpha:|\alpha|\le m\}\in\R^{M(m)}$. Consider the
variational problem
\begin{equation}\label{e:1.1}
F(u)=\int_\Omega f(x, u,\cdots, D^mu)dx,
\end{equation}
where  the function $f:\overline\Omega\times\R^{M(m)}\to\R,\; (x,
\xi)\mapsto f(x,\xi)$ is measurable in $x$ for all values of $\xi$,
and twice continuously differentiable in $\xi$ for almost all $x$;
and there are continuous, positive, nondecreasing function $g_1$ and
nonincreasing function $g_2$ such that the functions
$$
\overline\Omega\times\R^{M(m)}\to\R,\; (x, \xi)\mapsto
f_{\alpha\beta}(x,\xi)=\frac{\partial^2f(x,\xi)}{\partial
x_\alpha\partial x_\beta}
$$
satisfy:
\begin{eqnarray*}
&& |f_{\alpha\beta}(x,\xi)|\le
g_1\Biggl(\sum_{|\gamma|<m-n/2}|\xi_\gamma|\Biggr)\cdot\left(1+
\sum_{m-n/2\le |\gamma|\le
m}|\xi_\gamma|^{p_\gamma}\right)^{p_{\alpha\beta}},\\
&&
\sum_{|\alpha|=|\beta|=m}f_{\alpha\beta}(x,\xi)\eta_\alpha\eta_\beta\ge
g_2\Biggl( \sum_{|\gamma|<m-n/2}|\xi_\gamma|\Biggr)\cdot\left(
\sum_{|\alpha|= m}\eta^2_\alpha\right),
\end{eqnarray*}
for any $\eta\in\R^{M_0}$ ( $M_0=M(m)-M(m-1)$), where $p_\gamma$ is
an arbitrary positive number if $|\gamma|=m-\frac{n}{2}$, and
$p_\gamma=\frac{2n}{n-2(m-|\gamma|)}$ if $m-\frac{n}{2}<|\gamma|\le
m$, and $p_{\alpha\beta}=p_{\beta\alpha}$ are defined by
\begin{eqnarray*}
&&p_{\alpha\beta}=\left\{\begin{array}{ll}
 1-\frac{1}{p_\alpha}-\frac{1}{p_\beta}\quad
 &\hbox{if}\;|\alpha|=|\beta|=m,\\
 1-\frac{1}{p_\alpha},\quad &\hbox{if}\;m-\frac{n}{2}\le |\alpha|\le
 m,\; |\beta|<m-\frac{n}{2},\\
 1 \quad &\hbox{if}\; |\alpha|, |\beta|<m-\frac{n}{2},
 \end{array}
 \right.\\
 &&
 0<p_{\alpha\beta}<1-\frac{1}{p_\alpha}-\frac{1}{p_\beta}\quad\hbox{if}\;|\alpha|,\;|\beta|\ge
 m-\frac{n}{2},\;|\alpha|+|\beta|<2m.
\end{eqnarray*}
Generally speaking, under the assumptions above, as stated on the
pages 118-119 of \cite{Skr1} (see  \cite{Skr} for detailed
arguments) the functional $F$ in (\ref{e:1.1}) is $C^1$ and
satisfies the (PS) condition on $W^{m,2}_0(\Omega)$, and the mapping
$F'$ is
 only $G$-differentiable on $W^{m,2}_0(\Omega)$; moreover,  on Banach spaces
  on $W^{m,p}_0(\Omega)$ with $p>2$, it is $C^2$, but does not satisfy the (PS) condition.
  Furthermore, Morse inequalities were also obtained in  \cite[Chapter 5]{Skr} under the
  assumptions that the functional $F$ have only nondegenerate
  critical points. A similar question appears in some optimal control problems (see
 Vakhrameev \cite{Va1}).

Another important problem comes from the study of periodic solutions
of Lagrangian systems on compact manifolds, whose variational
functional is given by
\begin{equation}\label{e:1.2}
\mathcal{ L}_{\tau}(\gamma)=\int^{\tau}_0L(t, \gamma(t),\dot
\gamma(t))dt
\end{equation}
on the Riemannian-Hilbert manifold
 $H_{\tau}=W^{1,2}(\R/\tau\Z, M)\;(\subset C(\R/\tau\Z, M))$,
where $M$ is a $n$-dimensional compact smooth manifold without
boundary, and  $L:\R\times TM\to\R$ is a $C^2$-smooth function
satisfying the following conditions (L1)-(L3):
\begin{enumerate}
\item[(L1)]
 $L(t+1, q, v)=L(t, q, v)\quad\forall (t,q,v)$.
\end{enumerate}
In any local coordinates $(q_1,\cdots, q_n)$, there exist constants
$0<c<C$, depending on the local coordinates, such that
\begin{enumerate}
\item[(L2)] $c|{\bf u}|^2\le \sum_{ij}\frac{\partial^2 L}{\partial v_i\partial v_j}(t,q,
v)u_iu_j\le C|{\bf u}|^2\quad \forall {\bf
u}=(u_1,\cdots,u_n)\in\R^n$,
\item[(L3)] $\Bigl| \frac{\partial^2L}{\partial q_i\partial
v_j}(t,q,v)\Bigr|\le C(1+ |v|)$ \;\hbox{and}\; $\Bigl|
\frac{\partial^2L}{\partial q_i\partial q_j}(t,q,v)\Bigr|\le C(1+
|v|^2)\quad\forall (t, q, v)$.
\end{enumerate}
Under these assumptions the functional  $\mathcal{ L}_\tau$ is only
$C^{2-0}$ on the Hilbert manifold $H_\tau$ (as showed \cite{AbSc1}
recently), but satisfies the (PS) condition on $H_\tau$. The usual
regularity theory shows that all critical points of $\mathcal{
L}_\tau$ on $H_\tau$ sit in the Banach manifold
$X_\tau=C^1(\R/\tau\Z, M)$. It is very unfortunate that the (PS)
condition cannot be satisfied on $X_\tau$  though $\mathcal{
L}_\tau$ is $C^2$ on it. So far one do not find a suitable space on
which the functional $\mathcal{ L}_\tau$ is not only $C^2$ but also
satisfies the (PS) condition.

 The common points of the
two functionals above are: one hand on a Hilbert manifold they have
smoothness lower than $C^2$,  but satisfy the (PS) condition; on the
other hand their critical points are contained in a densely and
continuously imbedded Banach manifold on which the functional
possesses at least $C^2$ smoothness, but does not satisfy the (PS)
condition. To my knowledge there is no a suitable splitting lemma,
which can be used to deal with the above functionals. These motivate
us to look for a new splitting theorem.

With the regularity theory and prior estimation techniques of
differential equations our theory can also be applied to some
variational problems not satisfying our theorems (such as general
Tonelli Lagrangian systems and geodesics on Finsler manifolds, see
\cite[Remarks 5.9,6.1]{Lu1}, \cite{Lu4} and the references cited
therein) by modifying the original Euler-Lagrangian functions.

\subsection{Notion and terminology}
 Since there often exists some small differences in references we state
some necessary notions and terminologies for reader's conveniences.
Let $E_1$ and $E_2$ be two real normed linear spaces.
Denote by $L(E_1, E_2)$ the space of the continuous linear operator from $E_1$ to
$E_2$, and by $L(E_1)=L(E_1,E_1)$. A map $T$ from an open subset $U$ of $E_1$ to $E_2$
 is called {\it directional differentiable} at
$x\in U$ if for every $u\in E_1$ there exists an element of $E_2$,
denoted by $DT(x, u)$, such that $\lim_{t\to 0}\|\frac{T(x+
tu)-T(x)}{t}- DT(x,u)\|=0$; $DT(x, u)$ is called the {\it directional
derivative} of $T$ at $x$ in the direction $u$. If the map $U\times
E_1\to E_2, (x,u)\mapsto DT(x,u)$ is continuous we say $T$ to be
{\it continuously directional differentiable} on $U$. (This implies
that $T$ is G\^{a}teaux differentiable at every point of $U$ in the
following sense). If there exists a $B\in L(E_1, E_2)$ such that
$DT(x_0, u)=Bu\;\forall u\in E_1$, $T$ is called {\it G\^{a}teaux
differentiable} at $x_0\in U$, and $B$ is called the {\it
G\^{a}teaux derivative} of $T$ at $x_0$, denoted by $DT(x_0)$ (or
$T'(x_0)$). By Definition 3.2.2 of \cite{Schi}, $T$ is called {\it
strictly} G (G\^{a}teaux) {\it differentiable} at $x_0\in U$ if for
any $v\in E_1$,
$$
\|T(x+ tv)-T(x)-T'(x_0)(v)\|=o(|t|)\quad\hbox{as}\;x\to
x_0\;\hbox{and}\;t\to 0;
$$
if this convergence uniformly holds for $v$ in any compact subset we
say $T$ to be {\it strictly} H (Hadamard) {\it differentiable}
\footnote{This is called {\it strictly differentiable} in \cite[page
30]{Cl}. } at $x_0\in U$; moreover $T$ is called {\it strictly}
\footnote{It is also called strongly F-differentiable in some books,
for instance. Question 7) at the end of \cite[Chap.8, \S6]{Di}.} F
(Fr\'{e}chet) {\it differentiable} at $x_0\in U$ if
$$
\|T(x)-T(y)-T'(x_0)(x-y)\|=o(\|x-y\|)\quad\hbox{as}\;x\to
x_0\;\hbox{and}\;y\to x_0
$$
(this implies that $T$  has Fr\'{e}chet derivative  $T'(x_0)$ at
$x_0$). By \cite[Prop.2.2.1]{Cl} or \cite[Prop.3.2.4(iii)]{Schi},
$T$ is  strictly H-differentiable at $x_0\in U$ if and only if $T$
is locally Lipschitz  continuous around $x_0$ and strictly
G-differentiable at $x_0\in U$. Specially, the strict
$F$-differentiability of $T$ at $x_0$ implies  that $T$ is Lipschitz
continuous in some neighborhood of $x_0$. By
\cite[Prop.3.4.2]{Schi}, the continuous F-differentiability of $T$
at $x_0$ implies that $T$ is strictly F-differentiable at $x_0$. If
$T$ is $F$-differentiable in $U$, then $dT=T'$ is continuous at
$x_0\in U$ (i.e. $T$ is continuously differentiable at $x_0$) if and
only if $T$ is strictly $F$-differentiable at $x_0$, see Questions
3a) and 7a) at the end of \cite[Chap.8, \S6]{Di}. By
Proposition~\ref{prop:B.1}  the continuously directional
differentiability of $T$ in $U$ implies the strict
$H$-differentiability of $T$ in $U$ (and thus the locally Lipschitz
continuality of $T$ in $U$).

\subsection{Method and overview}

 The main methods to the splitting lemma in past references are
 the implicit function theorem method such as \cite{GrM} and
dynamical system one as in  \cite[Th. 5.1]{Ch} and
 \cite[Th.8.3]{MaWi}. Our method is different from theirs
 completely. Recently, Duc-Hung-Khai \cite{DHK}
gave a new proof to the Morse-Palais lemma based on elementary
differential calculus. It seems that the parameterized versions of
the new Morse lemma cannot be applied to the above two typical
functionals yet. After carefully analyzing the functionals we
combine it with some techniques from \cite{JM, Skr, Va1} to
successfully design a splitting lemma which is applicable to our
above functionals.   For completeness and reader's convenience we
state the parameterized versions of Duc-Hung-Khai's Morse-Palais
lemma in \cite{DHK}  and outline its proof in Appendix~\ref{app:A}.
 Some results on functional analysis are given in Appendix~\ref{app:B}.

In Section~\ref{sec:2} we state our main results, which include  a new
 splitting lemma, Theorem~\ref{th:2.1}, and the corresponding shifting
theorem, Corollary~\ref{cor:2.6}. We also obtain critical group
characteristics for local minimum and critical points of mountain
pass type under weaker conditions in
Corollaries~\ref{cor:2.7},~\ref{cor:2.9}, respectively.
Corollary~\ref{cor:2.5} and Theorem~\ref{th:2.10} study relations
between critical groups of a  functional and its restriction on a
densely imbedded Banach space, which are very key for our  work
\cite{Lu4}. A theorem of Poincar\'e-Hopf type,
Theorem~\ref{th:2.12}, is proved in Section~\ref{sec:5}. We also
study the functor properties of our splitting lemma in
Section~\ref{sec:6}, and estimate behavior of the functional
$\mathcal{L}$ of Theorem~\ref{th:2.1} near $\theta$ in
Section~\ref{sec:7}. As concluding remarks it is shown in
Section~\ref{sec:8} that the most results in  Theorem~\ref{th:2.1}
still hold true under weaker conditions.



These result  have been used in \cite{Lu3}  to generalize some
previous results on computations of critical groups and some
critical point theorems  to weaker versions.

 This paper consists of the sections~1,2 and the appendix of
\cite{Lu2}, which is not to be published elsewhere. The fourth
section of \cite{Lu2} has been rewritten and extended into a
separate paper. The author would like to express his deep gratitude
to the anonymous referee for many valuable revision suggestions and
for pointing out many misprints.


\section{Statements of  main results}\label{sec:2}
\setcounter{equation}{0}

 Let $H$ be a Hilbert space with inner product $(\cdot,\cdot)_H$
and the induced norm $\|\cdot\|$, and let $X$ be a Banach space with
norm $\|\cdot\|_X$, such that
\begin{enumerate}
\item[\bf (S)]  $X\subset H$ is dense in $H$ and
 the inclusion $X\hookrightarrow H$ is continuous, i.e. we
may assume $\|x\|\le \|x\|_X\;\forall x\in X$.
\end{enumerate}
For an open neighborhood $V$ of the origin $\theta\in H$, $V\cap X$
is also an open neighborhood of $\theta$ in $X$, \textsf{denoted by $V^X$ for clearness without special
statements}.
 Suppose that  a functional $\mathcal{ L}:V\to\mathbb{R}$ satisfies the following
conditions:
\begin{enumerate}
\item[\bf (F1)] $\mathcal{L}$ is  continuously directional
differentiable (and thus $C^{1-0}$) on $V$.

\item[\bf (F2)] There exists a  continuously directional
differentiable (and thus $C^{1-0}$) map $A: V^X\to X$, which is
 strictly Fr\'{e}chet
differentiable  at $\theta$, such that
$$
D\mathcal{ L}(x)(u)=(A(x), u)_H\quad\forall x\in V^X\;\hbox{and}\; u\in X.
$$
(This actually implies that $\mathcal{ L}|_{V^X}\in C^1(V^X,
\R)$.)
\item[\bf (F3)] There exists a map $B$ from
$V^X$ to the space $L_s(H)$ of bounded self-adjoint linear
operators of $H$ such that
$$
(DA(x)(u), v)_H=(B(x)u, v)_H\;\forall x\in V^X\;\hbox{and}\; u,
v\in X.
$$
(This and (F1)-(F2) imply: (a) $A$ is G\^{a}tuax differentiable and
$DA(x)=B(x)|_X$ for all $x\in V^X$, (b) $B(x)(X)\subset
X\;\forall x\in V^X$, (c) $d(\mathcal{ L}|_{V^X})$ is
strictly Frech\'et differentiable at $\theta\in V^X$, and
$d^2(\mathcal{ L}|_{V^X})(\theta)(u,v)=(B(\theta)u,v)_H$ for any
$u, v\in X$.)
\item[\bf (C1)] The origin $\theta\in X$ is a critical point of $\mathcal{ L}|_{V^X}$ (and thus $\mathcal{ L}$),
$0$ is either not in the spectrum $\sigma(B(\theta))$ or is an
isolated point of $\sigma(B(\theta))$. \footnote{The claim in the
latter sentence is actually implied in the following condition (D)
by Proposition~\ref{prop:B.2}. In order to state some results
without the condition (D) we still list it.}

\item[\bf (C2)]  If $u\in H$ such that $B(\theta)(u)=v$ for some
$v\in X$, then $u\in X$.

\item[\bf (D)] The map $B:V^X\to
L_s(H)$  has a decomposition \footnote{Actually, this and (D4)
imply the claim in the second sentence in (C1) by
Proposition~\ref{prop:B.2}.}
$$
B(x)=P(x)+ Q(x)\quad\forall x\in V^X,
$$
where $P(x):H\to H$ is a positive definitive linear operator and
$Q(x):H\to H$ is a compact linear  operator with the following
properties:
\begin{enumerate}
\item[\bf (D1)]  All eigenfunctions of the operator $B(\theta)$ that correspond
to negative eigenvalues belong to $X$;

\item[\bf (D2)] For any sequence $\{x_k\}\subset
V\cap X$ with $\|x_k\|\to 0$ it holds that
$\|P(x_k)u-P(\theta)u\|\to 0$ for any $u\in H$;

\item[\bf (D3)] The  map $Q:V\cap X\to
L(H)$ is continuous at $\theta$ with respect to the topology
induced from $H$ on $V\cap X$;

\item[\bf (D4)] For any sequence $\{x_n\}\subset V\cap X$ with $\|x_n\|\to 0$ (as $n\to\infty$), there exist
 constants $C_0>0$ and $n_0\in\N$ such that
$$
(P(x_n)u, u)_H\ge C_0\|u\|^2\quad\forall u\in H,\;\forall n\ge n_0.
$$
\end{enumerate}\end{enumerate}

Sometimes we need to replace the condition ({\rm D4}) by the
following slightly stronger
\begin{enumerate}
\item[\bf (D4*)] There exist positive constants $\eta_0>0$ and  $C'_0>0$ such that
$$
(P(x)u, u)\ge C'_0\|u\|^2\quad\forall u\in H,\;\forall x\in
B_H(\theta,\eta_0)\cap X.
$$
\end{enumerate}

Here is a way looking for the map $B$. Suppose that $\mathcal{
L}|_{V^X}$ is twice G\^{a}teaux differentiable at every point
$x\in V^X$, i.e. for any $u_1, u_2\in X$ the  limit
\begin{eqnarray*}
D\mathcal{ L}|_{V^X}(x;u_1,u_2) = \lim_{t_2\to 0}\lim_{t_1\to
0}\frac{1}{t_1t_2}\triangle^2_{t_1u_1,t_2u_2}\mathcal{ L}(x)
\end{eqnarray*}
exists and is linear continuous with respect to $u_i$, $i=1,2$,
where
$$\triangle^2_{t_1u_1,t_2u_2}\mathcal{ L}(x)=\mathcal{ L}(x+
t_1u_1+ t_2u_2)- \mathcal{L}(x+ t_1u_1)-\mathcal{ L}(x+ t_2u_2)-
\mathcal{L}(x).
$$
By (F2) the map $A: V^X\to X$ is G\^{a}teaux
differentiable and
$$D\mathcal{ L}|_{V^X}(x;u_1,u_2)=(A'(x)u_2,
u_1)_H\quad\forall x\in V^X,\; u_1, u_2\in X.
$$
If $(u_1,
u_2)\mapsto D\mathcal{ L}|_{V^X}(x;u_1,u_2)$ is symmetric then
$A'(x)\in L(X)$ is self-adjoint with respect to the inner $(\cdot,
\cdot)_H$. By Question 17) at the end of \cite[Chap.11, \S5]{Di},
$A'(x)$ can be extended into an element $\hat B(x)\in L_s(H)$ with
the following properties: (a) $\|\hat
B(x)\|_{L(H)}\le\rho_X(A'(x))\le\|A'(x)\|_{L(X)}$ and $\sigma(\hat
B(x))\subset\sigma(A'(x))$, (b) if $A'(x)$ is compact in $(X,
\|\cdot\|_X)$ so is $\hat B(x)$ in $(H, \|\cdot\|)$. In the case, if
 $B$ is a map satisfying the conditions ({\rm F3}), ({\rm C1})-({\rm
C2}) and ({\rm D}), it holds that $B(x)=\hat B(x)\;\forall
x\in V^X$.

By the assumption (D) each $B(x)$ is Fredholm. In particular,
$H^0:={\rm Ker}(B(\theta))$ is finitely dimensional. Let
$H^\pm:=(H^0)^\bot$ be the range of $B(\theta)$. There exists an
orthogonal decomposition $H=H^0\oplus H^\pm=H^0\oplus H^-\oplus
H^+$, where $H^-$ and $H^+$ are subspaces invariant under
$B(\theta)$ such that $B(\theta)|_{H^+}$ is positive definite and
$B(\theta)|_{H^+}$ is negative definite. Clearly, we have also
\begin{equation}\label{e:2.1}
\left.\begin{array}{ll}
&(B(\theta)u, v)_H=0\;\forall u\in H^+\oplus H^-,\;v\in H^0,\\
&(B(\theta)u, v)_H=0\;\forall  u\in H^-\oplus H^0,\;v\in H^+,\\
&(B(\theta)u, v)_H=0\;\forall u\in H^+\oplus H^0,\;v\in H^-.
\end{array}\right\}
\end{equation}
 By the condition (C1)
there exists a small $a_0>0$ such that $[-2a_0,
2a_0]\cap\sigma(B(\theta))$ at most contains a point $0$. Hence
\begin{equation}\label{e:2.2}
\left.\begin{array}{ll}
  (B(\theta)u, u)_H\ge
2a_0\|u\|^2\quad\forall u\in H^+,\\
  (B(\theta)u, u)_H\le
-2a_0\|u\|^2\quad\forall u\in H^-.
\end{array}\right\}
\end{equation}
The conditions (C2) and (D) imply that both $H^0$ and $H^-$ are
finitely dimensional subspaces contained in $X$ by
Proposition~\ref{prop:B.2}.  Denote by $P^\ast$ the orthogonal
projections onto $H^\ast$, $\ast=+, -, 0$, and by $X^\ast=X\cap
H^\ast=P^\ast(X),\;\ast=+, -$. Then $X^+$ is dense in $H^+$, and
$(I-P^0)|_X=(P^++P^-)|_X: (X, \|\cdot\|_X)\to (X^\pm, \|\cdot\|)$ is
also continuous because all norms are equivalent on a linear space
of finite dimension, where $X^\pm:=X\cap (I-P^0)(H)=X\cap H^\pm=X^-+
P^+(X)=X^-+ H^+\cap X$.  These give the following topological direct
sum decomposition:
$$
X=H^0\oplus X^\pm=H^0\oplus X^+\oplus X^-.
$$
Let $\nu=\dim H^0$ and $\mu=\dim H^-$. We  call them the {\it
nullity} and the {\it Morse index} of critical point $\theta$ of
$\mathcal{ L}$, respectively.  In particular, the critical point
$\theta$ is said to be {\it nondegenerate} if $\nu=0$. Since the norms
$\|\cdot\|$ and $\|\cdot\|_X$ are equivalent on the finite dimension
space $H^0$ we shall not point out the norm used without occurring
of confusions. In this paper, for a normed
vector space $(E, \|\cdot\|)$ and $\delta>0$ let $B_E(\theta,
\delta)=\{x\in E\,:\,\|x\|=\|x-\theta\|<\delta\}$ and $\bar
B_E(\theta, \delta)=\{x\in E\,:\,\|x\|\le\delta\}$. Moreover, \textsf{we always use $\theta$
to denote the origins of all linear spaces without occurring of confusions}.

\begin{theorem}\label{th:2.1}
Under the above assumptions {\rm (S)}, {\rm (F1)-(F3)} and {\rm
(C1)-(C2)}, {\rm (D)}, if $\nu>0$ there exist a positive
$\epsilon\in\R$, a (unique) Lipschitz continuous map
$h:B_{H^0}(\theta,\epsilon)=B_{H}(\theta,\epsilon)\cap H^0\to X^\pm$
satisfying $h(\theta)=\theta$ and
\begin{equation}\label{e:2.3}
 (I-P^0)A(z+ h(z))=0\quad\forall z\in B_{H^0}(\theta,\epsilon),
 \end{equation}
an open neighborhood $W$ of $\theta$ in $H$ and an origin-preserving
homeomorphism
\begin{equation}\label{e:2.4}
\Phi: B_{H^0}(\theta,\epsilon)\times
\left(B_{H^+}(\theta,\epsilon) +
B_{H^-}(\theta,\epsilon)\right)\to W
\end{equation}
of form $\Phi(z, u^++ u^-)=z+ h(z)+\phi_z(u^++ u^-)$ with
$\phi_z(u^++ u^-)\in H^\pm$  such that
\begin{equation}\label{e:2.5}
\mathcal{ L}\circ\Phi(z, u^++ u^-)=\|u^+\|^2-\|u^-\|^2+ \mathcal{
L}(z+ h(z))
\end{equation}
for all $(z, u^+ + u^-)\in B_{H^0}(\theta,\epsilon)\times
\left(B_{H^+}(\theta,\epsilon) +
B_{H^-}(\theta,\epsilon)\right)$, and that
\begin{equation}\label{e:2.6}
\Phi\left(B_{H^0}(\theta,\epsilon)\times
\bigl(B_{H^+}(\theta,\epsilon)\cap X +
B_{H^-}(\theta,\epsilon)\bigr)\right)\subset X.
\end{equation}
Moreover, the homeomorphism $\Phi$ has also properties:
\begin{enumerate}
\item[\bf (a)] For
each $z\in B_{H^0}(\theta,\epsilon)$, $\Phi(z, \theta)=z+ h(z)$,
$\phi_z(u^++ u^-)\in H^-$ if and only if $u^+=\theta$;

\item[\bf (b)] The restriction of $\Phi$ to
$B_{H^0}(\theta,\epsilon)\times B_{H^-}(\theta, \epsilon)$ is a
homeomorphism from $B_{H^0}(\theta,\break \epsilon)\times B_{H^-}(\theta,
\epsilon)\subset X\times X$ onto
$\Phi(B_{H^0}(\theta,\epsilon)\times B_{H^-}(\theta,
\epsilon))\subset X$ even if the topologies on these two sets
 are chosen as the induced one
by $X$.
\end{enumerate}
The map $h$ and the function $B_{H^0}(\theta,\epsilon)\ni z\mapsto
\mathcal{ L}^\circ(z):=\mathcal{ L}(z+ h(z))$ \footnote{If $A$ is
$C^1$ then maps $h$ and $\mathcal{ L}^\circ$ have higher smoothness
too, see Remark~\ref{lem:3.2}.} also satisfy:
\begin{enumerate}
\item[\bf (i)]  The map $h$ is strictly Fr\'echet differentiable at
$\theta\in H^0$ and
$$
h'(\theta)z=-[(I-P^0)A'(\theta)|_{X^\pm}]^{-1}(I-P^0)A'(\theta)z
\quad\forall z\in H^0;
$$
\item[\bf (ii)]  $\mathcal{ L}^\circ$ is $C^{2-0}$,
 $$
d\mathcal{ L}^\circ(z_0)(z)=(A(z_0+ h(z_0)), z)_H\quad\forall z_0\in
B_{H^0}(\theta, \epsilon),\; z\in H^0,
 $$
 and $d\mathcal{ L}^\circ$ is strictly F-differentiable at $\theta\in
H^0$ and $d^2\mathcal{ L}^\circ(\theta)=0$;

\item[\bf (iii)] If $\theta$ is an isolated critical point of $\mathcal{
L}|_{V^X}$, then $\theta$ is also an isolated critical point
of $\mathcal{L}^\circ$.
\end{enumerate}
\end{theorem}

If the strictly Fr\'{e}chet
differentiability   at $\theta$ of the map $A: V^X\to X$ in (F2) is replaced by weaker conditions
we shall show in Section~\ref{sec:8} that the most results in Theorem~\ref{th:2.1} still hold true.

 Under the conditions (L1)-(L3) it was proved in \cite{Lu1} that
 the functional $\mathcal{L}_\tau$ in (\ref{e:1.2}) satisfies the
 assumptions of Theorem~\ref{th:2.1} near a critical point of it.
In fact,  a special version of Theorem~\ref{th:2.1} was used there.
As stated in \cite[\S 5.2]{Skr} the arguments of \cite[Chap.3]{Skr}
showed that the functional $F$ in (\ref{e:1.1}) satisfies the
 assumptions of Theorem~\ref{th:2.1} near a critical point of it
 too. Our frame conditions in Theorem~\ref{th:2.1} seem strange and
 complex. But they come from abstract and analysis for the studies
in \cite{Skr}. Of course, the theory of this paper can be used to
improve one of \cite{Skr}. This work is in progress.

\begin{remark}\label{rm:2.2}
{\rm (i) Note that our proof only use the Banach fixed point theorem
or the implicit function theorem in the case $H^0\ne\{0\}$. If
$H^0=\{0\}$, we do not require the completeness of $(X,
\|\cdot\|_X)$, that is, the condition ({\rm S}) can be replaced by
the following
\begin{enumerate}
\item[\bf (S')] $(X, \|\cdot\|_X)$ is a normed vector space, $X\subset H$ is dense in $H$ and
 the inclusion $X\hookrightarrow H$ is continuous, i.e. we
may assume $\|x\|\le \|x\|_X\;\forall x\in X$;
\end{enumerate}
And the conclusions of Theorem~\ref{th:2.1} become: There exist a
positive $\epsilon\in\R$, an open neighborhood $W$ of $\theta$ in
$H$ and an origin-preserving homeomorphism, $\phi: B_{H^+}(\theta,\epsilon) +
B_{H^-}(\theta,\epsilon)\to W$,
 such that
\begin{equation}\label{e:2.7}
\mathcal{L}\circ\phi(u^++ u^-)=\|u^+\|^2-\|u^-\|^2
\end{equation}
for all $(u^+, u^-)\in B_{H^+}(\theta,\epsilon)\times
B_{H^-}(\theta,\epsilon)$, and that
$$\phi\left((B_{H^+}(\theta,\epsilon)\cap X) +
B_{H^-}(\theta,\epsilon)\right)\subset X.
$$
Moreover, $\phi(u^++
u^-)\in H^-$ if and only if $u^+=\theta$, and  the restriction of
$\phi$ to $B_{H^-}(\theta, \epsilon)$ is a homeomorphism from
$B_{H^-}(\theta, \epsilon)\subset X$ onto $\phi(B_{H^-}(\theta,
\epsilon))\subset X$ even if the topologies on $B_{H^-}(\theta,
\epsilon)\subset X$ and $\phi(B_{H^-}(\theta, \epsilon))\subset X$
are chosen as the induced ones by $X$.\\
(ii) Suppose that $\mathcal{ L}$ is only defined on $V\cap X$ and
that the condition ({\rm F1}) can be replaced by the following
\begin{enumerate}
\item[\bf (F1')] $\mathcal{ L}$ is  continuously directional differentiable (and so $C^{1-0}$) on
$V\cap X$ with respect to the topology of $H$.
\end{enumerate}
Then the origin-preserving homeomorphism in (\ref{e:2.4}) should be
changed into
\begin{equation}\label{e:2.8}
\Phi: B_{H^0}(\theta,\epsilon)\times
\left(B_{H^+}(\theta,\epsilon)\cap X +
B_{H^-}(\theta,\epsilon)\right)\to W\cap X
\end{equation}
(with respect to the topology of $H$), which satisfies (\ref{e:2.5})
for all $(z, u^+, u^-)\in B_{H^0}(\theta,\epsilon)\times
\left(B_{H^+}(\theta,\epsilon)\cap X +
B_{H^-}(\theta,\epsilon)\right)$. }
\end{remark}

\begin{remark}\label{rm:2.3}{\rm
Since Lemmas~\ref{lem:3.3},~\ref{lem:3.4} are only used in the proof
of Lemma~\ref{lem:3.5}. Carefully checking the proof of the latter
one easily see that the condition ({\rm D}) can be replaced by the
following
\begin{enumerate}
\item[\bf (D')] There exist a  small neighborhood $U\subset V$ of $\theta$ in
$H$, a positive number $c_0$ and a function $\omega: U\cap X\to [0,
\infty)$ with property $\omega(x)\to 0$ as $x\in U\cap X$ and
$\|x\|\to 0$, to satisfy
\begin{enumerate}
\item[(${\rm D}'_1$)] The kernel $H^0$ and negative definite
subspace $H^-$ of $B(\theta)$ are finitely dimensional subspaces
contained in $X$; \footnote{It seems to be sufficient for us to
assume only that $H^0\subset X$ and is closed in $X$. }
\item[(${\rm D}'_2$)] $(B(x)v, v)_H\ge c_0\|v\|^2\;\forall v\in H^+$;
\item[(${\rm D}'_3$)] $|(B(x)u,v)_H-(B(\theta)u,v)_H|\le\omega(x)\|u\|\cdot\|v\|\quad\forall u\in H,  v\in
H^-\oplus H^0$;
\item[(${\rm D}'_4$)] $(B(x)u,u)_H\le-c_0\|u\|^2\;\forall u\in H^-$.
\end{enumerate}
\end{enumerate}}
\end{remark}

\begin{remark}\label{rm:2.4}{\rm
When $(X, \|\cdot\|_X)=(H, \|\cdot\|)$  the conditions (F1)-(F3) are
reduced to:
\begin{enumerate}
\item[\bf (F)] $\mathcal{ L}$ is $C^1$, $\nabla\mathcal{ L}$ is  continuously directional
differentiable (and so G\^{a}teaux differentiable) in $V$ and
strictly Fr\'echet differentiable at $\theta\in H$, and
$B(x):=D(\nabla\mathcal{ L})(x)\in L_s(H)$ for any $x\in V$.
\end{enumerate}
Clearly, this holds if $\mathcal{ L}\in C^2(V,\R)$. In fact, the
condition (C1) for $B(\theta)=d^2\mathcal{ L}(\theta)$ also imply
the condition (D) in the case $\dim H^0\oplus H^-<\infty$. In order
to see this we can write $B(x)=P(x)+ Q(x)$, where
$P(x)=P^+B(x)-P^-B(x)+ P^0$ and $Q(x)=2P^-B(x)+ P^0+ P^0B(x)$. The
latter is finite rank and therefore compact. The continuity of the
map $B:V\to L_s(H)$ implies that  both maps $P$ and $Q$ are
continuous, and that there exists a $\delta>0$ such that
$$
\|B(x)-B(\theta)\|_{L(H)}<\min\{a_0, 1\}/4\quad\forall x\in
B_H(\theta,\delta).
$$
Note that $(P(\theta)u,u)_H\ge \min\{a_0,
1\}\|u\|^2\;\forall u\in H$ and that
$$
|(P(x)u,u)_H-(P(\theta)u,u)_H|\le
2\|B(x)-B(\theta)\|_{L(H)}\cdot\|u\|^2\quad\forall u\in H.
$$
We get
$$
(P(x)u,u)_H\ge \frac{\min\{a_0, 1\}}{2}\|u\|^2\;\forall u\in H.
$$
These show that the condition (D) is satisfied. Hence
Theorem~\ref{th:2.1} is a generalization of \cite[Th.3]{Ho} and
\cite[Th.8.3]{MaWi}, \cite[Th.2.2]{LiLiLiuP},  and \cite[Th.5.1.
p.44]{Ch} in the case $\dim H^0\oplus H^-<\infty$ (a condition
naturally satisfied in applications). Since the strictly Fr\'echet
 differentiability of $\nabla\mathcal{ L}$  at
$\theta\in H$ implies that $\nabla\mathcal{ L}$ is $C^{1-0}$ near
$\theta$,  we cannot guarantee that Theorem~\ref{th:2.1} include
\cite[Cor.3]{IoSch}. ({\it Note}: By \cite[Th.4.5]{Con} the
assumptions in \cite[Th.1.2]{MoSo} is actually the same as that of
\cite[Cor.3]{IoSch}, but the author cannot verify the equalities
$h_2\circ h_3=id=h_3\circ h_2$ below (2.19) of \cite{MoSo}.) }
\end{remark}

 For an open
neighborhood $W$ of $\theta$ in $H$, we write $W^X=W\cap X$ as an  open neighborhood
of $\theta$ in $X$. Note that $(\mathcal{ L}|_{V^X})_0\cap(W\cap
X)=(\mathcal{ L}|_{V^X})_0\cap W=\mathcal{L}_0\cap W^X$.

\begin{corollary}\label{cor:2.5}
For any Abel group ${\bf K}$ and an open neighborhood $W$ of $\theta$ in
$H$, the inclusion
$$
I^{xw}: (\mathcal{ L}_0\cap W^X, \mathcal{ L}_0\cap W^X\setminus\{\theta\})\hookrightarrow (\mathcal{ L}_0\cap W,
\mathcal{ L}_0\cap W\setminus\{\theta\})
$$
induce surjective homomorphisms
$$
H_\ast(\mathcal{ L}_0\cap W^X, \mathcal{ L}_0\cap W^X
\setminus\{\theta\};{\bf K})\rightarrow H_\ast(\mathcal{ L}_0\cap
W, \mathcal{ L}_0\cap W\setminus\{\theta\};{\bf K}).
$$
\end{corollary}
Hereafter $H_q(A,B;{\bf K})$ denotes the $q$th relative singular homology group of
a pair $(A,B)$ of topological spaces with coefficients in ${\bf K}$.

One of important applications of the splitting lemma is to compute
critical groups of critical points. Recall that for $q\in\N\cup\{0\}$ {\it the $q$th critical group} (with coefficients in ${\bf K}$)
of a real continuous functional $f$ on a metric space $\mathcal{M}$ at a point $x\in\mathcal{M}$
is defined by
$$
C_q(f,x;{\bf K})=H_q(f_c\cap U, f_c\cap U\setminus\{x\};{\bf K}),
$$
where $c=f(x)$ and $U$ is a neighborhood of $x$ in $\mathcal{M}$.
The definition of the critical groups are independent of the special choice of $U$ because of
the excision property of the singular homology. If $\mathcal{M}$ is a Banach space and
$f$ is $C^1$ then the $q$th critical group of an isolated critical point $x$ may equivalently be defined as
$$
C_q(f,x;{\bf K})=H_q((\mathring{f}_c\cup\{x\})\cap U,
\mathring{f}_c\cap U;{\bf K}),
$$
where $c=f(x)$, $\mathring{f}_c=\{f<c\}$ and  $U$ is as above. (See
\cite[Prop.3.7]{Cor}).

If the critical point $\theta$ of $\mathcal{ L}$ is isolated, then
it is also an isolated critical point of $\mathcal{ L}|_{V^X}$. By
Theorem~\ref{th:2.1} $\theta\in H^0$ is an isolated critical point
of $\mathcal{ L}^\circ$. Since $\mathcal{ L}^\circ$ is also
$C^{2-0}$ and $\dim H^0<\infty$ we can construct a $C^{2-0}$
function on $H^0$ that satisfies the (PS) condition and is equal to
$\mathcal{ L}^\circ$ near $\theta$.  With the same proof method as
in \cite[Th.8.4]{MaWi} or \cite[Th.5.1.17]{Ch1} we can use
Theorem~\ref{th:2.1} to derive:

\begin{corollary}[Shifting]\label{cor:2.6}
Under the assumptions of Theorem~\ref{th:2.1},  if $\theta$ is an
isolated critical point  of $\mathcal{ L}$, for any Abel group ${\bf
K}$ it holds that
$$
C_q(\mathcal{ L}, \theta;{\bf K})\cong C_{q-\mu}(\mathcal{
L}^{\circ}, \theta; {\bf K})\quad\forall q=0, 1,\cdots,
$$
where  $\mathcal{ L}^{\circ}(z)=\mathcal{ L}(h(z)+z)$. {\rm
(}Consequently, $C_q(\mathcal{ L}, \theta;{\bf K})=0$ for $q\notin
[\mu, \mu+\nu]$, and $C_q(\mathcal{ L}, \theta;{\bf K})$ is
isomorphic to a finite direct sum $r_1{\bf K}\oplus\cdots\oplus
r_s{\bf K}$ for each $q\in [\mu, \mu+\nu]$, where each
$r_j\in\{0,1\}$, see Proposition~\ref{prop:4.5}.{\rm )}
\end{corollary}

Corresponding with Proposition 3.2 of \cite{BaChWa}, but no
requirement for the (PS) condition, we have

\begin{corollary}\label{cor:2.7}
Under the assumptions of Theorem~\ref{th:2.1}, if $\theta$ is an
isolated critical point  of $\mathcal{ L}$, the following are
equivalent.
\begin{enumerate}
\item[{\bf (i)}] $\theta$ is a local minimum;
\item[{\bf (ii)}] $C_q(\mathcal{ L}, \theta;{\bf K})\cong \delta_{q0}{\bf K}\quad\forall
q\in\Z$;

\item[{\bf (iii)}] $C_0(\mathcal{ L}, \theta;{\bf K})\ne 0$.
\end{enumerate}
\end{corollary}

Actually our proof shows that (iii) implies $\theta$ to be a strict
minimum.

Since $d^2\mathcal{ L}|_{V^X}(\theta)(u,v)=(B(\theta)u,v)_H\;\forall
u, v\in X$ we arrive at $H^0=\{\theta\}=H^-$ provided that
$d^2(\mathcal{ L}|_{V^X})(\theta)(u,u)>0$ for any $u\in
X\setminus\{\theta\}$. From Theorem~\ref{th:2.1} or Step 3 in the
proof of Lemma~\ref{lem:3.5} we easily derive  a similar conclusion
of Tromba's main result Theorem 1.3 in \cite{Tro2} without
requirement for completeness of $(X, \|\cdot\|_X)$.

\begin{corollary}\label{cor:2.8}
Under the assumptions of Theorem~\ref{th:2.1}, but no requirement
for completeness of $(X, \|\cdot\|_X)$, i.e., the condition $({\rm
S}$) is replaced by $({\rm S}')$, suppose also that $d^2(\mathcal{
L}|_{V^X})(\theta)(u,u)>0$ for any $u\in X\setminus\{\theta\}$.
Then $\theta$ is a strict minimum for $\mathcal{L}$ and thus
$\mathcal{ L}|_{V^X}$.
\end{corollary}

According to Hofer \cite{Ho} the critical point $\theta$ is called
{\it mountain pass type} if for any small neighborhood $\mathcal{
O}$ of $\theta$ in $H$ the set $\{x\in\mathcal{ O}\,|\,\mathcal{
L}(x)<0\}$ is nonempty and not path-connected.

\begin{corollary}\label{cor:2.9}
Under the assumptions of Theorem~\ref{th:2.1} (and hence without the {\rm
(PS)} condition), let $\theta$ be an isolated critical point  of
$\mathcal{ L}$ with Morse index $\mu$ and nullity $\nu$.
\begin{enumerate}
\item[{\bf (i)}] If $C_1(\mathcal{ L}, \theta;{\bf K})\ne 0$ and $\nu=\dim{\rm
Ker}(B(\theta))=1$ then
$$
C_q(\mathcal{ L}, \theta;{\bf K})\cong
\delta_{q1}{\bf K}\;\forall q\in\Z;
$$

\item[{\bf (ii)}] If  $\nu=\dim{\rm Ker}(B(\theta))=1$ in the case $\mu=\dim H^-=0$, then $\theta$ is mountain pass type
if and only if $C_q(\mathcal{ L}, \theta;{\bf K})\cong
\delta_{q1}{\bf K}\;\forall q\in\Z$;

\item[{\bf (iii)}] If $C_\mu(\mathcal{ L}, \theta;{\bf K})\ne 0$,  then $C_q(\mathcal{ L}, \theta;{\bf K})\cong \delta_{q\mu}{\bf
K}\;\forall q\in\Z$.
\end{enumerate}
\end{corollary}

The proofs of (i) and (ii) are the same as those of
\cite[Th.II.1.6]{Ch} and \cite[Prop.3.3]{BaChWa}, respectively, with
some slight replacements by Theorem~\ref{th:2.1}. (iii) corresponds
to Proposition 2.4 in \cite{Ba1} and can be proved similarly. (Note
that Theorem 4.6 in \cite[page. 43]{Ch} does not need the (PS)
condition in finite dimension space.) Since {\rm (F1)} implies that
$\mathcal{ L}:V\to\R$ is G\^{a}teaux differentiable, if $V=X$ and
$D\mathcal{ L}:X\to X^\ast$ is continuous from the norm topology of
$X$ to the weak*-topology of $X^\ast$ one may use a generalized
version of mountain pass lemma in \cite{GhPr} to yield a critical
point of mountain pass type provided that $\mathcal{ L}$ also
satisfies the condition (C) (weaker than (PS)).

If the critical point $\theta$ of $\mathcal{ L}$ is isolated,
Corollary~\ref{cor:2.5} yields  surjective homomorphisms from
critical groups $C_\ast(\mathcal{ L}|_{V^X}, \theta;{\bf K})$ to
$C_\ast(\mathcal{ L}, \theta;{\bf K})$, which are also isomorphisms
provided that ${\bf K}$ is a field and both groups are finite
dimension vector spaces over ${\bf K}$ of same dimension. When
$\mathcal{ L}\in C^2(V, \R)$ and $A\in C^1(V^X, X)$ it follows from
\cite[Cor.2.8]{JM} that $C_\ast(\mathcal{ L}|_{V^X}, \theta;{\bf
K})\cong C_\ast(\mathcal{ L}, \theta;{\bf K})$ for any Abel group
${\bf K}$. The following theorem generalizes and refines this
result.

\begin{theorem}\label{th:2.10}
Under the assumptions of Theorem~\ref{th:2.1}, let $\theta\in H$ be an isolated critical point of $\mathcal{L}$
and let $(Y, \|\cdot\|_Y)$ be another Banach space such that $X\subset Y\subset
H$ and that $(X, \|\cdot\|_X)$  is a densely embedded Banach space in $(Y, \|\cdot\|_Y)$
{\rm (}and hence $(Y, \|\cdot\|_Y)$ is a densely embedded Banach space in
 $(H, \|\cdot\|)$ due to \hbox{\rm (S)}. We may assume that
$\|y\|\le\|y\|_Y\;\forall y\in Y$ and $\|x\|_Y\le\|x\|_X\;\forall
x\in X${\rm )}. For an open neighborhood $V$ of the origin
$\theta\in H$, write $V^X=V\cap X$ (resp. $V^Y=V\cap Y$) as an open subset
of $X$ (resp. $Y$) as before. Assume also that
\begin{enumerate}
\item[{\bf (i)}] $\mathcal{L}|_{V^Y}\in C^2(V^Y, \mathbb{R})$.
\item[{\bf (ii)}] The map $A$ in $({\rm F2})$ belongs to $C^1(V^X, X)$.
\footnote{This and (i) imply $\mathcal{ L}|_{V^X}\in C^2(V^X, \mathbb{R})$.}
\item[{\bf (iii)}] The map $B$ in $({\rm F2})$ can be extended into a continuous  map $B: V^Y\to L_s(H)$ satisfying
$$
d^2(\mathcal{ L}|_{V^Y})(y)(u, v)=(B(y)u, v)_H\quad\forall y\in
V^Y\;\hbox{and}\; u, v\in Y.
$$
\end{enumerate}
 Then for any open neighborhood $W$ of $\theta$ in $V$ and a field ${\F}$
the inclusions
\begin{eqnarray*}
&&I^{xw}: \left(\mathcal{ L}_0\cap W^X, \mathcal{ L}_0\cap W^X\setminus\{\theta\}\right)\to
\left(\mathcal{ L}_0\cap W, \mathcal{
L}_0\cap W\setminus\{\theta\}\right),\\
&& I^{yw}: \left(\mathcal{ L}_0\cap W^Y, \mathcal{ L}_0\cap
W^Y\setminus\{\theta\}\right)\to \left(\mathcal{ L}_0\cap W,
\mathcal{ L}_0\cap W\setminus\{\theta\}\right)
\end{eqnarray*}
induce isomorphisms
\begin{eqnarray*}
&&\!\!\!\!I^{xw}_\ast: H_\ast\left(\mathcal{ L}_0\cap W^X,
\mathcal{ L}_0\cap W^X\setminus\{\theta\};{\F}\right)\to
H_\ast\left(\mathcal{ L}_0\cap W,
\mathcal{ L}_0\cap W\setminus\{\theta\};{\F}\right),\\
&&\!\!\!\!I^{yw}_\ast: H_\ast\left(\mathcal{ L}_0\cap W^Y,
\mathcal{ L}_0\cap W^Y\setminus\{\theta\};{\F}\right)\to
H_\ast\left(\mathcal{ L}_0\cap W, \mathcal{ L}_0\cap
W\setminus\{\theta\};{\F}\right).
\end{eqnarray*}
Consequently, $C_\ast(\mathcal{ L}|_{V^X},\theta; \F)\cong
C_\ast(\mathcal{ L}|_{V^Y},\theta; \F)\cong C_\ast(\mathcal{
L},\theta; \F)$.
\end{theorem}

The first isomorphism in the final claims is due to Jiang \cite{JM},
see Corollary~\ref{cor:4.4}. Taking $Y=X$ we get

\begin{corollary}\label{cor:2.11}
Under the assumptions of Theorem~\ref{th:2.1}, also assume: {\rm
(i)} $\theta$ is an isolated critical point of $\mathcal{ L}$, {\rm
(ii)} $\mathcal{ L}|_{V^X}\in C^2(V^X, \mathbb{R})$, {\rm
(iii)} the map $A$ in $({\rm F2})$ belongs to $C^1(V^X, X)$,
{\rm (iv)} the map $B$ in $({\rm F3})$ is continuous, Then for any
open neighborhood $W$ of $\theta$ in $V$ and a field ${\F}$ the
inclusion
\begin{eqnarray*}
I^{xw}: \left(\mathcal{ L}_0\cap W^X, \mathcal{ L}_0\cap W^X\setminus\{\theta\}\right)\to
\left(\mathcal{ L}_0\cap W, \mathcal{
L}_0\cap W\setminus\{\theta\}\right),
\end{eqnarray*}
induces isomorphisms between their relative homology groups with
coefficients in ${\F}$. Specially, $C_\ast(\mathcal{ L}|_{V^X},\theta;
\F)\cong C_\ast(\mathcal{ L},\theta;\F)$.
\end{corollary}

If $\Omega\subset\R^n$ is a bounded open domain with smooth boundary
$\partial\Omega$, and $f\in C^1(\overline{\Omega}\times\R,\R)$
satisfies the condition:
 $|f'_t(x,t)|\le
C(1+|t|^{\alpha})$ for some constants $C>0$ and
 $\alpha\le\frac{n+2}{n-2}$ (if $n>2$), then for an isolated critical point $u_0$ of
 the functional
 $$J(u)=\int_\Omega\Bigl(
 \frac{1}{2}|\nabla u|^2-F(x,u)\Bigr)dx
 $$
 (where $F$ is the primitive of $f$ with respect to $u$) on
 $H=H_0^1(\Omega)$ it follows from Corollary~\ref{cor:2.11} that
$C_\ast(J, u_0;\K)\cong C_\ast(J|_X, u_0;\K)$ provided that $u_0\in
X=C_0^1(X)$ is also an isolated critical point of $J|_X$. This
result was obtained by Chang \cite{Ch2} under the assumption that
$J$ satisfies the $({\rm PS})_c$ condition. Br\'ezis and Nirenberg
\cite{BrNi} firstly proved it as $u_0$ is a minimizer.

Theorem~\ref{th:2.1} and Corollary~\ref{cor:2.6} cannot be applied
to the geodesic problems on Finsler geometry directly.
But as outlined
in Remark~5.9 of \cite{Lu1} we may develop an method of infinite dimensional Morse theory
 for geodesics on Finsler manifolds  based on them in \cite{Lu4}, that is, giving the
shifting theorem of critical groups of the energy functional of a
Finsler manifold at a nonconstant critical
orbit and relations of critical groups under iterations. In particular,
Corollary~\ref{cor:2.5} is a key for us to realize the second goal.

Finally we give a theorem of Poincar\'e-Hopf type. By the condition
(F1) the functional $\mathcal{ L}:V\to\R$ is G\^{a}teaux
differentiable. Its gradient $\nabla\mathcal{ L}$ is equal to $A$ on
$V\cap X$ by the condition (F2). Furthermore, under the assumptions
(F3) and (D) we can prove that for a small $\epsilon>0$ the
restriction of $\nabla\mathcal{ L}$ to $B_H(\theta, 2\epsilon)$ has
a unique zero $\theta$ and is a demicontinuous map of class $(S)_+$.
According to \cite{Bro} and \cite{Skr1} we have a degree $\deg_{\rm
BS}(\nabla\mathcal{ L}, B_H(\theta,\epsilon), \theta)$. Under the
conditions {\rm (C1)} and {\rm (C2)}, $A'(\theta):X\to X$ is a
bounded linear Fredholm operator of index zero, see the first
paragraph in Step 1 of proof of Lemma~\ref{lem:3.1}. If the map $A$
in ({\rm F2}) is $C^1$, then $A$ is a Fredholm map of index zero
near $\theta\in X$ and thus for sufficiently small $\epsilon>0$
there exists a degree $\deg_{\rm FPR}(A, B_X(\theta,
\epsilon),\theta)$ or $\deg_{\rm BF}(A, B_X(\theta,
\epsilon),\theta)$ according to \cite{FiPeRa, PeRa} or \cite{BeFu1,
BeFu2}.

\begin{theorem}\label{th:2.12}
Under the assumptions of Theorem~\ref{th:2.1}, one has:
\begin{enumerate}
\item[\bf (i)]  If the map $A$ in the condition $({\rm F2})$ is $C^1$ near $\theta\in
X$, then for small $\epsilon>0$
\begin{eqnarray*}
\deg_{\rm FPR}(A, B_X(\theta, \epsilon),\theta)&=&\deg_{\rm BF}(A,
B_X(\theta, \epsilon),\theta)\\
&=&(-1)^{\mu}\deg(\nabla\mathcal{ L}^\circ, B_X(\theta,
\epsilon)\cap H^0,
\theta)\\
&=&\sum^\infty_{q=0}(-1)^q{\rm rank}C_q(\mathcal{ L}, \theta;{\bf
K})
\end{eqnarray*}
provided a suitable orientation for $A$.
\item[\bf (ii)] If  $\theta$ is also an isolated critical point  of $\mathcal{ L}$,
and the condition {\rm (D${\rm 4}^\ast$)} holds true,  then for a
small $\epsilon>0$,
\begin{eqnarray*}
\deg_{\rm BS}(\nabla\mathcal{ L}, B_H(\theta,
\epsilon),\theta)&=&\sum^\infty_{q=0}(-1)^q{\rm rank}C_q(\mathcal{
L},
\theta;{\bf K})\\
&=&(-1)^{\mu}\sum^\infty_{q=0}(-1)^q{\rm rank}C_q(\mathcal{
L}^\circ,
\theta;{\bf K})\\
&=&(-1)^{\mu}\deg(\nabla\mathcal{ L}^\circ, B_X(\theta,
\epsilon)\cap H^0, \theta).
\end{eqnarray*}
\end{enumerate}
Here $\deg$ is the classical Brouwer degree.
\end{theorem}

The first equality in (ii) of Theorem~\ref{th:2.12} is a direct
consequence of \cite[Th.1.2]{CiDe} once we prove that the map
$\nabla\mathcal{ L}$ is a demicontinuous map of class $(S)_+$ near
$\theta\in H$.

Using Theorem~\ref{th:2.1} we also gave a handle body theorem under
the our weaker framework  in Theorem 2.8 of \cite{Lu3}.

\section{Proof of Theorem~\ref{th:2.1}}\label{sec:3}

We shall complete the proof of Theorem~\ref{th:2.1} by a series of
lemmas.

\begin{lemma}\label{lem:3.1}
Under the above assumption $({\rm S})$, for an open neighborhood $V$
of $\theta\in H$ let $\mathcal{ L}|_{V\cap X}:V\cap X\to\R$ be
continuous   and continuously directional differentiable
\footnote{The former can be derived from the latter with mean value
theorem \cite[Prop.3.3.3]{Schi}.} (with respect to the induced
topology on $V\cap H$ from $H$). Let $B(\theta)\in L_s(H)$ satisfy
the conditions {\rm (C1)} and {\rm (C2)}. Suppose that a map $A:
V^X\to X$ is strictly F-differentiable at $\theta$ and satisfies
$A'(\theta)=B(\theta)|_X$ and
$$
D\mathcal{ L}(x)(u)=(A(x), u)_H\quad\forall x\in V\cap
X\;\hbox{and}\; u\in X.
$$
 Then there exist a positive $r_0\in\R$, a
unique map $h:B_{H^0}(\theta, r_0)\to X^\pm$ such that
\begin{enumerate}
\item[\bf (i)] $h(\theta)=\theta$ and  $(I-P^0)A(z+ h(z))=\theta$
for all $z\in B_{H^0}(\theta, r_0)$;
\item[\bf (ii)] $h$
is also Lipschitz continuous,  strictly F-differentiable at
$\theta\in H^0$ and $h'(\theta)z=\theta$ for any $z\in H^0$.
\end{enumerate}
 Moreover,  the
function $\mathcal{ L}^\circ(z)=\mathcal{ L}(z+ h(z))$
 is  $C^{2-0}$,
$$
d\mathcal{ L}^\circ(z_0)(z)=(A(z_0+ h(z_0)), z)_H\quad\forall z_0\in
B_{H^0}(\theta, r_0),\; z\in H^0,
 $$
and $d\mathcal{ L}^\circ$ is strictly F-differentiable at
$\theta\in H^0$ and $d^2\mathcal{ L}^\circ(\theta)=0$. {\rm
(}Clearly, if $\theta$ is an isolated critical point of $\mathcal{
L}|_{V^X}$ (thus an isolated zero of $A$) then $\theta$ is
also an isolated critical point of $\mathcal{ L}^\circ$.{\rm )}
\end{lemma}

\begin{proof}
 The proof method seems to be standard. For
completeness and the reader's conveniences we give its detailed
proof in two steps.

\noindent{\bf Step 1}.  Since $B(\theta)\in L_s(H)$ and
$A'(\theta)=B(\theta)|_X$ (so $B(\theta)(X)\subset X$), using
(C1)-(C2) it was proved in \cite{JM} that $B(\theta)(X^\pm)\subset
X^\pm$ and $B(\theta)|_{X^\pm}: X^\pm\to X^\pm$ is an isomorphism.
({\it Note}: It is where the assumption (C1) is used to prove that
the range $R(B(\theta))$ of $B(\theta)$ is closed in $H$ by
Proposition~\ref{prop:B.3}.)

Since $A$ is strictly F-differentiable at $\theta\in X$. It follows
that
\begin{equation}\label{e:3.1}
\|A(x_1)- B(\theta)x_1-A(x_2)+ B(\theta)x_2\|_X\le K_r\|x_1-x_2\|_X
\end{equation}
for all $x_1, x_2\in B_X(\theta, r)$ with constant $K_r\to 0$ as
$r\to 0$. (See  the proof of \cite[Cor.3]{IoSch}). In particular,
this implies that $A$ is continuous in $B_X(\theta, r)$. Let
\begin{equation}\label{e:3.2}
C_1=\|(B(\theta)|_{X^\pm})^{-1}\|_{L(X^\pm,
X^\pm)}\quad\hbox{and}\quad C_2=\|I-P^0\|_{L(X, X^\pm)}.
\end{equation}
Fix a small $r_1>0$ so that $C_1C_2K_{2r_1}<1/2$. Consider the map
\begin{equation}\label{e:3.3}
S: B_{H^0}(\theta, r_1)\times (B_X(\theta, r_1)\cap X^\pm)\to X^\pm
\end{equation}
given by $S(z,x)=-(B(\theta)|_{X^\pm})^{-1}(I-P^0)A(z+x)+ x$. Let
$z_1, z_2\in B_{H^0}(\theta, r_1)$ and $x_1, x_2\in B_X(\theta,
r_1)\cap X^\pm$. Noting  $B(\theta)x_i\in X^\pm$ and
$B(\theta)z_i=0$, $i=1,2$, we get
\begin{eqnarray}\label{e:3.4}
&&\|S(z_1, x_1)-S(z_2, x_2)\|_{X^\pm}\nonumber\\
&&\le C_1\cdot\|(I-P^0)A(z_1+x_1)- B(\theta)x_1
-(I-P^0)A(z_2+x_2)+ B(\theta)x_2\|_{X^\pm}\nonumber\\
&&= C_1\cdot\|(I-P^0)A(z_1+x_1)- (I-P^0)B(\theta)(z_1+ x_1)\nonumber\\
&&\hspace{20mm}-(I-P^0)A(z_2+x_2)
 +(I-P^0)B(\theta)(z_2+ x_2)\|_{X^\pm}\nonumber\\
&&\le C_1C_2\cdot \|A(z_1+x_1)- B(\theta)(z_1+ x_1)
-A(z_2+x_2)+ B(\theta)(z_2+ x_2)\|_{X}\nonumber\\
&&\le C_1C_2K_{2r_1}\cdot \|z_1+ x_1-z_2-
x_2\|_{X}\nonumber\\
&&<\frac{1}{2}\|z_1+ x_1-z_2-x_2\|_X\quad\hbox{if}\quad (z_1,x_1)\ne
(z_2,x_2).
\end{eqnarray}
Here the first two inequalities come from (\ref{e:3.2}), and the
third one is due to (\ref{e:3.1}). In particular, for any $z\in
B_{H^0}(\theta, r_1)$ and $x_1, x_2\in B_X(\theta, r_1)\cap X^\pm$,
it holds that
$$
\|S(z, x_1)-S(z,
x_2)\|_{X^\pm}<\frac{1}{2}\|x_1-x_2\|_X\quad\hbox{if}\;x_1\ne x_2.
$$
Moreover, since $A(x)\to \theta$ as $x\to\theta$ we can choose
$r_0\in (0, r_1)$ such that
\begin{eqnarray*}
\|S(z,\theta)\|_{X^\pm}&=&\|(B(\theta)|_{X^\pm})^{-1}(I-P^0)A(z)\|_{X^\pm}\\
&\le& C_1C_2\|A(z)\|_X<r_1(1-1/2)=\frac{r_1}{2}
\end{eqnarray*}
for any $z\in B_{H^0}(\theta, r_0)$. By Theorem 10.1.1 in
\cite[\S10.1]{Di} there exists a unique  map $h:B_{H^0}(\theta,
r_0)\to B_X(\theta, r_1)\cap X^\pm$ such that $S(z, h(z))=h(z)$ or
equivalently
\begin{equation}\label{e:3.5}
 (I-P^0)A(z+ h(z))=\theta\quad
\forall z\in B_{H^0}(\theta, r_0).
\end{equation}
Clearly, $h(\theta)=\theta$. From this and (\ref{e:3.4}) it follows
that
\begin{equation}\label{e:3.6}
\|h(z_1)-h(z_2)\|_{X}\le 2\|z_1-z_2\|_X\quad\forall z_1, z_2\in
B_{H^0}(\theta, r_0).
\end{equation}
That is, $h$ is Lipschitz  continuous.

For small $z_i\in B_{H^0}(\theta, r_0)$ set $x_i=h(z_i)$ in
(\ref{e:3.4}), $i=1,2$. We get
\begin{eqnarray}\label{e:3.7}
&&\quad\|h(z_1)-h(z_2)\|_{X^\pm}\nonumber\\
&&=\|S(z_1, h(z_1))-S(z_2, h(z_2))\|_{X^\pm}\nonumber\\
&&\le C_1C_2\cdot \|A(z_1 + h(z_1))- B(\theta)(z_1+ h(z_1))\nonumber\\
&&\hspace{20mm}-A(z_2 + h(z_2))+ B(\theta)(z_2+ h(z_2))\|_{X}.
\end{eqnarray}
By (\ref{e:3.1}), for any $\varepsilon>0$ there exists a number
$\delta>0$ such that
\begin{eqnarray}\label{e:3.8}
\|A(y_2)-A'(\theta)(y_2)- A(y_1) + A'(\theta)(y_1) \|_{X}
\le\varepsilon\|y_2-y_1\|_X
\end{eqnarray}
for $y_1, y_2\in B_X(\theta, \delta)$. Let us choose $\delta_0\in
(0, \delta)$ such that $z+ h(z)\in B_X(\theta, \delta)$ for any
$z\in B_{H^0}(\theta, \delta_0)$. From (\ref{e:3.7})- (\ref{e:3.8})
and (\ref{e:3.6}) it follows that
$$
\|h(z_2)- h(z_1)\|_{X^\pm}\le 3C_1C_2\varepsilon \|z_2-z_1\|_X
\quad\forall z_1, z_2\in B_{H^0}(\theta, \delta_0).
$$
Hence $h$ is strictly F-differentiable at $\theta\in H^0$ and
$h'(\theta)=0$.

 \noindent{\bf Step 2}. Let us prove the remainder ``Moreover'' part. Since $\mathcal{
L}|_{V\cap X}$ is continuous and continuously directional
differentiable (with respect to the induced topology on $V\cap H$ from $H$),
 for $z_0\in B_{H^0}(\theta, r_0)$, $z\in H^0$ and
$t\in\R\setminus\{0\}$ with $z_0+ tz\in B_{H^0}(\theta, r_0)$, by
the mean value theorem we have $s\in (0, 1)$ such that
\begin{eqnarray}\label{e:3.9}
&&\mathcal{ L}^\circ(z_0+ tz)-\mathcal{ L}^\circ(z_0)\nonumber\\
&=&D\mathcal{ L}(z_{s,t})(tz + h(z_0+ tz)-h(z_0))\nonumber\\
&=&(A(z_{s,t}), tz + h(z_0+ tz)-h(z_0))_H\nonumber\\
&=&(A(z_{s,t}), tz)_H + ((I-P^0)A(z_{s,t}), h(z_0+ tz)-h(z_0))_H
\end{eqnarray}
because $h(z_0+ tz)-h(z_0)\in X^\pm\subset H^\pm$, where
$z_{s,t}=z_0+ h(z_0)+ s[tz+ h(z_0+ tz)-h(z_0)]$. Note that
(\ref{e:3.6}) implies
$$
\|h(z_0+ tz)-h(z_0)\|_H\le\|h(z_0+ tz)-h(z_0)\|_X\le 2|t|r_0.
$$
Let $t\to 0$, we have
\begin{eqnarray*}
&&\left|\frac{((I-P^0)A(z_{s,t}), h(z_0+
tz)-h(z_0))_H}{t}\right|\\
&\le&\frac{\|(I-P^0)A(z_{s,t})\|_H\cdot\|h(z_0+
tz)-h(z_0)\|_H}{|t|}\\
&\le& 2r_0\|(I-P^0)A(z_{s,t})\|_{X^\pm}\\
&\to& 2r_0\|(I-P^0)A(z_0+ h(z_0))\|_{X^\pm}=0
\end{eqnarray*}
because of (\ref{e:3.5}) and the continuity of $A$ in
$B_X(\theta,r)$. From this and (\ref{e:3.9}) it follows that
$$
D\mathcal{ L}^\circ(z_0)(z)=\lim_{t\to 0}\frac{\mathcal{
L}^\circ(z_0+ tz)-\mathcal{ L}^\circ(z_0) }{t}= (A(z_0+ h(z_0)),
z)_H.
$$
Namely, $\mathcal{ L}^\circ$ is G\^{a}teaux differentiable at $z_0$.
Clearly, $z\mapsto D\mathcal{ L}^\circ(z_0)(z)$ is linear and
continuous, i.e. $\mathcal{ L}^\circ$ has a linear bounded
G\^{a}teaux derivative at $z_0$, $D\mathcal{ L}^\circ(z_0)$, given
by $D\mathcal{ L}^\circ(z_0)z=(A(z_0+ h(z_0)), z)_H=(P^0A(z_0+
h(z_0)), z)_H\;\forall z\in H^0$.

Note that  $B(\theta)|_{H^0}=0$, $B(\theta)(H^\pm)\subset H^\pm$ and
$h(z_0), h(z_0')\in X^\pm\subset H^\pm$ for any $z_0, z_0'\in
B_{H^0}(\theta, r_0)$. We have
\begin{eqnarray*}
(P^0B(\theta)(z_0+ h(z_0)), z)_H=(P^0B(\theta)(z'_0+ h(z'_0)),
z)_H=0\quad\forall z\in H^0.
\end{eqnarray*}
From this  it easily follows that
\begin{eqnarray}\label{e:3.10}
&&|D\mathcal{ L}^\circ(z_0)z- D\mathcal{ L}^\circ(z'_0)z|\nonumber\\
&=&\left|\bigl(P^0A(z_0+ h(z_0))-P^0A(z'_0+ h(z'_0)), z\bigr)_H\right|\nonumber\\
&=&\bigl|\bigl(P^0A(z_0+ h(z_0))- P^0B(\theta)(z_0+ h(z_0)), z\bigr)_H\nonumber\\
&&-\bigl(P^0A(z'_0+ h(z'_0))- P^0B(\theta)(z'_0+ h(z'_0)), z\bigr)_H\bigr|\nonumber\\
&\le &\|P^0A(z_0+ h(z_0))- P^0B(\theta)(z_0+ h(z_0))\nonumber\\
&&-P^0A(z'_0+h(z'_0))+ P^0B(\theta)(z'_0+ h(z'_0))\|_H\cdot\|z\|_H\nonumber\\
&\le &\|A(z_0+ h(z_0))- B(\theta)(z_0+ h(z_0))\nonumber\\
&&- A(z'_0+h(z'_0))+ B(\theta)(z'_0+ h(z'_0))\|_H\cdot\|z\|_H\nonumber\\
&\le &\|A(z_0+ h(z_0))- B(\theta)(z_0+ h(z_0))\nonumber\\
&&- A(z'_0+h(z'_0))+ B(\theta)(z'_0+ h(z'_0))\|_X\cdot\|z\|_H\\
&\le &K_{r_0+ r_1}\|z_0+ h(z_0)- z'_0- h(z'_0)\|_X\cdot\|z\|_H\nonumber\\
&\le &3K_{r_0+ r_1}\|z_0- z'_0\|_X\cdot\|z\|_H\nonumber
\end{eqnarray}
because of (\ref{e:3.1}) and (\ref{e:3.6}). Hence $z_0\mapsto
D\mathcal{ L}^\circ(z_0)$ is continuous and
\begin{eqnarray*}
\|D\mathcal{ L}^\circ(z_0)- D\mathcal{
L}^\circ(z'_0)\|_{L(H^0,\R)}\le3K_{r_0+ r_1}\|z_0- z'_0\|_X
\end{eqnarray*}
for every $z_0, z_0'\in B_X(\theta, r_0)\cap H^0$.
({\it Note}: Since $H$ and $X$ induce equivalent norms on $H^0$ and thus on
$L(H^0, \R)$, the alternative cannot lead to any troubles for the arguments.)
 By \cite[Th.2.1.13]{Ber}, this implies that $\mathcal{ L}^\circ$ is Fr\'echet differentiable at
$z_0$ and its Fr\'echet differential $d\mathcal{
L}^\circ(z_0)=D\mathcal{ L}^\circ(z_0)$ is Lipschitz continuous in
$z_0\in B_{H^0}(\theta, r_0)$.

Now for any $\varepsilon>0$ let $\delta>0$ such that (\ref{e:3.8})
holds. For $\delta_0\in (0, \delta)$ below (\ref{e:3.8}), by
(\ref{e:3.10}) and (\ref{e:3.6}) we obtain
\begin{eqnarray*}
&&|d\mathcal{ L}^\circ(z_0)z- d\mathcal{ L}^\circ(z'_0)z|\nonumber\\
&\le &\|A(z_0+ h(z_0))- B(\theta)(z_0+ h(z_0))\nonumber\\
&&- A(z'_0+h(z'_0))+ B(\theta)(z'_0+ h(z'_0))\|_X\cdot\|z\|_H\\
&\le &3\varepsilon\|z_0-z'_0\|_X\cdot\|z\|_H
\end{eqnarray*}
and hence $\|d\mathcal{ L}^\circ(z_0)- d\mathcal{
L}^\circ(z'_0)\|_{L(H^0,\R)}\le 3\varepsilon\|z_0-z'_0\|_X$ for any
$z_0, z'_0\in B_{H^0}(\theta, \delta_0)$. This shows that
$d\mathcal{ L}^\circ$ is strictly F-differentiable at $\theta\in
H^0$ and $d^2\mathcal{ L}^\circ(\theta)=0$.  Lemma~\ref{lem:3.1} is
proved.
 \end{proof}

 Since $\|\cdot\|$ and $\|\cdot\|_X$ are
equivalent norms on $H^0$ we may choose $\delta>0$ so small that
$\bar B_H(\theta, \delta)\cap H^0\subset B_X(\theta, r_0)\cap H^0$
and that
\begin{equation}\label{e:3.11}
z+ h(z)+ u\in V\quad\forall (z,u)\in (\bar B_H(\theta, \delta)\cap
H^0)\times (\bar B_H(\theta, \delta)\cap H^\pm).
\end{equation}

\begin{remark}\label{lem:3.2}
{\rm If $A\in C^1(V^X, X)$, we can directly apply the implicit
function theorem \cite[Th.3.7.2]{Schi} to $C^1$-map
$$
T:(H^0\cap V)\times (X^\pm\cap V)\to X^\pm,\; (z,x)\mapsto
(I-P^0)A(z+ x),
$$
and get that  the maps $h$ and $\mathcal{ L}^\circ$ are $C^1$ and
$C^2$, respectively. Precisely,
$$
h'(z)=-\bigl[(I-P^0)A'(z+ h(z))|_{X^\pm}\bigr]^{-1}(I-P^0)A'(z+
h(z))|_{H^0}.
$$
\hfill$\Box$\vspace{2mm}}
\end{remark}

 Define a continuous map $F:\bar B_{H^0}(\theta, \delta)\times B_{H^\pm}(\theta, \delta)\to\R$ as
\begin{equation}\label{e:3.12}
 F(z, u)=\mathcal{ L}(z+ h(z)+ u)-\mathcal{ L}(z+ h(z)).
\end{equation}
 Then for each
$z\in \bar B_{H^0}(\theta, \delta)$ the map $F(z,\cdot)$ is
continuously directional differentiable in $B_{H^\pm}(\theta,
\delta)$, and the directional derivative of it at $u\in
B_{H^\pm}(\theta, \delta)$ in any direction $v\in H^\pm$ is given by
\begin{eqnarray}\label{e:3.13}
D_2F(z,u)(v)&=&(\nabla{\mathcal L}(z+ h(z)+ u), v)_H\nonumber\\
&=&((I-P^0)\nabla{\mathcal L}(z+ h(z)+ u), v)_H.
\end{eqnarray}
 It follows from this and (\ref{e:3.5}) that
\begin{eqnarray}\label{e:3.14}
F(z, \theta)=0\quad\hbox{and}\quad D_2F(z,
\theta)(v)=0\;\forall v\in H^\pm.
\end{eqnarray}

Now we wish to apply Theorem~\ref{th:A.1} to the function $F$. In
order to check that $F$ satisfies the conditions in
Theorem~\ref{th:A.1} we need two lemmas.

\begin{lemma}\label{lem:3.3}
 There exists a function $\omega:V\cap X\to [0, \infty)$  such that $\omega(x)\to 0$ as $x\in V\cap X$ and $\|x\|\to
0$, and that
$$
|(B(x)u, v)_H- (B(\theta)u, v)_H |\le \omega(x) \|u\|\cdot\|v\|
$$
for any $x\in V\cap X$,  $u\in H^0\oplus H^-$ and $v\in H$.
\end{lemma}

\begin{proof}
 Note that  the condition (D2) can be equivalently expressed as:
{\it For any $u\in H$ it holds that $\|P(x)u-P(\theta)u\|\to 0$ as
$x\in V\cap X$ and $\|x\|\to 0$}. Let $e_1,\cdots, e_m$ be a basis
of $H^0\oplus H^-$ with $\|e_i\|=1$, $i=1,\cdots, m$. Then
\begin{eqnarray*}
&&|(B(x)e_i, v)_H- (B(\theta)e_i, v)_H | \\
&\le &|(P(x)e_i-P(\theta)e_i, v)_H|+ |([Q(x)-Q(\theta)]e_i, v)_H|\\
&\le &\|P(x)e_i- P(\theta)e_i\|\cdot\|v\|+
\|Q(x)-Q(\theta)\|\cdot\|v\|.
\end{eqnarray*}
Since $H^0\oplus H^-$ is of finite dimension, there exists a
constant $C_4>0$ such that
$$
\Bigl(\sum^m_{i=1}|t_i|^2\Bigr)^{1/2}\le C_4\|u\|\quad\forall
u=\sum^m_{i=1}t_ie_i\in H^0\oplus H^-.
$$
Hence for any $u=\sum^m_{i=1}t_ie_i\in H^0\oplus H^-$ we have
\begin{eqnarray*}
&&|(B(x)u, v)_H- (B(\theta)u, v)_H | \\
&\le &\sum^m_{i=1}|t_i|\|P(x)e_i- P(\theta)e_i\|\cdot\|v\|+
\sum^m_{i=1}|t_i|\|Q(x)-Q(\theta)\|\cdot\|v\|\\
&\le &\left(\sum^m_{i=1}\|P(x)e_i-
P(\theta)e_i\|^2\right)^{1/2}\left(\sum^m_{i=1}|t_i|^2\right)^{1/2}\|v\|\\
&&\quad +
\sqrt{m}\left(\sum^m_{i=1}|t_i|^2\right)^{1/2}\|Q(x)-Q(\theta)\|\cdot\|v\|\\
&\le &\hspace{-3mm}
\left[C_4
\left(\sum^m_{i=1}\|P(x)e_i-P(\theta)e_i\|^2\right)^{1/2}
+ C_4\sqrt{m}\|Q(x)-Q(\theta)\| \right]\|u\|\|v\|\\
&=&\omega(x)\|u\|\|v\|,
\end{eqnarray*}
 where
$$
\omega(x)=\hspace{-2mm}\left[C_4 \left(\sum^m_{i=1}\|P(x)e_i-
P(\theta)e_i\|^2\right)^{1/2} + C_4\sqrt{m}\|Q(x)-Q(\theta)\|
\right]\to 0
$$
as $x\in V\cap X$ and $\|x\|\to 0$ (because of the conditions (D2)
and (D3)). \end{proof}

When $H^0=\{\theta\}$ under the stronger assumptions the following
lemma was proved in \cite{Skr, Va1}. We also give proof of it for
clearness.

\begin{lemma}\label{lem:3.4}
There exists a  small neighborhood $U\subset V$ of $\theta$ in $H$
and a number $a_1\in (0, 2a_0]$ such that for any $x\in U\cap X$,
\begin{enumerate}
\item[{\rm (i)}] $(B(x)u, u)_H\ge a_1\|u\|^2\;\forall u\in H^+$;
\item[{\rm (ii)}] $|(B(x)u,v)_H|\le\omega(x)\|u\|\cdot\|v\|\;\forall u\in H^+, \forall v\in
H^-\oplus H^0$;
\item[{\rm (iii)}] $(B(x)u,u)_H\le-a_0\|u\|^2\;\forall u\in H^-$.
\end{enumerate}
\end{lemma}

\begin{proof}
 {\bf (i)} By (\ref{e:2.2}), we have
\begin{equation}\label{e:3.15}
(B(\theta)u, u)_H\ge 2a_0\|u\|^2\;\forall u\in H^+.
\end{equation}
Assume by  contradiction  that (i) does not hold. Then there exist
sequences $\{x_n\}\subset V\cap X$ with $\|x_n\|\to 0$, and
$\{u_n\}\in H^+$ with $\|u_n\|=1\;\forall n$, such that
$$(B(x_n)u_n,
u_n)_H<1/n\;\forall n=1,2,\cdots.
$$
Passing a subsequence, we may
assume that
\begin{equation}\label{e:3.16}
(B(x_n)u_n, u_n)_H\to\beta\le 0\;\hbox{as}\;n\to\infty,
\end{equation}
and that $u_n\rightharpoonup u_0$ in $H$. We claim: $u_0\ne
\theta$. In fact, by the condition (D4) we have constants $C_0>0$
and $n_0\in\N$ such that $(P(x_n)u, u)\ge C_0\|u\|^2$ for any $u\in
H$ and $n\ge n_0$. Hence
\begin{eqnarray}\label{e:3.17}
(B(x_n)u_n, u_n)_H&=&(P(x_n)u_n, u_n)_H + (Q(x_n)u_n, u_n)_H\nonumber\\
&\ge & C_0+ (Q(x_n)u_n, u_n)_H\quad\forall n>n_0.
\end{eqnarray}
 Moreover, a direct computation gives
\begin{eqnarray}\label{e:3.18}
&&\!\!\!\!\!\quad |(Q(x_n)u_n, u_n)_H-(Q(\theta)u_0, u_0)_H|\\
&&\!\!\!\!\!=|((Q(x_n)-Q(\theta))u_n, u_n)_H+ (Q(\theta)u_n, u_n)_H-(Q(\theta)u_0, u_n)_H\nonumber\\
&&\hspace{70mm}+ (Q(\theta)u_0, u_n-u_0)_H|\nonumber\\
&&\!\!\!\!\!\le \|Q(x_n)-Q(\theta)\|\cdot\|u_n\|^2+
\|Q(\theta)u_n-Q(\theta)u_0\|\cdot\|u_n\|\nonumber\\
&&\hspace{40mm}+
|(Q(\theta)u_0, u_n-u_0)_H|\nonumber\\
&&\!\!\!\!\!\le \|Q(x_n)-Q(\theta)\|+ \|Q(\theta)u_n-Q(\theta)u_0\|+
|(Q(\theta)u_0, u_n-u_0)_H|.\nonumber
\end{eqnarray}
Since $u_n\rightharpoonup u_0$ in $H$,
$\lim_{n\to\infty}|(Q(\theta)u_0, u_n-u_0)_H|=0$. We have also
\begin{equation}\label{e:3.19}
\lim_{n\to\infty}\|Q(\theta)u_n-Q(\theta)u_0\|=0
\end{equation}
by the  compactness  of $Q(\theta)$, and
\begin{equation}\label{e:3.20}
\lim_{n\to\infty}\|Q(x_n)-Q(\theta)\|=0
\end{equation}
 by the condition (D3).
Hence  (\ref{e:3.18})-(\ref{e:3.20}) give
\begin{equation}\label{e:3.21}
\lim_{n\to\infty}(Q(x_n)u_n, u_n)_H=(Q(\theta)u_0, u_0)_H.
\end{equation}
Then this and (\ref{e:3.16})-(\ref{e:3.17}) yield
$$
0\ge \beta=\lim_{n\to\infty}(B(x_n)u_n, u_n)_H\ge C_0 +
(Q(\theta)u_0,u_0)_H.
$$
This implies $u_0\ne\theta$. Note that $u_0$ also sits in $H^+$.

As above, using (\ref{e:3.20}) we derive
\begin{eqnarray}\label{e:3.22}
&&|(Q(x_n)u_0, u_n)_H- (Q(\theta)u_0, u_0)_H|\\
&\le& |(Q(x_n)u_0, u_n)_H- (Q(\theta)u_0, u_n)_H|+ |(Q(\theta)u_0,
u_n)_H- (Q(\theta)u_0,
u_0)_H|\nonumber\\
&\le&  \|Q(x_n)-Q(\theta)\|\cdot\|u_0\|+ |(Q(\theta)u_0, u_n-u_0)_H
|\to 0.\nonumber
\end{eqnarray}
 Note that
\begin{eqnarray*}
&&(B(x_n)(u_n-u_0), u_n-u_0)_H\\
&=&(P(x_n)(u_n-u_0), u_n-u_0)_H + (Q(x_n)(u_n-u_0), u_n-u_0)_H\\
&\ge& C_0\|u_n-u_0\|^2+ (Q(x_n)(u_n-u_0), u_n-u_0)_H\\
&\ge& (Q(x_n)u_n, u_n)_H-2(Q(x_n)u_0, u_n)_H+ (Q(\theta)u_0, u_0)_H.
\end{eqnarray*}
It follows from this and (\ref{e:3.21})-(\ref{e:3.22}) that
\begin{eqnarray}\label{e:3.23}
&&\liminf_{n\to\infty}(B(x_n)(u_n-u_0),
u_n-u_0)_H\nonumber\\
&&\ge\lim_{n\to\infty}(Q(x_n)(u_n-u_0), u_n-u_0)_H = 0.
\end{eqnarray}
Note that $u_n\rightharpoonup u_0$ implies that $(P(\theta)u_0,
u_n-u_0)_H\to 0$. We get
\begin{eqnarray*}
&&|(B(x_n)u_0, u_n)_H-(B(\theta)u_0, u_0)_H|\\
&=&|(P(x_n)u_0, u_n)_H+ (Q(x_n)u_0, u_n)_H- (P(\theta)u_0, u_0)_H-
(Q(\theta)u_0, u_0)_H|\\
&\le & |(P(x_n)u_0, u_n)_H-(P(\theta)u_0, u_0)_H|+ |(Q(x_n)u_0,
u_n)_H-(Q(\theta)u_0, u_0)_H| \\
&\le & |(P(x_n)u_0, u_n)_H-(P(\theta)u_0, u_n)_H|+|(P(\theta)u_0,
u_n)_H-(P(\theta)u_0, u_0)_H| \\
&&\quad + |(Q(x_n)u_0, u_n)_H-(Q(\theta)u_0, u_0)_H|
\\
&\le & \|P(x_n)u_0- P(\theta)u_0\| + |(P(\theta)u_0, u_n-u_0)_H| \\
&&\quad + |(Q(x_n)u_0, u_n)_H-(Q(\theta)u_0, u_0)_H| \to 0
\end{eqnarray*}
because of the condition (D2) and (\ref{e:3.22}). Similarly, we have
$$
\lim_{n\to\infty}(B(x_n)u_0, u_0)_H=(B(\theta)u_0, u_0)_H.
$$
From these, (\ref{e:3.16}) and (\ref{e:3.23}) it follows that
\begin{eqnarray*}
0&\le& \liminf_{n\to\infty}(B(x_n)(u_n-u_0), u_n-u_0)_H\\
&=& \liminf_{n\to\infty}[(B(x_n)u_n,u_n)_H-2(B(x_n)u_0, u_n)_H+
(B(x_n)u_0,
u_0)_H]\\
&=&\lim_{n\to\infty}(B(x_n)u_n,u_n)_H- (B(\theta)u_0, u_0)_H\\
&=&\beta- (B(\theta)u_0, u_0)_H.
\end{eqnarray*}
Namely, $(B(\theta)u_0,u_0)_H\le\beta\le 0$. It contradicts to
(\ref{e:3.15}) because $u_0\in H^+\setminus\{0\}$.

\noindent{\bf (ii)} By (\ref{e:2.1}), $(B(\theta)u,v)_H=0$ for $u\in
H^+$ and $v\in H^0\oplus H^-$. The conclusion follows from
Lemma~\ref{lem:3.3} immediately.

\noindent{\bf (iii)} By the choice of $a_0$ we have $(B(\theta)v,
v)_H\le -2a_0\|v\|^2\;\forall v\in H^-$. By Lemma~\ref{lem:3.3}, for
any $x\in U\cap X$ and $v\in H^-$ we have
\begin{eqnarray*}
(B(x)v, v)_H&=&(B(\theta)v, v)_H+ (B(x)v, v)_H-(B(\theta)v,v)_H\\
&\le& (B(\theta)v,v)_H+ \omega(x)\|v\|^2\\
&\le& -2a_0\|v\|^2+\omega(x)\|v\|^2.
\end{eqnarray*}
By shrinking $U$ (if necessary) we can require that $\omega(x)<a_0$
for any $x\in U\cap X$. Then  the desired conclusion is proved.
\end{proof}

Since $h(\theta)=\theta$, for the neighborhood $U$ in
Lemma~\ref{lem:3.4} we may take $\varepsilon\in (0, \delta)$ so
small that
 \begin{equation}\label{e:3.24}
 z+ h(z)+
u^++ u^-\in U
\end{equation}
for all $z\in\bar B_{H^0}(\theta,\varepsilon)$, $u^+\in \bar
B_{H^+}(\theta,\varepsilon)$ and $u^-\in\bar
B_{H^-}(\theta,\varepsilon)$.

\begin{lemma}\label{lem:3.5}
For the above $\varepsilon>0$ the restriction of the function $F$ in
(\ref{e:3.12}) to $\bar B_{H^0}(\theta,\varepsilon)\times \bigl(\bar
B_{H^+}(\theta,\varepsilon)\oplus \bar
B_{H^-}(\theta,\varepsilon)\bigr)$ satisfies the conditions in
Theorem~\ref{th:A.1}.
\end{lemma}

\begin{proof}
 By (\ref{e:3.14}) we only need to prove that
$F$ satisfies conditions (ii)-(iv) in Theorem~\ref{th:A.1}.

\noindent{\bf Step 1}. For $z\in\bar B_{H^0}(\theta,\varepsilon)$, $u^+\in
\bar B_H(\theta,\varepsilon)\cap X^+$ and $u^-_1, u^-_2\in\bar
B_{H^-}(\theta,\varepsilon)$, by the condition (F2) we have
\begin{eqnarray*}
&&[D_2F(z, u^+ + u^-_2)-D_2F(z, u^++ u^-_1)](u^-_2-u^-_1)\\
&=&(A(z+ h(z)+ u^++u^-_2), u^-_2-u^-_1)_H\\
& &- (A(z+ h(z)+ u^++u^-_1), u^-_2-u^-_1)_H.
\end{eqnarray*}
Moreover, $A$ is continuously directional differentiable so is the
function
$$
u\mapsto (A(z+ h(z)+ u^++u), u^-_2-u^-_1)_H.
$$
By the mean value theorem we have $t\in (0, 1)$ such that
\begin{eqnarray*}
&&(A(z+ h(z)+ u^++u^-_2), u^-_2-u^-_1)_H \\
&&- (A(z+ h(z)+ u^++u^-_1), u^-_2-u^-_1)_H\\
&=&\left(DA(z+ h(z)+ u^++ u^-_1+ t(u^-_2-u^-_1))(u^-_2-u^-_1),
u^-_2-u^-_1\right)_H\\
&\stackrel{(F3)}{=}&\left(B(z+ h(z)+ u^++ u^-_1+
t(u^-_2-u^-_1))(u^-_2-u^-_1),
u^-_2-u^-_1\right)_H\\
&\le& -a_0\|u^-_2-u^-_1\|^2
\end{eqnarray*}
by Lemma~\ref{lem:3.4}(iii). Hence
\begin{eqnarray*}
[D_2F(z, u^+ + u^-_2)-D_2F(z, u^++ u^-_1)](u^-_2-u^-_1)\le
-a_0\|u^-_2-u^-_1\|^2.
\end{eqnarray*}
 Since $\bar
B_H(\theta,\varepsilon)\cap X^+$ is dense in $\bar
B_H(\theta,\varepsilon)\cap H^+$ we get
\begin{eqnarray}\label{e:3.25}
 [D_2F(z, u^+ + u^-_2)-D_2F(z, u^++ u^-_1)](u^-_2-u^-_1)\le
-a_0\|u^-_2-u^-_1\|^2.
\end{eqnarray}
for all $z\in\bar B_{H^0}(\theta,\varepsilon)$, $u^+\in \bar
B_H(\theta,\varepsilon)\cap H^+$ and $u^-_i\in\bar
B_{H}(\theta,\varepsilon)\cap H^-$, $i=1, 2$. This implies the
condition (ii).

\noindent{\bf Step 2}.  Let $z\in\bar B_{H^0}(\theta,\varepsilon)$, $u^+\in
\bar B_H(\theta,\varepsilon)\cap X^+$ and $u^-\in\bar
B_{H^-}(\theta,\varepsilon)$. Then by (\ref{e:3.14}), the mean value
theorem and  (F2)-(F3), for some $t\in (0, 1)$ we have
\begin{eqnarray*}
&&D_2F(z, u^++u^-)(u^+-u^-)\\
&=&D_2F(z, u^++u^-)(u^+-u^-)- D_2F(z, \theta)(u^+-u^-)\\
&=&(A(z+ h(z)+ u^++u^-), u^+-u^-)_H-(A(z+ h(z)+ \theta), u^+-u^-)_H\\
&=&\left(B(z+ h(z)+ t(u^++u^-))(u^++u^-), u^+-u^-\right)_H\\
&=&\left(B(z+ h(z)+ t(u^++u^-))u^+, u^+\right)_H\\
&-&\left(B(z+ h(z)+
t(u^++u^-))u^-, u^-\right)_H\\
&\ge & a_1\|u^+\|^2+ a_0\|u^-\|^2
\end{eqnarray*}
by Lemma~\ref{lem:3.4}(i) and (iii).  As above this inequality also
holds for all $u^+\in \bar B_{H^+}(\theta,\varepsilon)$ because
$\bar B_H(\theta,\varepsilon)\cap X^+$ is dense in $\bar
B_H(\theta,\varepsilon)\cap H^+$. Hence $D_2F(z,
u^++u^-)(u^+-u^-)>0$ for $(u^+, u^-)\ne (\theta, \theta)$. The
condition (iii) is proved.

\noindent{\bf Step 3}. For $z\in\bar B_{H^0}(\theta,\varepsilon)$ and $u^+\in
\bar B_H(\theta,\varepsilon)\cap X^+$, as above we have $t\in (0,
1)$ such that
\begin{eqnarray*}
D_2F(z, u^+)u^+
&=&D_2F(z, u^+)u^+- D_2F(z, \theta)u^+\\
&=&(A(z+ h(z)+ u^+), u^+)_H-(A(z+ h(z)+ \theta), u^+)_H\\
&=&\left(B(z+ h(z)+ tu^+)u^+, u^+\right)_H\\
&\ge& a_1\|u^+\|^2
\end{eqnarray*}
because of Lemma~\ref{lem:3.4}(i). It follows that
$$
D_2F(z, u^+)u^+ \ge a_1\|u^+\|^2> p(\|u^+\|)\quad\forall u^+\in \bar
B_H(\theta,\varepsilon)\cap H^+\setminus\{\theta\},
$$
where $p:(0, \varepsilon]\to (0, \infty)$ is a non-decreasing
function given by $p(t)=\frac{a_1}{2}t^2$. This proves the condition
(iv).
 \end{proof}

 By Lemma~\ref{lem:3.5} we can apply Theorem~\ref{th:A.1} to $F$  to
 get a
positive number $\epsilon$, an open neighborhood
 $\mathcal{ W}$ of $\bar B_{H^0}(\theta,\varepsilon)\times\{\theta\}$ in $\bar B_{H^0}(\theta,\varepsilon)\times H^\pm$,
  and  an origin-preserving homeomorphism
\begin{eqnarray}\label{e:3.26}
\phi: \bar B_{H^0}(\theta,\varepsilon)\times \left(B_{H^+}(\theta,
\epsilon)+ B_{H^-}(\theta, \epsilon)\right)\to \mathcal{ W}
\end{eqnarray}
of form
\begin{eqnarray*}
\phi(z, u^+ + u^-)=(z, \phi_z(u^+ + u^-))\in \bar
B_{H^0}(\theta,\varepsilon)\times H^\pm
\end{eqnarray*}
such that $\phi_z(\theta)=\theta$ and
\begin{eqnarray}\label{e:3.27}
 &&\mathcal{ L}(z+ h(z)+ \phi_z(u^+, u^-))-\mathcal{ L}(z+ h(z))\nonumber\\
 &=&F(\phi(z, u^+, u^-))=\|u^+\|^2-\|u^-\|^2
\end{eqnarray}
for all $(z, u^+, u^-)\in \bar B_{H^0}(\theta,\varepsilon)\times
B_{H^+}(\theta, \epsilon)\times B_{H^-}(\theta, \epsilon)$.
Moreover, $\phi_z(u^++u^-)\in H^-$ if and only if $u^+=\theta$, and
$\phi$ is also a homeomorphism from $\bar
B_{H^0}(\theta,\varepsilon)\times  B_{H^-}(\theta, \epsilon)$ onto
$\mathcal{ W}\cap (\bar B_{H^0}(\theta,\varepsilon)\times H^-)$ even
if the last two sets are equipped with the induced topology from
$X$, or, equivalently, for $(z_0, u^-_0)\in \bar
B_{H^0}(\theta,\varepsilon)\times B_{H^-}(\theta, \epsilon)$ and
$\{(z_k, u^-_k)\}\subset\bar B_{H^0}(\theta,\varepsilon)\times
B_{H^-}(\theta, \epsilon)$ it holds that
\begin{equation}\label{e:3.28}
\|z_k+ u^-_k-z_0-u^-_0\|_X\to 0\Longleftrightarrow
\left\{\begin{array}{ll}
 \|z_k-z_0\|_X\to 0\quad\hbox{and}\\
\|\phi_{z_k}(u^-_k)-\phi_{z_0}(u^-_0)\|_X\to 0.
\end{array}\right.
\end{equation}
Consider the continuous map
\begin{eqnarray}\label{e:3.29}
&&\Phi:B_{H^0}(\theta,\varepsilon)\times \left(B_{H^+}(\theta,
\epsilon)+ B_{H^-}(\theta, \epsilon)\right)\to H,\\
&&\hspace{10mm} (z, u^+ + u^-)\mapsto z+ h(z)+ \phi_z(u^+ +
u^-).\nonumber
\end{eqnarray}
Then (\ref{e:3.27}) gives (\ref{e:2.5}), i.e. $\mathcal{ L}(\Phi(z,
u^+, u^-))=\|u^+\|^2-\|u^-\|^2+ \mathcal{ L}(z+ h(z))$. Since  $H^0$
and $H^-$ are finitely dimensional subspaces contained in $X$, from
Steps 1,4 in the proof of Theorem~\ref{th:A.1} it is easily seen
that
$$
\phi_z \left(B_{H^+}(\theta, \epsilon)\cap X + B_{H^-}(\theta,
\epsilon)\right)\subset X\quad\forall z\in
B_{H^0}(\theta,\varepsilon).
$$
Then (\ref{e:2.6}) follows from this and the fact that ${\rm
Im}(h)\subset X^\pm\subset X$. In particular, it holds that
$\Phi(B_{H^0}(\theta,\varepsilon)\times B_{H^-}(\theta,
\epsilon))\subset X$.  Now we can complete the proof of
Theorem~\ref{th:2.1} by the following lemma.

\begin{lemma}\label{lem:3.6}
Let $W={\rm Im}(\Phi)$. Then it is an open neighborhood of $\theta$
in $H$ and $\Phi$  is an origin-preserving homeomorphism onto $W$.
Moreover, if the topologies on $B_{H^0}(\theta,\varepsilon) \times
B_{H^-}(\theta, \epsilon)\subset X$ and
$\Phi(B_{H^0}(\theta,\varepsilon)\times  B_{H^-}(\theta,
\epsilon))\subset X$ are chosen as ones induced by $X$,  the
restriction of $\Phi$ to $B_{H^0}(\theta,\varepsilon)\times
B_{H^-}(\theta, \epsilon)$ is a homeomorphism from
$B_{H^0}(\theta,\varepsilon)\times B_{H^-}(\theta, \epsilon)\subset
X$ onto $\Phi(B_{H^0}(\theta,\varepsilon)\times B_{H^-}(\theta,
\epsilon))\subset X$.
\end{lemma}

\begin{proof}
 Assume that $\Phi(z_1, u^+_1+ u^-_1)=\Phi(z_2, u^+_2+ u^-_2)$ for
$(z_1, u^+_1+ u^-_1)$ and $(z_2, u^+_2+ u^-_2)$ in
$B_{H^0}(\theta,\varepsilon)\times (B_{H^+}(\theta, \epsilon)+
B_{H^-}(\theta, \epsilon))$. Then
$$
z_1=z_2\quad\hbox{and}\quad h(z_1)+ \phi_{z_1}(u^+_1+ u^-_1)=h(z_2)+
\phi_{z_2}(u^+_2+ u^-_2).
$$
It follows that $h(z_1)=h(z_2)$ and $\phi_{z_1}(u^+_1+ u^-_1)=
\phi_{z_2}(u^+_2+ u^-_2)$. This shows that $\phi(z_1, u^+_1+ u^-_1)=
\phi(z_2, u^+_2+ u^-_2)$ and thus $(u^+_1, u^-_1)=(u^+_2, u^-_2)$.
So $\Phi$ is a bijection.

Let $(z, u^++ u^-)$ and a sequence $\{(z_k, u^+_k+ u^-_k)\}$ sit in
$B_{H^0}(\theta,\varepsilon)\times (B_{H^+}(\theta, \epsilon)+
B_{H^-}(\theta, \epsilon))$. Suppose that $\Phi(z_k, u^+_k+
u^-_k)\to\Phi(z, u^++ u^-)$. Then
\begin{eqnarray*}
&&P^0\Phi(z_k, u^+_k+ u^-_k)\to P^0\Phi(z, u^++
u^-)\quad\hbox{and}\\
&& (P^++P^-)\Phi(z_k, u^+_k+ u^-_k)\to(P^++P^-)\Phi(z, u^++ u^-).
\end{eqnarray*}
It follows that $z_k\to z$, and thus $h(z_k)\to h(z)$ and
$\phi_{z_k}(u^+_k+ u^-_k)\to \phi_z(u^++ u^-)$. This shows that
$\phi(z_k, u^+_k+ u^-_k)\to\phi(z, u^++ u^-)$ and hence $(z_k,
u^+_k+ u^-_k)\to (z, u^++ u^-)$ since $\phi$ is a homeomorphism.
That is, $\Phi^{-1}$ is continuous. The first claim is proved.

To prove the second claim, it suffices to prove that for $(z_0,
u^-_0)\in \bar B_{H^0}(\theta,\varepsilon)\times B_{H^-}(\theta,
\epsilon)$  and $\{(z_k, u^-_k)\}\subset \bar
B_{H^0}(\theta,\varepsilon)\times B_{H^-}(\theta, \epsilon)$
\begin{equation}\label{e:3.30}
\left.\begin{array}{ll}
 \|z_k+ u^-_k-z_0-u^-_0\|_X\to 0\quad\hbox{if and only if}\\
 \|z_k+ h(z_k)+\phi_{z_k}(u^-_k)-z_0-h(z_0)-\phi_{z_0}(u^-_0)\|_X\to
0.
\end{array}\right\}
\end{equation}
Note that $h\in C(B_{H^0}(\theta, \delta), X^\pm)$  and that $X$ and
$H$ induce equivalent topologies on $H^0+H^-$. Since $\|z_k+
u^-_k-z_0-u^-_0\|_X\to 0$ if and only if $\|z_k-z_0\|_X\to 0$ and
$\|u^-_k-u^-_0\|_X\to 0$, it follows from (\ref{e:3.28}) that in
(\ref{e:3.30}) the left side implies the right side. Conversely, if
the right of (\ref{e:3.30}) holds, then
\begin{eqnarray*}
\|z_k-z_0\|&=&\|P^0(z_k+
h(z_k)+\phi_{z_k}(u^-_k))-P^0(z_0+h(z_0)+\phi_{z_0}(u^-_0))\|\\
&\le &\|z_k+
h(z_k)+\phi_{z_k}(u^-_k)-z_0-h(z_0)-\phi_{z_0}(u^-_0)\|\\
&\le&\|z_k+ h(z_k)+\phi_{z_k}(u^-_k)-
z_0-h(z_0)-\phi_{z_0}(u^-_0)\|_X \to 0,
\end{eqnarray*}
and hence $\|z_k-z_0\|_X\to 0$. It follows that
$\|h(z_k)-h(z_0)\|_X\to 0$ and therefore
\begin{eqnarray*}
&&\|\phi_{z_k}(u^-_k)-\phi_{z_0}(u^-_0)\|_X\\
&\le& \|z_k-z_0\|_X+ \|h(z_k)-h(z_0)\|_X\\
&& + \;\|z_k+
h(z_k)+\phi_{z_k}(u^-_k)-z_0-h(z_0)-\phi_{z_0}(u^-_0)\|_X\to 0.
\end{eqnarray*}
From these and (\ref{e:3.28}) we derive that
$\|z_k+u^-_k-z_0-u^-_0\|_X\to 0$.  (\ref{e:3.30})   is proved.
\end{proof}

In summary we have completed the proof of Theorem~\ref{th:2.1}.

\section{Proofs of Corollaries~\ref{cor:2.5},
~\ref{cor:2.7} and Theorem~\ref{th:2.10}}\label{sec:4}

\subsection{Proof of Corollaries~\ref{cor:2.5} and ~\ref{cor:2.7}}

\begin{proof}[Proof of Corollary~\ref{cor:2.5}]
 By the excision property of relative homology
groups we only need to prove the corollary for some open
neighborhood $W$ of $\theta$ in $H$. Let $W$ be as in
Theorem~\ref{th:2.1}, that is,
$$
W=\Phi\left(B_{H^0}(\theta,\epsilon)\times
 (B_{H^+}(\theta,\epsilon)+ B_{H^-}(\theta,\epsilon))\right).
 $$
 Set
$W_{0-}:=\Phi\left(B_{H^0}(\theta,\epsilon)\times
 B_{H^-}(\theta,\epsilon)\right)$. It is contained in $X$ by (\ref{e:2.6}).
We write $W_{0-}$ as $W^X_{0-}$ when it is considered a topological
subspace of $X$. Clearly, $\mathcal{ L}_0\cap W_{0-}=(\mathcal{
L}|_{V\cap X})_0\cap W^X_{0-}= \mathcal{ L}_0\cap W^X_{0-}$ as sets.
Define a deformation
 $\eta: W\times [0, 1]\to W$ as
 $$
\eta(\Phi(z, u^++ u^-), t)=\Phi(z, tu^++ u^-).
 $$
It gives a deformation retract from $\mathcal{ L}_0\cap W$ onto
$\mathcal{ L}_0\cap W_{0-}$. Hence  the inclusion
$$
I:\left(\mathcal{ L}_0\cap W_{0-}, \mathcal{ L}_0\cap
W_{0-}\setminus\{\theta\}\right)\hookrightarrow\left(\mathcal{
L}_0\cap W, \mathcal{ L}_0\cap W\setminus\{\theta\}\right)
$$
induces isomorphisms between their relative singular homology groups with
inverse $(\eta_1)_\ast$, where $\eta_1(\cdot)=\eta(1, \cdot)$. That
means that  each
$$\alpha\in H_q\left(\mathcal{L}_0\cap W,
\mathcal{ L}_0\cap W\setminus\{\theta\}; {\bf K}\right)
$$
has
 a relative singular cycle representative,
$c=\sum_jg_j\sigma_j$, such that
$$
|c|:=\cup_j\sigma_j(\triangle^q)\subset \mathcal{ L}_0\cap
W_{0-}\quad\hbox{and}\quad |\partial c|\subset \mathcal{ L}_0\cap
W_{0-}\setminus\{\theta\}.
$$
 By the conclusion (b) in Theorem~\ref{th:2.1}  the identity map
 $$
 \imath^{0-}:\left(\mathcal{ L}_0\cap W^X_{0-}, \mathcal{ L}_0\cap W^X_{0-}\setminus\{\theta\}\right)
 \to \left(\mathcal{ L}_0\cap
W_{0-}, \mathcal{ L}_0\cap W_{0-}\setminus\{\theta\}\right)
$$
is a homeomorphism. So $c$ is also a relative singular cycle in
$$\left(\mathcal{ L}_0\cap W^X_{0-}, \mathcal{ L}_0\cap
W^X_{0-}\setminus\{\theta\}\right),$$
  denoted by $c^x$. Then
$\imath^{0-}\circ c^x=c$. Write $W^X=W\cap X$ as a topological
subspace of $X$. Denote by   the inclusion
$$
\jmath:\left(\mathcal{ L}_0\cap W^X, \mathcal{ L}_0\cap
W^X\setminus\{\theta\}\right)\hookrightarrow \left(\mathcal{
L}_0\cap W, \mathcal{ L}_0\cap W\setminus\{\theta\}\right),
$$
and by the inclusion
$$
I^X:\left(\mathcal{ L}_0\cap W^X_{0-}, \mathcal{ L}_0\cap
W^X_{0-}\setminus\{\theta\}\right)\hookrightarrow\left(\mathcal{
L}_0\cap W^X, \mathcal{ L}_0\cap W^X\setminus\{\theta\}\right).
$$
Since $I_\ast([c])=\alpha$, $(\imath^{0-})_\ast([c^x])=[c]$ and $I\circ\imath^{0-}=\jmath\circ I^X$
we obtain
$$
\alpha=I_\ast\circ(\imath^{0-})_\ast[c^x]=\jmath_\ast\circ
(I^X)_\ast[c^x]=\jmath_\ast\bigl((I^X)_\ast[c^x]\bigr).
$$
This completes the proof of Corollary~\ref{cor:2.5}.
 \end{proof}

\begin{proof}[Proof of Corollary~\ref{cor:2.7}]
 As in
the proof of \cite[Prop.3.2]{BaChWa} we only need to prove the
implication (iii)$\Longrightarrow$(i). If $\nu=\dim H^0=0$, by (i)
of Remark~\ref{rm:2.2} and (\ref{e:2.7}) we have $C_q(\mathcal{
L},\theta;\K)=\delta_{q\mu}\;\forall q\in\Z$, where $\mu=\dim H^-$.
Hence $\mu=0$. Then (\ref{e:2.7}) shows that $\theta$ is a strict
minimum. If $\nu>0$, by Corollary~\ref{cor:2.6} it must hold that
$\mu=\dim H^-=0$ and $C_0(\mathcal{ L}^\circ,\theta;\K)\ne 0$. Since
$\mathcal{ L}^\circ$ is $C^{2-0}$ and $\dim H^0<\infty$ we can
construct a $C^{2-0}$ function $g$ on $H^0$ satisfying (PS) such
that it coincides with $\mathcal{ L}^\circ$ near $\theta\in H^0$. By
Theorem 4.6 on the page 43 of \cite{Ch}, $\theta$ is a minimum of
$\mathcal{L}^\circ$. It follows from (\ref{e:2.5}) that $\theta$ is
a strict minimum of $\mathcal{ L}$.
\end{proof}

\subsection{Proofs of Theorem~\ref{th:2.10}}\label{sec:2.4}

Recall that $H^0={\rm
Ker}(B(\theta))$ and
 $X^\pm=X\cap H^\pm=(I_H-P^0)(X)$. Set $Y^\pm=Y\cap
 H^\pm=(I_H-P^0)(Y)$. We need the following theorem by Ming Jiang.

\begin{theorem}[\hbox{\cite[Th.2.5]{JM}}]\label{th:4.1}
Under the assumptions of Theorem~\ref{th:2.10}, (but it suffices to
assume the density of $X$ in $Y$), there exists a ball $B_Y(\theta,
\kappa)$, an origin-preserving local homeomorphism $\Psi$ defined on
$B_Y(\theta, \kappa)$ and a $C^1$ map $\rho:B_{Y}(\theta,
\kappa)\cap H^0\to X^\pm$ such that
$$
\mathcal{ L}\circ\Psi(y)=\frac{1}{2}(B(\theta)y^\pm, y^\pm)_H +
\mathcal{ L}(z+\rho(z))\quad\forall y\in B_{Y}(\theta, \kappa),
$$
where $z=P^0(y)$ and $y^\pm=(I-P^0)(y)$. Moreover,
$\Psi(B_{Y}(\theta, \kappa)\cap X)\subset X$ and $\Psi:B_Y(\theta,
\kappa)\cap X\to \Psi(B_Y(\theta, \kappa)\cap X)$ is also an
origin-preserving local homeomorphism even if both $B_Y(\theta,
\kappa)\cap X$ and $\Psi(B_Y(\theta, \kappa)\cap X)$ are equipped
with the induced topology by $X$.
\end{theorem}

\begin{remark}\label{rm:4.2}
{\rm (i) From the arguments of Lemma~\ref{lem:3.1} and the proof of
\cite{JM} it is easily seen that near $\theta\in N$ the map $\rho$
is equal to $h$ in Lemma~\ref{lem:3.1}.\\
(ii) It was proved in \cite[Prop.2.1]{JM} that the condition (iii)
in Theorem~\ref{th:2.10} can be derived from others of this proposition and the
following two conditions:
\begin{enumerate}
\item[\bf (FN3a)] $\forall x\in V\cap X$, $\exists\; C(x)>0$ such that
$$
|d^2(\mathcal{ L}|_{V^X})(x)(\xi, \eta)|\le
C(x)\|\xi\|\cdot\|\eta\|\;\forall \xi, \eta\in X.
$$
\item[\bf (FN3b)] $\forall \varepsilon>0$, $\exists\; \delta >0$ such that
for any $x_1, x_2\in V\cap X$ with $\|x_1- x_2\|_Y<\delta$,
$$
|d^2(\mathcal{ L}|_{V^X})(x_1)(\xi, \eta)- d^2(\mathcal{
L}|_{V^X})(x_2)(\xi, \eta) |\le \varepsilon
\|\xi\|\cdot\|\eta\|\;\forall \xi, \eta\in X.
$$
\end{enumerate}}
\end{remark}

If $H^-\subset Y$, then $P^+Y\subset Y$ because $H^0\subset X\subset
Y$. In this case, for $y\in Y$  we can write
$y^\bot=(I-P^0)y=y^++y^-=P^+y+ P^-y$ and hence
$$
(B(\theta)y^\bot, y^\bot)_H=(P^+B(\theta)P^+y^+, y^+)_H+
(P^-B(\theta)P^-y^-, y^-)_H
$$
Define a functional
$$
\mathcal{ L}^{\diamond}:B_{H^0}(\theta, \kappa)\cap
H^0\to\R,\;z\mapsto \mathcal{ L}^{\diamond}(z)=\mathcal{ L}(z+
\rho(z)).
 $$
  Then $\theta\in H^0$ is its critical
point, and also isolated if  $\theta$ is an isolated critical point
of $\mathcal{ L}|_{V^X}$.  By Remark~\ref{lem:3.2}, $\rho$ is $C^1$,
and Lemma~\ref{lem:3.1} and Remark~\ref{rm:4.2}(i) show that near
$\theta\in H^0$,
\begin{eqnarray*}
d\mathcal{ L}^\diamond(z)(\xi)=(A(z+ \rho(z)), \xi)_H=(A(z+ h(z)),
\xi)_H\quad\forall \xi\in H^0.
\end{eqnarray*}

 If $\theta$ is an isolated critical point of $\mathcal{
L}|_{V^X}$ (and hence $\mathcal{L}|_{V^Y}$), then by
Theorem~\ref{th:4.1} we can use the same proof method as
 in \cite[Th.8.4]{MaWi} or \cite[Th.5.1.17]{Ch1} to
derive:

\begin{corollary}[Shifting]\label{cor:4.3}
Under the assumptions of Theorem~\ref{th:4.1}, if $\theta$ is an
isolated critical point of $\mathcal{ L}|_{V^Y}$, $H^-\subset Y$ and
$\dim H^0\oplus H^-<\infty$, then
$$
C_q(\mathcal{ L}|_{V^Y}, \theta;{\bf K})\cong
C_{q-\mu}(\mathcal{ L}^{\diamond}, \theta;{\bf K})\quad\forall
q\in\N\cup\{0\}
$$
for any Abel group ${\bf K}$,  where  $\mu:=\dim H^-$.
\end{corollary}

\begin{corollary}[\hbox{\cite[Cor.2.8]{JM}}]\label{cor:4.4}
Under the assumptions of Theorem~\ref{th:4.1}, if $\theta$ is  an
isolated critical point of $\mathcal{ L}|_{V^Y}$, and $H^-\subset
X$, then for any Abel group ${\bf K}$,
$$
C_q(\mathcal{ L}|_{V^X}, \theta;{\bf K})\cong C_{q}(\mathcal{
L}|_{V^Y}, \theta;{\bf K})\quad\forall q=0, 1,\cdots.
$$
\end{corollary}

Actually, from the proof of \cite[Cor.2.8]{JM} one can  get the
following stronger conclusion:

\begin{proposition}\label{prop:4.5}
For any open neighborhood $U^Y$
of $\theta$ in $V^Y$ and the corresponding one of $\theta$ in
$V^X$, $U^X=U^Y\cap X$, the inclusion
$$
\iota: \left(\mathcal{ L}_0\cap U^X, \mathcal{ L}_0\cap
U^X\setminus\{\theta\}\right)\to \left(\mathcal{ L}_0\cap U^Y,
\mathcal{ L}_0\cap U^Y\setminus\{\theta\}\right)
$$
induces  isomorphisms
$$
\iota_\ast: H_\ast\left(\mathcal{ L}_0\cap U^X, \mathcal{ L}_0\cap
U^X\setminus\{\theta\};{\bf K}\right)\to H_\ast\left(\mathcal{
L}_0\cap U^Y, \mathcal{ L}_0\cap U^Y\setminus\{\theta\};{\bf
K}\right)
$$
for any Abel group ${\bf K}$, where $\mathcal{ L}_0=\{x\in V\,|\,
\mathcal{ L}(x)\le 0\}$.
\end{proposition}

\begin{proof}
 By the excision property of the singular
homology theory we only need to prove it for some open neighborhood
$U^Y$ of $\theta$ in $V^Y$. By \cite[Claim 1]{JM})
$$
\|y\|_D=\|(P^0+P^-)y\|_Y+ \|P^+y\|_Y
$$
gives a norm on $Y$ equivalent to $\|\cdot\|_Y$. Let $\kappa_0\in
(0, \kappa)$ be so small that
\begin{equation}\label{e:4.1}
B^Y_{\kappa_0}:=\{y\in Y\,|\, \|y\|_D<\kappa_0\}\subset B_Y(\theta,
\delta)
\end{equation}
and that $U^Y=\Psi(B^Y_{\kappa_0})$ (resp. $\Psi(B^Y_{\kappa_0}\cap
X)$) is a neighborhood of $\theta$ in $Y$ (resp. $X$) which only
contains $\theta$ as a unique critical point of $\mathcal{
L}|_{V^Y}$ (resp. $\mathcal{ L}|_{V^X}$). (This can be assured by
the second claim in Theorem~\ref{th:4.1}). For conveniences let
$$
\mathcal{ Y}=\mathcal{ L}_0\cap U^Y\quad\hbox{and}\quad \mathcal{
X}=\mathcal{ Y}\cap X=\mathcal{ L}_0\cap U^X=\{y\in U^Y\cap X\,|\,
\mathcal{ L}(y)\le 0\},
$$
and let $\iota:(\mathcal{ X}, \mathcal{
X}\setminus\{\theta\})\hookrightarrow (\mathcal{ Y}, \mathcal{
Y}\setminus\{\theta\})$ be the inclusion.
 By Theorem~\ref{th:4.1} we have
$$
\Psi^{-1}(\mathcal{ Y})=\left\{y\in B^Y_{\kappa_0}\,\Bigm|\,
\frac{1}{2}(B(\theta)y^\bot, y^\bot)+ \mathcal{ L}(z+ \rho(z))\le
0\right\}
$$
and isomorphisms
\begin{eqnarray*}
&&(\Psi^{-1}|_\mathcal{ Y})_\ast: H_\ast(\mathcal{ Y}, \mathcal{
Y}\setminus\{\theta\};{\bf K})\cong H_\ast(\Psi^{-1}(\mathcal{ Y}),
\Psi^{-1}(\mathcal{
Y})\setminus\{\theta\};{\bf K}),\\
&&(\Psi^{-1}|_\mathcal{ X})_\ast: H_\ast(\mathcal{ X}, \mathcal{
X}\setminus\{\theta\};{\bf K})\cong H_\ast(\Psi^{-1}(\mathcal{
Y})\cap X, \Psi^{-1}(\mathcal{ Y})\cap X\setminus\{\theta\};{\bf
K}).
\end{eqnarray*}
Define $\Psi^{-1}(\mathcal{ Y})_{0-}=\Psi^{-1}(\mathcal{ Y})\cap
(H^0+ H^-)$. Then $\Psi^{-1}(\mathcal{ Y})_{0-}\subset X$ and thus
\begin{equation}\label{e:4.2}
\Psi^{-1}(\mathcal{ Y})_{0-}=\Psi^{-1}(\mathcal{ Y})_{0-}\cap X.
\end{equation}
For $B^Y_{\kappa_0}$ in (\ref{e:4.1}) let $\Re:[0, 1]\times
B^Y_{\kappa_0}\to Y$ be the continuous map defined by
$$
\Re(t,y)=(P^0+ P^-)y+ (1-t)P^+y.
$$
Clearly, $\Re(0,\cdot)=id$, $\Re(t,\cdot)|_{\Psi^{-1}(\mathcal{
Y})_{0-} }=id$ and $\Re(1, \Psi^{-1}(\mathcal{ Y}))\subset
\Psi^{-1}(\mathcal{ Y})_{0-}$. It was  proved in \cite{JM}  that
$\Re$ is also a continuous map from $[0,1]\times(B^Y_{\kappa_0}\cap
X)$ to $X$ (with respect to the induced topology from $X$) and that
\begin{enumerate}
\item[(I)] $\Re(1,\Psi^{-1}(\mathcal{
Y})\setminus\{\theta\})\subset \Psi^{-1}(\mathcal{
Y})_{0-}\setminus\{\theta\}$,
\item[(II)] $\Re(t, \Psi^{-1}(\mathcal{
Y})\setminus\{\theta\})\subset\Psi^{-1}(\mathcal{
Y})\setminus\{\theta\}$ for $t\in [0, 1]$.
\end{enumerate}
 These show that $\Re$ gives not only a deformation retract from $ (\Psi^{-1}(\mathcal{
Y}), \Psi^{-1}(\mathcal{ Y})\setminus\{\theta\})$ to
$(\Psi^{-1}(\mathcal{ Y})_{0-}, \Psi^{-1}(\mathcal{
Y})_{0-}\setminus\{\theta\})$, but also one from
$(\Psi^{-1}(\mathcal{ Y})\cap X, \Psi^{-1}(\mathcal{ Y})\cap
X\setminus\{\theta\})$ to
$$
(\Psi^{-1}(\mathcal{ Y})_{0-}\cap X, \Psi^{-1}(\mathcal{
Y})_{0-}\cap X\setminus\{\theta\})=(\Psi^{-1}(\mathcal{ Y})_{0-},
\Psi^{-1}(\mathcal{ Y})_{0-}\setminus\{\theta\})
$$
(with respect to the induced topology from $X$). Hence inclusions
{
\begin{eqnarray*}
&&\hspace{-5mm}i^y:(\Psi^{-1}(\mathcal{ Y})_{0-},
\Psi^{-1}(\mathcal{ Y})_{0-}\setminus\{\theta\})\hookrightarrow
(\Psi^{-1}(\mathcal{ Y}),
\Psi^{-1}(\mathcal{ Y})\setminus\{\theta\})\quad\hbox{and}\\
&&\hspace{-5mm}i^x:(\Psi^{-1}(\mathcal{ Y})_{0-}\cap X,
\Psi^{-1}(\mathcal{ Y})_{0-}\cap
X\setminus\{\theta\})\hookrightarrow (\Psi^{-1}(\mathcal{ Y})\cap X,
\Psi^{-1}(\mathcal{ Y})\cap X\setminus\{\theta\})
\end{eqnarray*}}
induce isomorphisms
 {\footnotesize
\begin{eqnarray*}
&&\hspace{-7mm}H_\ast(\Psi^{-1}(\mathcal{ Y})_{0-},
\Psi^{-1}(\mathcal{ Y})_{0-}\setminus\{\theta\};{\bf
K})\xrightarrow{i^y_\ast} H_\ast(\Psi^{-1}(\mathcal{ Y}),
\Psi^{-1}(\mathcal{ Y})\setminus\{\theta\};{\bf K})\quad\hbox{\normalsize and}\\
&&\hspace{-7mm}H_\ast(\Psi^{-1}(\mathcal{ Y})_{0-}\cap X,
\Psi^{-1}(\mathcal{ Y})_{0-}\cap X\setminus\{\theta\};{\bf
K})\xrightarrow{i^x_\ast} H_\ast(\Psi^{-1}(\mathcal{ Y})\cap X,
\Psi^{-1}(\mathcal{ Y})\cap X\setminus\{\theta\};{\bf K}).
\end{eqnarray*}}
Consider the inclusions
{
\begin{eqnarray*}
&&i^{xy}:(\Psi^{-1}(\mathcal{ Y})\cap X, \Psi^{-1}(\mathcal{ Y})\cap
X\setminus\{\theta\})\hookrightarrow (\Psi^{-1}(\mathcal{ Y}),
\Psi^{-1}(\mathcal{ Y})\setminus\{\theta\})\quad\hbox{and}\\
&&i^{xy}_0:(\Psi^{-1}(\mathcal{ Y})_{0-}\cap X, \Psi^{-1}(\mathcal{
Y})_{0-}\cap X\setminus\{\theta\})\hookrightarrow
(\Psi^{-1}(\mathcal{ Y})_{0-}, \Psi^{-1}(\mathcal{
Y})_{0-}\setminus\{\theta\}).
\end{eqnarray*}}
It is obvious that $i^{xy}\circ i^x=i^y\circ i^{xy}_0$. Since $H^0+
H^-\subset X$, both $(H^0+ H^-, \|\cdot\|_X)$ and $(H^0+ H^-,
\|\cdot\|_Y)$ are complete. Hence the norms $\|\cdot\|_X$ and
$\|\cdot\|_Y$ are equivalent on $H^0+ H^-$. It follows from this and
(\ref{e:4.2}) that $i^{xy}_0$ is a homeomorphism. This shows that
$(i^{xy}_0)_\ast$ and hence $i^{xy}_\ast$ is an isomorphism. Note
that $(\Psi^{-1}|_\mathcal{
Y})\circ\iota=i^{xy}\circ(\Psi^{-1}|_\mathcal{ X})$.
Proposition~\ref{prop:4.5}
 follows immediately.
 \end{proof}

Before proving Theorem~\ref{th:2.10} we also need the following
observation,  which is contained in the proof of \cite[Th.3.2, page
100]{Ch} and seems to be obvious. But the author cannot find
where it is explicitly pointed out.

\begin{remark}\label{rm:4.6}
{\rm Let $H$ be a real Hilbert space, and let $f\in C^2(H,\R)$
satisfy the (PS) condition. Assume that $df(x)=x-Tx$, where $T$ is a
compact mapping, and that $p_0$ is an isolated critical point of
$f$. Then for any field $\F$ and each $q\in\N\cup\{0\}$, $C_q(f, p_0;\F)$ is
a finite dimension vector space over $\F$. In particular, if $f\in
C^2(\R^n, \R)$ has an isolated critical point $p_0\in\R^n$ then
$C_q(f, p_0;\F)$, $q=0,1,\cdots$, are vector spaces over $\F$ of
finite dimensions. In fact, by \cite[(3.2), page 101]{Ch} we have
$$
C_\ast(f, p_0;\F)=H_\ast(W, W_-;\F)=H_\ast\left(\tilde
f_{\frac{2}{3}\gamma}\cap W, \tilde f_{-\frac{2}{3}\gamma}\cap
W;\F\right),
$$
where $(W, W_-)$ is a Gromoll-Meyer pair of $f$ at $p_0$, and
$\tilde f$ has only nondegenerate critical points $\{p_j\}^m_1$ in
$W$, finite in number, contained in $B_H(p_0, \delta)\subset {\rm
Int}(W)\cap f^{-1}[-\gamma/3, \gamma/3]$. Hence $C_\ast(f,
p_0;\F)=\oplus^m_{j=1}C_\ast(\tilde f, p_j;\F)$. The claim follows
because each $C_q(\tilde f, p_j;\F)$ is either $\F$ or $0$.}
\end{remark}

\begin{proof}[Proof of Theorem~\ref{th:2.10}]
 By assumptions
$(X, H, \mathcal{ L})$ and  $(X, Y, H, \mathcal{ L})$ satisfy the
conditions in Corollaries~\ref{cor:2.6},~\ref{cor:4.3} and
\ref{cor:4.4} respectively. By Remark~\ref{rm:4.2}  near $\theta\in
H^0$ the maps $h$ and $\rho$ are same. Then
Corollaries~\ref{cor:2.6},~\ref{cor:4.3} and \ref{cor:4.4} lead to
\begin{equation}\label{e:4.3}
C_\ast(\mathcal{ L},\theta;{\bf K})\cong C_\ast(\mathcal{ L}|_{V^Y},\theta;{\bf K})\cong
C_\ast(\mathcal{ L}|_{V^X}, \theta;{\bf
K})
\end{equation}
for any Abel group ${\bf K}$.

 Note that we may assume that  $W$ is given by
Theorem~\ref{th:2.1} because of the excision property of the
singular homology groups. By Proposition~\ref{prop:4.5} the
inclusion
$$
I^{xy}:\left(\mathcal{ L}_0\cap W^X,
 \mathcal{ L}_0\cap W^X\setminus\{\theta\}\right)\hookrightarrow \left(\mathcal{ L}_0\cap
W^Y, \mathcal{ L}_0\cap W^Y\setminus\{\theta\}\right)
$$
induces an isomorphism
$$
I^{xy}_\ast: H_\ast\left(\mathcal{ L}_0\cap W^X, \mathcal{ L}_0\cap
W^X\setminus\{\theta\}; {\bf K}\right)\to H_\ast\left(\mathcal{
L}_0\cap W^Y, \mathcal{ L}_0\cap W^Y\setminus\{\theta\};{\bf
K}\right).
$$
By (\ref{e:4.3}) and Remark~\ref{rm:4.6},  for a field $\F$ and each
$q\in\N\cup\{0\}$,
\begin{eqnarray*}
&&C_q(\mathcal{ L}|_{V^X}, \theta;{\F})\cong H_q\left(\mathcal{
L}_0\cap
W^X, \mathcal{ L}_0\cap W^X\setminus\{\theta\};{\F}\right),\\
&&C_q(\mathcal{ L}|_{V^Y},\theta;{\F})\cong H_q\left(\mathcal{
L}_0\cap
W^Y, \mathcal{ L}_0\cap W^Y\setminus\{\theta\};{\F}\right),\\
&& C_q(\mathcal{ L},\theta;{\F})\cong H_q\left(\mathcal{ L}_0\cap W,
\mathcal{ L}_0\cap W\setminus\{\theta\};{\F}\right)
\end{eqnarray*}
are isomorphic vector spaces over ${\F}$ of finite dimension. Then
any surjective (or injective) homomorphism among them must be an
isomorphism. By Corollary~\ref{cor:2.5}
$I^{xw}_\ast$ is a surjection and hence an isomorphism. Since
$I^{xw}_\ast=I^{yw}_\ast\circ
I^{xy}_\ast$, $I^{yw}_\ast$ is also an isomorphism.
 \end{proof}

\section{Proof of Theorem~\ref{th:2.12}}\label{sec:5}

We use the ideas of \cite{Ho} to prove (i) in Step 1, and then
derive (ii) in Step 2 from \cite[Th.1.2]{CiDe} by checking that
$\nabla\mathcal{ L}$ is a demicontinuous  map of class $(S)_+$.

\noindent{\bf Step 1}. By the first paragraph in Step 1 of the proof
of Lemma~\ref{lem:3.1},
 $(I-P^0)B(\theta)|_{X^\pm}: X^\pm\to X^\pm$ is a Banach isomorphism.
 Consider the $C^1$ map
$\Theta: [2,3]\times (V\cap X^\pm)\to X^\pm$ given by
\begin{equation}\label{e:5.1}
(t,u)\mapsto (3-t)(I-P^0)A(u)+(t-2)(I-P^0)B(\theta)u.
\end{equation}
Then $D_2\Theta(t,\theta)=(I-P^0)B(\theta)|_{X^\pm}$ for all $t\in
[2, 3]$. By the inverse function theorem there exist positive
constants $\rho\in (0, r_0]$ and $C_7>0, C_8>0$ such that
\begin{equation}\label{e:5.2}
C_7\|u\|_X\le \|\Theta(t,u)\|_X\le C_8\|u\|_X\quad\forall u\in
B_X(\theta, \rho)\cap X^\pm,\;t\in [2, 3].
\end{equation}
 Following the
notations in Lemma~\ref{lem:3.1}, we can shrink $\rho>0$ (if
necessary) such that the following (i)-(iii) are satisfied:
\begin{enumerate}
\item[\bf (i)] $\theta$ is a unique zero of $A$ in $B_X(\theta,
2\rho)$,

\item[\bf (ii)] $z+ h(z)\in B_X(\theta, r_0/2)$ for any $z\in
B_X(\theta,2\rho)$,

\item[\bf (iii)] $\|z\|_X<r_0$ and $\|u\|_X<r_0$ for any $z+u\in
B_X(\theta, 2\rho)$ with $z\in H^0$ and $u\in X^\pm$. (This can be
realized because $H^0$ is a space of finite dimension.)
\end{enumerate}

Now we define a map $ \Gamma:[0,3]\times B_X(\theta, \rho)\to
X,\;(t, z+u)\mapsto\Gamma_t(z+u)$, where
$$
\Gamma_t(z+ u)=\left\{\begin{array}{ll}
 (I-P^0)A(z+u)+ P^0A\bigl(th(z)+ (1-t)u+ z)\quad\hbox{if}\;t\in
 [0,1],\\
 (I-P^0)A\bigl(u+ (2-t)z\bigr)+ P^0A(z+ h(z))\quad\hbox{if}\;t\in
 [1,2],\\
 (3-t)(I-P^0)A(u)+ (t-2)(I-P^0)A'(\theta)u+ P^0A(z+ h(z))\\
 \hspace{40mm}\hbox{if}\;t\in
 [2,3].
 \end{array}\right.
 $$
Clearly, $\Gamma$ is $C^0$, and every $\Gamma_t$ is $C^1$ and
satisfies $\Gamma_t(\theta)=\theta$. Let us prove:

\noindent{\bf Claim 5.1}. $\exists\; \epsilon\in (0, \rho)$ such that
$\Gamma_t(x)\ne\theta\;\forall (t,x)\in [0,3]\times(\bar B_X(\theta,
\epsilon)\setminus\{\theta\})$.

In fact, assume that $\Gamma_t(z+u)=\theta$ for some $t\in [0, 1]$
and $z+u\in \bar B_X(\theta, \rho)$. Then $(I-P^0)A(z+u)+
P^0A\bigl(th(z)+ (1-t)u+ z)=\theta$ and hence
$$
(I-P^0)A(z+u)=\theta\quad\hbox{and}\quad P^0A\bigl(th(z)+
(1-t)u+ z)=\theta.
$$
By the first equality, (\ref{e:3.5}) and the uniqueness we have
$u=h(z)$. So the second equality becomes
$$
\theta=P^0A\bigl(th(z)+ (1-t)u+ z)=P^0A\bigl(th(z)+ (1-t)h(z)+
z)=P^0A\bigl(z+ h(z)).
$$
This and (\ref{e:3.5}) give $A(z+h(z))=\theta$. By (i) we get
$z+h(z)=\theta$. That is, $z=\theta$ and $z+u=\theta$.

Similarly, let $\Gamma_t(z+u)=\theta$ for some $t\in [1, 2]$ and
$z+u\in \bar B_X(\theta, \rho)$. Then
$$
(I-P^0)A\bigl(u+ (2-t)z\bigr)=\theta\quad\hbox{and}\quad P^0A(z+
h(z))=\theta.
$$
(\ref{e:3.5}) and the second equality yield $A(z+h(z))=\theta$, and
hence $z=\theta$ as above. Since $\|u\|_X<r_0<r_1$, it follows from
the first equality and the construction of $h$ above (\ref{e:3.5})
that $u=h((2-t)z)=\theta$.

Finally, assume that $\Gamma_t(z+u)=\theta$ for some $t\in [2, 3]$
and $z+u\in B_X(\theta, \epsilon)$, where $\epsilon\in (0, \rho)$ is
such that $\|u\|_X<\rho$ for any $z+u\in B_X(\theta, \epsilon)$
(with $z\in H^0$ and $u\in X^\pm$).
 Then $P^0A(z+ h(z))=\theta$ and
$$
\Theta(t,u)=(3-t)(I-P^0)A(u)+ (t-2)(I-P^0)B(\theta)u=\theta.
$$
 The former implies $z=\theta$ as above, and  (\ref{e:5.2}) leads to $u=\theta$.
 \underline{Claim~5.1 is proved.}

By Lemma~\ref{lem:3.1}(i), $h'(\theta)=\theta$. Using this it is
easily proved that $d\Gamma_t(\theta)=A'(\theta)$ for any $t\in [0,
3]$. Since the $C^1$ Fredholm map is locally proper, we can shrink
$\epsilon>0$ such that the restriction of each $\Gamma_t$ to $\bar
B_X(\theta,\epsilon)$ is Fredholm and that the restriction of
$\Gamma$ to $[0, 3]\times\bar B_X(\theta,\epsilon)$ is proper. Hence
 $\Gamma:[0,3]\times B_X(\theta,\epsilon)\to X$ satisfies
the homotopy definition in the Benevieri-Furi degree theory
\cite{BeFu1, BeFu2}, and we arrive at
\begin{equation}\label{e:5.3}
\deg_{\rm BF}(A, B_X(\theta, \epsilon),\theta)=\deg_{\rm
BF}(\Gamma_0, B_X(\theta, \epsilon),\theta)=\deg_{\rm BF}(\Gamma_3,
B_X(\theta, \epsilon),\theta).
\end{equation}
Recall that $D\Gamma_3(\theta)=A'(\theta)=B(\theta)|_X$ and
$$
\Gamma_3(z+u)=(I-P^0)A'(\theta)u+ P^0A(z+ h(z))=I- [P^0B(\theta)u-
P^0A(z+ h(z))].
$$
Moreover $\dim H^0<\infty$ implies that the map
$$
\bar B_X(\theta,\epsilon)\to X,\;z+u\mapsto K(z+u):=P^0B(\theta)u-
P^0A(z+ h(z))
$$
 is compact. Hence the Leray-Schauder degree
$\deg_{\rm LS}(I-K, B_X(\theta,\epsilon),\theta)$ exists, and
\begin{eqnarray}\label{e:5.4}
\deg_{\rm FPR}(I-K, B_X(\theta,\epsilon),\theta)&=&\deg_{\rm
BF}(I-K, B_X(\theta,\epsilon),\theta)\nonumber\\
&=&\deg_{\rm LS}(I-K, B_X(\theta,\epsilon),\theta)
\end{eqnarray}
for a suitable orientation of the map $I-K$. By Remark~\ref{lem:3.2}
and Lemma~\ref{lem:3.1}  $\mathcal{ L}^\circ$ is $C^{2}$ and
$$
d\mathcal{ L}^\circ(z_0)(z)=(A(z_0+ h(z_0)), z)_H\quad\forall z_0\in
B_{H^0}(\theta, r_0),\; z\in H^0.
 $$
Hence the gradient of $\mathcal{ L}^\circ$ with respect to the
induced inner on $H^0$ (from $H$), denoted by  $\nabla\mathcal{
L}^\circ$,  is given by $\nabla\mathcal{
L}^\circ(z)=P^0A(z+h(z))\;\forall z\in B_{H^0}(\theta, r_0)$. By the
definition and properties of the Leray-Schauder degree it is easily
proved that
\begin{equation}\label{e:5.5}
\deg_{\rm LS}(I-K, B_X(\theta,\epsilon),\theta)=(-1)^{\dim
H^-}\deg_{\rm LS}(\nabla\mathcal{ L}^\circ,
B_{H^0}(\theta,\epsilon),\theta)
\end{equation}
Moreover, $B_X(\theta,\epsilon)$ is open, connected and simply
connected. After a suitable orientation is chosen it follows from
(\ref{e:5.3})-(\ref{e:5.5}) that
\begin{eqnarray*}
\deg_{\rm FPR}(A, B_X(\theta, \epsilon),\theta)&=&\deg_{\rm BF}(A,
B_X(\theta, \epsilon),\theta)\nonumber\\
&=&(-1)^{\dim H^-}\deg_{\rm LS}(\nabla\mathcal{ L}^\circ,
B_X(\theta,\epsilon)\cap H^0,\theta)\\
&=&(-1)^{\dim H^-}\sum^\infty_{q=0}(-1)^q{\rm rank}C_q(\mathcal{
L}^\circ, \theta;{\bf K}),
\end{eqnarray*}
where the final equality comes from \cite[Th.8.5]{MaWi}. Combing
this with Corollary~\ref{cor:2.6} the expected first conclusion is
obtained.

\noindent{\bf Step 2}.  Recall that a map $T$ from a reflexive real
Banach space to its dual $X^\ast$ is said to be {\it demicontinuous}
if $T$ maps strongly convergent sequences in $X$ to weakly
convergent sequences in $X^\ast$. Now since the Hilbert space $H$ is
self-adjoint and $D\mathcal{ L}(x)(u)=(\nabla\mathcal{ L}(x),u)_H$,
by the continuously directional differentiability of $\mathcal{ L}$,
if $\{x_n\}\subset V$ converges to $x\in V$ in $H$ then
$\{\nabla\mathcal{ L}(x_n)\}$ weakly converges to $\nabla\mathcal{
L}(x)$, i.e., $(\nabla\mathcal{ L}(x_n),u)_H\to (\nabla\mathcal{
L}(x),u)_H$ for every $u\in H$. This shows that the map
$\nabla\mathcal{ L}:V\to H=H^\ast$ is demicontinuous in the sense of
\cite[Th.4]{Bro}.

Next we show that the restriction of $\nabla\mathcal{ L}$ to a small
neighborhood of $\theta\in H$ is of class $(S)_+$ in the sense of
\cite[Def.2(b)]{Bro}.
 By (D3), for the constants $\eta_0$ and
$C_0'$ in ({\rm D4*}) and $\rho>0$ in (i)-(iii) above we can choose
$\rho_0\in (0, \rho)$ such that $2\rho_0<\eta_0$ and the following
(iv)-(v) are satisfied:
\begin{enumerate}
\item[\bf (iv)] $B_{H^0}(\theta, 2\rho_0)\subset B_X(\theta, \rho)$
and
\begin{equation}\label{e:5.6}
\|Q(x)-Q(\theta)\|<\frac{C_0'}{2}\quad\forall x\in B_H(\theta,
2\rho_0)\cap X;
\end{equation}
\item[\bf (v)] $\theta$ is a unique zero of  $\nabla\mathcal{ L}$ in $B_H(\theta,
 2\rho_0)\subset V$.
\end{enumerate}
Then (\ref{e:5.6}) and ({\rm D4*}) yield
\begin{eqnarray}\label{e:5.7}
\bigl(B(x)u,u\bigr)_H&=&\bigl(P(x)u,u\bigr)_H+
\bigl([Q(x)-Q(\theta)]u,u\bigr)_H+ \bigl(Q(\theta)u, u\bigr)_H\nonumber\\
&\ge& \frac{C_0'}{2}\|u\|^2+ \bigl(Q(\theta)u, u\bigr)_H
\end{eqnarray}
for all $x\in B_H(\theta, 2\rho_0)\cap X$ and $u\in H$. Take
$\rho_1\in (0, \rho_0)$ so small that
$$
z+ h(z)\in B_H(\theta, \rho_0)\quad\forall z\in B_{H^0}(\theta,
2\rho_1).
$$
(This assures that the functional $\mathcal{ L}^\circ$ in
Corollary~\ref{cor:2.6} is defined on $B_{H^0}(\theta, 2\rho_1)$).
Then for $x, x'\in B_H(\theta, 2\rho_1)\cap X$, by (F2)-(F3) and the
mean value theorem we have $\tau\in (0, 1)$ such that
\begin{eqnarray*}
&&(\nabla\mathcal{ L}(x), x-x')_H\\
&=&(\nabla\mathcal{ L}(x)-\nabla\mathcal{ L}(x'),
x-x')_H-(\nabla\mathcal{
L}(x'), x-x')_H\\
&=&(A(x)-A(x'), x-x')_H-(\nabla\mathcal{ L}(x'), x-x')_H\\
&=&\bigl(DA([\tau x+ (1-\tau)x'])(x-x'), x-x'\bigr)_H-(\nabla\mathcal{ L}(x'), x-x')_H\\
&=&\bigl(B([\tau x+ (1-\tau)x'])(x-x'), x-x'\bigr)_H-(\nabla\mathcal{ L}(x'), x-x')_H\\
&\ge& \frac{C_0'}{2}\|x-x'\|^2-(\nabla\mathcal{ L}(x'), x-x')_H+
(Q(\theta)(x-x'), x-x')_H,
\end{eqnarray*}
where the final inequality is because of (\ref{e:5.7}). Since
$\mathcal{ L}$ is continuously directional differentiable and
$B_H(\theta, 2\rho_1)\cap X$ is dense in $B_H(\theta, 2\rho_1)$ we
obtain
\begin{eqnarray}\label{e:5.8}
(\nabla\mathcal{ L}(x), x-x')_H &\ge&
\frac{C_0'}{2}\|x-x'\|^2-(\nabla\mathcal{
L}(x'), x-x')_H\nonumber\\
&+& (Q(\theta)(x-x'), x-x')_H
\end{eqnarray}
for any $x, x'\in B_H(\theta, 2\rho_1)$.

 Let $\{x_n\}\subset
B_H(\theta, 2\rho_1)$ weakly converge to $x\in B_H(\theta, 2\rho_1)$
and
$$
\varlimsup_{n\to\infty}(\nabla\mathcal{ L}(x_n), x_n-x)_H\le 0.
$$
Then $(\nabla\mathcal{ L}(x), x_n-x)_H\to 0$, and
$(Q(\theta)(x_n-x), x_n-x)_H\to 0$ by the compactness of
$Q(\theta)$. It follows from these and (\ref{e:5.8}) that
$$
\frac{C_0'}{2}\lim_{n\to\infty}\|x_n-x\|\le\frac{C_0'}{2}\varlimsup_{n\to\infty}\|x_n-x\|^2
\le\varlimsup_{n\to\infty}(\nabla\mathcal{ L}(x_n), x_n-x)_H\le 0,
$$
 This is, $\lim_{n\to\infty}\|x_n-x\|=0$. Hence the map $\nabla\mathcal{
L}:B_H(\theta, 2\rho_1)\to H$ is of class $(S)_+$.

Then three equalities in the formula of Theorem~\ref{th:2.12}(ii)
follow from  \cite[Th.1.2]{CiDe}, Corollary~\ref{cor:2.6} and
\cite[Th.8.5]{MaWi}, respectively. \hfill$\Box$\vspace{2mm}

\section{The functor properties of the
splitting lemma}\label{sec:6}

The splitting lemma for $C^2$ functionals on Hilbert spaces has some
natural functor properties. This section studies some corresponding
properties in our setting.

Consider a tuple $(H, X, \mathcal{ L}, A, B=P+ Q)$, where $H$ (resp.
$X$) is a Hilbert (resp. Banach) space satisfying the condition
({\rm S}) as in Section~\ref{sec:2}, the functional $\mathcal{
L}:H\to\R$ and maps $A:X\to X$ and $B:X\to L_s(H)$ satisfy, at least
near the origin $\theta\in H$, the conditions ({\rm F1})-({\rm F3}),
({\rm C1})-({\rm C2}) and ({\rm D}) in Section~\ref{sec:2}. (We can
assume that these conditions are satisfied on $H$ without loss of
generality.)

Let $(\widehat H, \widehat X, \widehat{\mathcal{ L}}, \widehat A,
\widehat B=\widehat P+ \widehat Q)$ be another such a tuple. Suppose
that $J:H\to\widehat H$ is a linear injection satisfying:
\begin{eqnarray}
&&(Ju, Jv)_{\widehat H}=(u, v)_H\quad\forall u, v\in H,\label{e:6.1}\\
&&J(X)\subset\widehat{X}\quad\hbox{and}\quad J|_X\in L(X,
\widehat{X}). \label{e:6.2}
\end{eqnarray}
 Furthermore, we assume
\begin{equation}\label{e:6.3}
\widehat{\mathcal{ L}}\circ J=\mathcal{ L}
\end{equation}
and
\begin{eqnarray}
&& \widehat A(J(x))=J\circ
A(x)\quad\forall x\in X,\label{e:6.4}\\
&&\widehat B(J(x))\circ J=J\circ B(x)\;\forall x\in X.\label{e:6.5}
\end{eqnarray}
Let $H=H^0\oplus H^+\oplus H^-$, $X=H^0\oplus X^+\oplus X^-$ and
$\widehat H=\widehat H^0\oplus \widehat H^+\oplus \widehat H^-$ and
$\widehat X=\widehat H^0\oplus \widehat X^+\oplus \widehat X^-$ be
the corresponding decompositions. Namely, $\widehat H^0={\rm
Ker}(\widehat B(\theta))$, and $\widehat H^+$ (resp. $\widehat H^-$)
is the positive (resp. negative) definite subspace of $\widehat
B(\theta)$. Denote by $P^\ast$ (resp. $\widehat P^\ast$) the
orthogonal projections from $H$ (resp. $\widehat H$) to $H^\ast$
(resp. $\widehat H^\ast$) for $\ast=+, -, 0$.
 Since $\widehat B(\theta)\circ J=J\circ B(\theta)$ by
(\ref{e:6.5}), we have
\begin{equation}\label{e:6.6}
  JH^\star\subset\widehat H^\star, \quad  \widehat P^\star\circ J=J\circ P^\star,\;\star=-,0,+.
\end{equation}

\noindent{\bf Claim~6.1}. $(\widehat B(\theta)|_{\widehat
X^\pm})^{-1}\circ (J|_{X^\pm})=J|_{X^\pm}\circ(
B(\theta)|_{X^\pm})^{-1}$.

In fact, for $v\in X^\pm$ let $y=(\widehat B(\theta)|_{\widehat
X^\pm})^{-1}\circ (J|_{X^\pm})v$. Then $y\in\widehat{X}^\pm$ because
$J(X^\pm)\subset\widehat{X}^\pm$ by (\ref{e:6.2}) and (\ref{e:6.6}),
and $Jv=\widehat{B}(\theta)y$. Note that we may write
$v=B(\theta)|_{X^\pm}u$ for a unique $u\in X^\pm$. It follows that
$J|_{X^\pm}\circ
B(\theta)|_{X^\pm}u=\widehat{B}(\theta)|_{\widehat{X}^\pm}y$ and
hence $\widehat{B}(\theta)(Ju)=\widehat{B}(\theta)y$ by
(\ref{e:6.5}). The latter implies $Ju=y$ since both $Ju$ and $y$ sit
in $\widehat{X}^\pm$. From this  and (\ref{e:6.5}) we deduce that
$Jv=\widehat{B}(\theta)y=\widehat{B}(\theta)(Ju)=J\circ B(\theta)u$
and hence $v=B(\theta)u$. Then $(\widehat B(\theta)|_{\widehat
X^\pm})^{-1}\circ (J|_{X^\pm})v=y=Ju=J\circ
(B(\theta)|_{X^\pm})^{-1}v$. Claim~6.1 is proved.

Assume that the  nullity of $\mathcal{ L}$ at $\theta\in H$
\begin{equation}\label{e:6.7}
 \nu(\mathcal{ L},
\theta):=\dim H^0>0\quad\hbox{and hence}\quad\nu(\widehat{\mathcal{
L}}, \theta)>0
\end{equation}
by (\ref{e:6.6}). Here $\nu(\widehat{\mathcal{ L}},
\theta):=\dim\widehat{H}^0$ is
 nullity of  $\widehat{\mathcal{ L}}$ at $\theta\in\widehat H$.
Corresponding to the map $S$ in (\ref{e:3.3}) let us consider the
map
\begin{eqnarray*}
&&\widehat S: B_{\widehat H^0}(\theta, r_1)\times (B_{\widehat
X}(\theta, r_1)\cap \widehat X^\pm)\to
\widehat X^\pm,\\
&&\qquad \widehat S(\hat{z},\hat{x})=-(\widehat B(\theta)|_{\widehat
X^\pm})^{-1}(I_{\widehat X}-\widehat P^0)\widehat A(\hat{z}+ \hat{x})+ \hat{x}
\end{eqnarray*}
for  $\hat{z}_1, \hat{z}_2\in B_{\widehat H^0}(\theta, r_1)$ and
$\hat{x}_1, \hat{x}_2\in B_{\widehat X}(\theta, r_1)\cap{\widehat
X}^\pm$. (Here $\widehat X^\pm=\widehat X^+\oplus\widehat X^-$, and
we may shrink $r_1>0$ if necessary). Then for all $z\in
B_{H^0}(\theta, r_1)$ and $x\in B_X(\theta, r_1)\cap X^\pm$ we
derive from  (\ref{e:6.4}) and Claim~6.1 that
\begin{eqnarray*}
\widehat S(Jz, Jx)&=&-(\widehat B(\theta)|_{\widehat
X^\pm})^{-1}(I_{\widehat X}-\widehat P^0)\widehat A(Jz+ Jx)+ Jx\\
&=&-(\widehat B(\theta)|_{\widehat
X^\pm})^{-1}(I_{\widehat X}-\widehat P^0)\circ J\circ A(z+ x)+ Jx\\
&=&-(\widehat B(\theta)|_{\widehat
X^\pm})^{-1}\circ J\circ(I_{X}-P^0) A(z+ x)+ Jx\\
&=&-J\circ(B(\theta)|_{X^\pm})^{-1}\circ (I_{X}-P^0) A(z+ x)+ Jx.
\end{eqnarray*}
That is, for all $z\in B_{H^0}(\theta, r_1)$ and $x\in B_X(\theta,
r_1)\cap X^\pm$ it holds that
\begin{equation}\label{e:6.8}
\widehat S(Jz, Jx)=J\circ S(z, x).
\end{equation}
  By the proof of Lemma~\ref{lem:3.1} there exist $r_0\in
(0, r_1)$ and a unique map $\hat h:B_{\widehat H^0}(\theta, r_1)\to
B_{\widehat X}(\theta, r_1)\cap \widehat X^\pm$ such that $\hat
h(\theta)=\theta$ and
$$
\widehat S(\hat{z}, \hat h(\hat{z}))=\hat h(\hat{z})\quad\hbox{(or equivaliently
$(I_{\widehat X}-\widehat P^0)\widehat A(\hat{z}+ \hat h(\hat{z}))=0$)}.
$$
Moreover, $\hat h$ satisfies the corresponding conclusions in
Lemma~\ref{lem:3.1}. For $z\in B_{H^0}(\theta, r_0)$ we have also
$(I_{X}- P^0) A({z}+  h({z}))=0$, i.e., $S(z,h(z))=h(z)$.   Hence by
the uniqueness and (\ref{e:6.8}) we arrive at
\begin{equation}\label{e:6.9}
\hat h(Jz)=J\circ h(z)\quad\forall z\in B_{H^0}(\theta, r_0).
\end{equation}

As in (\ref{e:3.12}), we have a  map $\widehat F:\bar B_{\widehat
H^0}(\theta, \delta)\times B_{\widehat H^\pm}(\theta, \delta)\to\R$
given by
\begin{equation}\label{e:6.10}
 \widehat F(\hat{z}, \hat{u})=\widehat{\mathcal{ L}}(\hat{z}+ \hat h(\hat{z})+ \hat{u})-
 \widehat{\mathcal{ L}}(\hat{z}+ \hat h(\hat{z})).
\end{equation}
Clearly, (\ref{e:6.3}), (\ref{e:6.9}) and (\ref{e:6.10}) lead to
\begin{equation}\label{e:6.11}
 \widehat F(Jz, Ju)=F(z, u)\quad\forall (z,u)\in\bar
B_{H^0}(\theta, \delta)\times B_{H^\pm}(\theta, \delta).
\end{equation}
By shrinking $\varepsilon>0$ in Lemma~\ref{lem:3.5} (if necessary)
we may assume that the restriction of $\widehat F$  to $\bar
B_{\widehat H^0}(\theta,\varepsilon)\times \bigl(\bar B_{\widehat
H^+}(\theta,\varepsilon)\oplus \bar B_{\widehat
H^-}(\theta,\varepsilon)\bigr)$ satisfies the conditions in
Theorem~\ref{th:A.1}. Then we have a homeomorphism  as in
(\ref{e:3.29}) (shrinking $\epsilon>0$ if necessary),
\begin{eqnarray}\label{e:6.12}
&&\widehat\Phi:B_{\widehat H^0}(\theta,\varepsilon)\times
\left(B_{\widehat H^+}(\theta,
\epsilon)+ B_{\widehat H^-}(\theta, \epsilon)\right)\to \widehat H,\\
&&\hspace{10mm} (\hat{z}, \hat{u}^+ + \hat{u}^-)\mapsto \hat{z}+
\hat h(\hat{z})+ \widehat\phi_{\hat{z}}(\hat{u}^+ +
\hat{u}^-),\nonumber
\end{eqnarray}
such that $\widehat\phi_{\hat{z}}(\theta )=\theta$ and
$$
 \widehat{\mathcal{L}}(\widehat\Phi(\hat{z}, \hat{u}^+, \hat{u}^-))
 =\widehat{\mathcal{L}}(\hat{z}+ \hat h(\hat{z}))+  (\hat{u}^+, \hat{u}^+)_{\widehat H}-
 (\hat{u}^-, \hat{u}^-)_{\widehat H}
$$
for all $(\hat{z}, \hat{u}^+, \hat{u}^-)\in \bar B_{\widehat
H^0}(\theta,\varepsilon)\times B_{\widehat H^+}(\theta,
\epsilon)\times B_{\widehat H^-}(\theta, \epsilon)$.

\noindent{\bf Claim 6.2.} {\it Under the assumptions above, if
\begin{eqnarray}\label{e:6.13}
\mu(\mathcal{ L}, \theta)=\mu(\widehat{\mathcal{ L}}, \theta),
\end{eqnarray}
then $\widehat\Phi(Jz, Ju^+ + Ju^-)= J\circ\Phi(z, u^+ + u^-)$
 for $(z, u^+, u^-)\in \bar B_{H^0}(\theta,\varepsilon)\times B_{H^+}(\theta,
\epsilon)\times B_{H^-}(\theta, \epsilon)$. Here $\mu(\mathcal{ L},
\theta):=\dim H^-$ and $\mu(\widehat{\mathcal{ L}},
\theta):=\dim\widehat{H}^-$. }

In fact, suppose $\mu(\mathcal{ L}, \theta)=\mu(\widehat{\mathcal{ L}}, \theta)=0$.
By $1^\circ$) in the proof of Theorem~\ref{th:A.1}
$$
\widehat\psi(\hat{z}, \hat{x})=\left\{\begin{array}{ll}
 \frac{\sqrt{
\widehat{\mathcal{ L}}(\hat{z}+ \hat h(\hat{z})+ \hat{x})-
 \widehat{\mathcal{ L}}(\hat{z}+ \hat h(\hat{z}))}}
 {\|\hat{x}\|_{\widehat H}}
 \hat{x} &\;\hbox{if}\;\hat{x}\ne \theta,\\
 \theta&\;\hbox{if}\;\hat{x}=\theta
 \end{array}\right.
 $$
for all $(\hat{z}, \hat{x})\in \bar B_{\widehat
H^0}(\theta,\varepsilon)\times B_{\widehat H^\pm}(\theta,
\epsilon_1)$, and
$$
\psi({z}, {x})=\left\{\begin{array}{ll}
 \frac{\sqrt{
{\mathcal{ L}}({z}+  h({z})+ {x})-
 {\mathcal{ L}}({z}+ h({z}))}}
 {\|{x}\|_{ H}}
 {x} &\;\hbox{if}\;{x}\ne \theta,\\
 \theta&\;\hbox{if}\;{x}=\theta
 \end{array}\right.
 $$
for all $({z}, {x})\in \bar B_{ H^0}(\theta,\varepsilon)\times B_{
H^\pm}(\theta, \epsilon_1)$. It follows from (\ref{e:6.3}) and
(\ref{e:6.9}) that
$$
\widehat\psi(Jz, Ju)=J\circ\psi(z, u) \quad\hbox{and thus}\quad
\widehat\phi_{Jz}(Ju)=J\circ\phi_{z}(u)
$$
 for $(z, u)\in \bar B_{H^0}(\theta,\varepsilon)\times B_{H^\pm}(\theta,
\epsilon)$. The desired results follow from (\ref{e:3.29}) and
(\ref{e:6.12}).

Next suppose $\mu(\mathcal{ L}, \theta)=\mu(\widehat{\mathcal{ L}}, \theta)>0$.
Recall the constructions of $\phi_z$ and $\widehat\phi_{\hat{z}}$.
 By (\ref{e:A.9}),
$$\widehat\phi_{\hat{z}}(\hat{u}^+ + \hat{u}^-)=\hat{x}^+ + \hat{x}^-$$
 for any $(\hat{z}, \hat{u}^+, \hat{u}^-)\in \bar B_{\widehat
H^0}(\theta,\varepsilon)\times B_{\widehat H^+}(\theta,
\epsilon)\times B_{\widehat H^-}(\theta, \epsilon)$,
 where $(\hat{x}^+, \hat{x}^-)$  is a unique point in
$B_{\widehat H^+}(\theta, 2\epsilon)\times B_{\widehat
H^-}(\theta, \delta)$ satisfying $\widehat\psi(\hat{z}, \hat{x}^+ +\hat{x}^-)=\hat{u}^+
+\hat{u}^-$. By Step 4 in the proof of Theorem~\ref{th:A.1} we know
$$
\widehat\psi(\hat{z}, \hat{x}^+ + \hat{x}^-)=\widehat\psi_1(\hat{z}, \hat{x}^+ + \hat{x}^-)+
\widehat\psi_2(\hat{z}, \hat{x}^+ + \hat{x}^-)
$$
for all $(\hat{z}, \hat{x}^+, \hat{x}^-)\in \bar B_{\widehat
H^0}(\theta,\varepsilon)\times B_{\widehat H^+}(\theta,
\epsilon_1)\times B_{\widehat H^-}(\theta,\delta)$, where
$$
\widehat\psi_1(\hat{z},
\hat{x}^++\hat{x}^-)=\left\{\begin{array}{ll}
 \frac{\sqrt{\widehat F(\hat{z}, \hat{x}^+ + \widehat\varphi_{\hat{z}}(\hat{x}^+))}}{\|\hat{x}^+\|_{\widehat H}}
 \hat{x}^+ &\;\hbox{if}\;\hat{x}^+\ne \theta,\\
 \theta&\;\hbox{if}\;\hat{x}^+=\theta
 \end{array}\right.
 $$
 and
 {
$$\widehat\psi_2(\hat{z}, \hat{x}^++ \hat{x}^-)=\left\{\begin{array}{ll}
 \frac{\sqrt{\widehat F(\hat{z}, \hat{x}^+ + \widehat\varphi_{\hat{z}}(\hat{x}^+))-
 \widehat F(\hat{z}, \hat{x}^+ +\hat{x}^-)}}{\|\hat{x}^--\widehat\varphi_{\hat{z}}(\hat{x}^+)\|_{\widehat H}}(\hat{x}^--\widehat\varphi_{\hat{z}}(\hat{x}^+))
  &\;\hbox{if}\;\hat{x}^-\ne\widehat\varphi_{\hat{z}}(\hat{x}^+),\\
 \theta&\;\hbox{if}\;\hat{x}^-=\widehat\varphi_{\hat{z}}(\hat{x}^+).
 \end{array}\right.
$$}
Here  for each $(\hat{z}, \hat{x}^+)\in \bar B_{\widehat
H^0}(\theta,\varepsilon)\times B_{\widehat H^+}(\theta,
\epsilon_1)$, as showed in Step 1 of the proof of
Theorem~\ref{th:A.1}, $\widehat\varphi_{\hat{z}}(\hat{x}^+)$
 is a unique point in
$B_{\widehat H^-}(\theta,\delta)$ such that
$$
\widehat F(\hat{z}, \hat{x}^++\widehat\varphi_{\hat{z}}(\hat{x}^+))=\max\bigl\{\widehat F(\hat{z},
\hat{x}^++ \hat{x}^-)\,|\, \hat{x}^-\in B_{\widehat H^-}(\theta,\delta)\bigr\}.
$$
For $(z, x^+)\in \bar B_{H^0}(\theta,\varepsilon)\times
B_{H^+}(\theta, \epsilon_1)$ we have $(Jz, Jx^+)\in \bar B_{\widehat
H^0}(\theta,\varepsilon)\times B_{\widehat H^+}(\theta, \epsilon_1)$
by  (\ref{e:6.6}), and $J(B_{ H^-}(\theta,\delta))=B_{\widehat
H^-}(\theta,\delta)$ by (\ref{e:6.1}), (\ref{e:6.6}) and
(\ref{e:6.13}). These and  (\ref{e:6.11}) lead to
\begin{eqnarray*}
\widehat F(Jz, J{x}^++\widehat\varphi_{J{z}}(J{x}^+))&=&\max\bigl\{\widehat F(J{z},
J{x}^++ \hat{x}^-)\,|\, \hat{x}^-\in B_{\widehat H^-}(\theta,\delta)\bigr\}\\
&=&\max\bigl\{\widehat F(J{z},
J{x}^++ \hat{x}^-)\,|\, \hat{x}^-\in J(B_{ H^-}(\theta,\delta))\bigr\}\\
&=&\max\bigl\{ F({z},
{x}^++ {x}^-)\,|\, {x}^-\in B_{ H^-}(\theta,\delta)\bigr\}\\
&=&F(z, {x}^++\varphi_{{z}}({x}^+))\\
&=&\widehat F(Jz, J{x}^++ J\varphi_{{z}}({x}^+)).
\end{eqnarray*}
By the uniqueness we arrive at
$$
\widehat\varphi_{Jz}(Jx^+)=J\varphi_z(x^+)\quad\forall (z, x^+)\in
\bar B_{H^0}(\theta,\varepsilon)\times B_{H^+}(\theta,
\epsilon_1),
$$
which  implies
$$
\widehat\psi(Jz, Jx^+ + Jx^-)=J\circ\psi(z, x^+ + x^-)
$$
for all $(z, x^+, x^-)\in \bar B_{H^0}(\theta,\varepsilon)\times
B_{H^+}(\theta, \epsilon_1)\times B_{H^-}(\theta,\delta)$. From
(\ref{e:6.6}) and the definition of
$\widehat\phi_{\hat{z}}(\hat{u}^+ +\hat{u}^-)$ we deduce that
$$
\widehat\phi_{Jz}(Ju^+ + Ju^-)=J\circ\phi_{z}(u^+ + u^-)
$$
 for $(z, u^+, u^-)\in \bar B_{H^0}(\theta,\varepsilon)\times B_{H^+}(\theta,
\epsilon)\times B_{H^-}(\theta, \epsilon)$. This, (\ref{e:3.29}) and
(\ref{e:6.12}) lead to the conclusion of Claim 6.2.

 Summarizing the above
arguments we have proved the following theorem under the assumptions
(\ref{e:6.7}) and (\ref{e:6.13}).

\begin{theorem}\label{th:6.1}
Let  $(H, X, \mathcal{L}, A, B=P+ Q)$ and $(\widehat H, \widehat X,
\widehat{\mathcal{L}}, \widehat A, \widehat B=\widehat P+ \widehat
Q)$ be two tuples satisfying  the conditions $({\rm S})$, $({\rm
F1})-({\rm F3})$, $({\rm C1})-({\rm C2})$ and $({\rm D})$ in
Section~\ref{sec:2}. Suppose that $J:H\to\widehat H$ is a linear
injection satisfying (\ref{e:6.1})-(\ref{e:6.5}).
 If  $\mu(\mathcal{ L}, \theta)=\mu(\widehat{\mathcal{ L}}, \theta)$  then for the continuous maps
 $h:B_{H^0}(\theta,\epsilon)\to
X^\pm$ and $\hat h:B_{\widehat H^0}(\theta,\epsilon)\to \widehat
X^\pm$, and the origin-preserving homeomorphisms constructed in
Theorem~\ref{th:2.1},
\begin{eqnarray*}
&&\Phi: B_{H^0}(\theta,\epsilon)\times
\left(B_{H^+}(\theta,\epsilon) +
B_{H^-}(\theta,\epsilon)\right)\to W,\\
&&\widehat\Phi: B_{\widehat H^0}(\theta,\epsilon)\times \bigl(
B_{\widehat H^+}(\theta,\epsilon) + B_{\widehat
H^-}(\theta,\epsilon)\bigr)\to \widehat W,
\end{eqnarray*}
it holds that
$$
\hat h(Jz)=J\circ h(z)\quad\hbox{and}\quad \widehat\Phi(Jz, Ju^+ +
Ju^-)= J\circ\Phi(z, u^+ + u^-)
$$
 for all $(z, u^+, u^-)\in B_{H^0}(\theta,\epsilon)\times
B_{H^+}(\theta,\epsilon)\times B_{H^-}(\theta,\epsilon)$.
Consequently,
\begin{eqnarray*}
&&\widehat{\mathcal{L}}\circ\widehat\Phi(Jz, Ju^++ Ju^-)=\mathcal{
L}\circ\Phi(z, u^++ u^-),\\
&&\widehat{\mathcal{L}}(Jz+ \hat h(Jz))=\mathcal{L}(z+ h(z))
\end{eqnarray*}
for all $(z, u^+, u^-)\in B_{H^0}(\theta,\epsilon)\times
B_{H^+}(\theta,\epsilon) \times B_{H^-}(\theta,\epsilon)$.
\end{theorem}

Here we understand $B_{H^0}(\theta,\epsilon)\times
B_{H^+}(\theta,\epsilon) \times B_{H^-}(\theta,\epsilon)$ as
$B_{H^0}(\theta,\epsilon)\times B_{H^+}(\theta,\epsilon)$ if $\dim
H^-=0$, and $B_{H^0}(\theta,\epsilon)\times B_{H^+}(\theta,\epsilon)
\times B_{H^-}(\theta,\epsilon)$ as $B_{H^-}(\theta,\epsilon)\times
B_{H^+}(\theta,\epsilon)$ if $\dim H^0=0$.

Let us prove the remainder cases. Firstly,  consider the case $\nu(\mathcal{ L},
\theta)=\nu(\widehat{\mathcal{ L}}, \theta)=0$. We only need to
remove $z$ and $\hat{z}$ in the arguments below Claim~6.2 and then
replace $\mathcal{F}$ and $\widehat{\mathcal{F}}$ by $\mathcal{L}$
and $\widehat{\mathcal{L}}$, respectively.

Finally,  the case $0=\nu(\mathcal{ L},
\theta)<\nu(\widehat{\mathcal{ L}}, \theta)$ can also be obtained by
combing the above three cases. Theorem~\ref{th:6.1} is proved.

By (\ref{e:6.3}) and (\ref{e:6.9}), for any $z\in B_{H^0}(\theta,
r_0)$ it holds that
\begin{equation}\label{e:6.14}
\widehat{\mathcal{L}}^\circ(Jz)=\widehat{\mathcal{L}}(Jz+\hat
h(Jz))=\mathcal{L}(z+ h(z))=\mathcal{L}^\circ(z).
\end{equation}

\begin{corollary}\label{cor:6.2}
Let  $(H, X, \mathcal{L}, A, B=P+ Q)$ and $(\widehat H, \widehat X,
\widehat{\mathcal{L}}, \widehat A, \widehat B=\widehat P+ \widehat
Q)$ be two tuples satisfying  the conditions $({\rm S})$, $({\rm
F1})-({\rm F3})$, $({\rm C1})-({\rm C2})$ and $({\rm D})$ in
Section~\ref{sec:2}. Suppose that $J:H\to\widehat H$ is a linear
injection satisfying (\ref{e:6.1})-(\ref{e:6.5}).
 If  $\nu(\mathcal{ L}, \theta)=\nu(\widehat{\mathcal{ L}}, \theta)>0$  then
$$
C_q(\widehat{\mathcal{L}}^\circ,\theta;{\bf
K})=C_q(\mathcal{L}^\circ,\theta;{\bf K})\quad\forall
q\in\N\cup\{0\}.
$$
\end{corollary}

\begin{theorem}\label{th:6.3}
Under the assumptions of Theorem~\ref{th:2.1}, let
$(\widehat{H},\widehat{X})$ be another pair of Hilbert-Banach spaces
satisfying {\bf (S)}, and let $J:H\to\widehat{H}$ be a Hilbert space
isomorphism which can induce a Banach space isomorphism $J_X:X\to
\widehat{X}$ (this means that $J(X)\subset \widehat{X}$ and
$J|_X:X\to \widehat{X}$ is a Banach space isomorphism). Set
$\widehat{V}=J(V)$ (and hence
$\widehat{V}^{\widehat{X}}:=\widehat{V}\cap\widehat{X}=J(V^X)$) and
$\widehat{\mathcal{L}}:\widehat{V}\to\R$ by
$\widehat{\mathcal{L}}=\mathcal{L}\circ J^{-1}$. Then
$(\widehat{H},\widehat{X}, \widehat{V}, \widehat{\mathcal{L}})$
satisfies the assumptions of Theorem~\ref{th:2.1} too.
\end{theorem}

\begin{proof}
 Define $\widehat{A}:\widehat{V}^{\widehat
X}\to \widehat{X}$ by $\widehat{A}=J_X\circ A\circ J_X^{-1}$, and
$\widehat{B}:\widehat{V}^{\widehat X}\to \mathcal{L}_s(\widehat{H})$
by $\widehat{B}(\hat x)=J\circ B(J_X^{-1}\hat x)\circ J^{-1}$.
Similarly, we also define $\widehat{P}(\hat x)=J\circ P(J_X^{-1}\hat
x)\circ J^{-1}$ and $\widehat{Q}(\hat x)=J\circ Q(J_X^{-1}\hat
x)\circ J^{-1}$. It is not hard to check that
$(\widehat{H},\widehat{X}, \widehat{V}, \widehat{\mathcal{L}},
\widehat{A},\widehat{B}=\widehat{P}+\widehat{Q})$ satisfies the
assumptions of Theorem~\ref{th:2.1}.
\end{proof}

\begin{theorem}\label{th:6.4}
Under the assumptions of Theorem~\ref{th:2.1}, suppose that
$\check{H}\subset H$ is a Hilbert subspace whose orthogonal
complementary in $H$ is  finite dimensional and is contained in $X$.
Then $(\mathcal{L}|_{\check{H}}, \check{H}, \check{X})$ with
$\check{X}:=X\cap\check{H}$ also satisfies the assumptions of
Theorem~\ref{th:2.1} around the critical point $\theta\in\check{H}$.
\end{theorem}

\begin{proof}
Let $P_{\check{H}}$ be the orthogonal
projection onto $\check{H}$. Then $x-P_{\check{H}}x\in X\;\forall
x\in X$ by the assumption $\check{H}^\bot\subset X$. It follows that
 $\check{A}(x):=P_{\check{H}}A(x)\in \check{X}$ for $x\in V^{\check{X}}:=V^X\cap\check{X}$.
Since $\check{H}^\bot\subset X$ and $\dim\check{H}^\bot<\infty$,
$P_{\check{H}}$ restricts to a bounded linear operator from
$\check{X}$ to $\check{X}$. This implies that
$\check{A}:V^{\check{X}}\to\check{X}$ has the same differentiability
as $A$. It is easily checked that $D\mathcal{
L}|_{\check{H}}(x)(u)=(\check{A}(x), u)_H\;\forall u\in \check{X}$,
and that
$$
(D\check{A}(x)(u),
v)_H=(P_{\check{H}}DA(x)(u),v)_H=(P_{\check{H}}B(x)(u),v)_H=(\check{B}(x)u,v)_H
$$
for any $x\in V^{\check{X}}$, $u,v\in \check{X}$, where
$\check{B}(x):=P_{\check{H}}B(x)|_{\check{H}}\in\mathcal{L}_s(\check{H})$.
Obverse that
\begin{eqnarray*}
\|\check{B}(x_1)-\check{B}(x_2)\|_{{\mathcal{
L}}_s(\check{H})}&=&\sup\{\|\check{B}(x_1)u-\check{B}(x_2)u\|_{\check{H}}:
u\in\check{H}, \|u\|=1\}\\
&\le&\|{B}(x_1)-{B}(x_2)\|_{\mathcal{L}_s({H})}
\end{eqnarray*}
for any $x_1,x_2\in V^{\check{X}}$. So some kind of continuality of
$B$ implies the same continuous property of $\check{B}$. Suppose
that $\check{B}(0)u=v$ for some $u\in\check{H}$ and $v\in
\check{X}$. Then $P_{\check{H}}B(0)u=v$ and therefore $B(0)u=v+
P_{\check{H}^\bot}B(0)u\in X$ because
$P_{\check{H}^\bot}(\check{H})=\check{H}^\bot\subset X$ by the
assumptions. It follows that $u\in X$ and hence $u\in
X\cap\check{H}=\check{X}$. That is, {\bf (C2)} is satisfied. Since
the eigenvectors of $\check{B}(0)$ are those of $B(0)$ too the
condition {\bf (D1)} holds naturally.  For $x\in V\cap \check{X}$
take $\check{P}(x)=P_{\check{H}}\circ P(x)|_{\check{H}}$ and
$\check{Q}(x)=P_{\check{H}}\circ Q(x)|_{\check{H}}$. It is also
clear that $\check{B}(x)=\check{P}(x)+ \check{Q}(x)$ satisfies the
other conditions in {\bf (D)}.
 \end{proof}

\section{An estimation for behavior of $\mathcal{L}$}\label{sec:7}

In this section we shall estimate  behavior of $\mathcal{L}$ near
$\theta$. Such a result will be used in the proof of Theorem~5.1 of
\cite{Lu3}.

 We shall replace the condition ({\bf D4}) in Section
2 by the following stronger
\begin{enumerate}
\item[\bf (D4**)] There
exist positive constants $\eta_0'$ and $C'_2>C'_1$ such that
$$
C'_2\|u\|^2\ge (P(x)u, u)\ge C'_1\|u\|^2\quad\forall u\in
H,\;\forall x\in B_H(\theta,\eta_0')\cap X.
$$
\end{enumerate}

Note that $B(\theta)|_{H^\pm}:H^\pm\to H^\pm=H^-\oplus H^+$ is
invertible. Set
$$
\left.\begin{array}{ll}
&B_\rho^{(\ast)}=\{h\in H^\ast\,|\,\|h\|\le \rho\},\;\ast=+, 0, -,\\
&B^\pm_{(r,s)}= B_r^{(-)}\oplus B_s^{(+)}.
\end{array}\right\}
$$
 For the
neighborhood $U$ in Lemma~\ref{lem:3.4} we fix a small $\rho\in (0,
\eta_0')$ so that
$$
B_\rho^{(0)}\oplus B_\rho^{(-)}\oplus B_\rho^{(+)} \subset U.
$$
We may assume that $a_1$ is no more than $a_0$ in
Lemma~\ref{lem:3.4}. Set
\begin{equation}\label{e:7.1}
a'_1:= \frac{(2C'_2+ \|Q(\theta)\|+1)}{2}+ \frac{1}{3a_1}.
\end{equation}
Since $h(\theta)=\theta$ we can choose $\rho_0\in (0, \rho]$ so
small that  $\omega$ in Lemma~\ref{lem:3.3} and $Q$ in ({\bf D3})
satisfy
\begin{eqnarray}
&&\|Q(z+ h(z)+ u)-Q(\theta)\|\le \frac{C'_1}{2},\label{e:7.2}\\
&&\omega(z+ h(z)+ u)<\sqrt{\frac{a_1}{2}},\label{e:7.3}\\
&&\omega(z+ h(z)+ u)\le\frac{k}{8a_1'}\label{e:7.4}
\end{eqnarray}
for all $z\in B_{\rho_0}^{(0)}$ and $u\in
B^\pm_{(\rho_0,\rho_0)}\cap X$. As before we write
$B_{H^\pm}(\theta,\delta)\cap X$ as $B_{H^\pm}(\theta,\delta)^X$
when it is considered as an open subset of $X^\pm$, and $F^X$ as the
restriction of the functional $F$ in (\ref{e:3.12}) to
$\bar{B}_{H^0}(\theta,\delta)\times B_{H^\pm}(\theta,\delta)^X$.

\begin{proposition}\label{prop:7.1}
Under the assumptions of Theorem~\ref{th:2.1} with ({\bf D4})
replaced by ({\bf D4**}), suppose that the map
$A:V^X\to X$ in the condition ({\bf F2}) is Fr\'echet
differentiable. (This implies that the functional $B_{H^\pm}(\theta,\delta)^X\ni u\to F^X(z,u)$
 is twice Fr\'echet differentiable  for each
fixed $z$). Let $s, r\in (0, \rho_0]$ satisfy
\begin{equation}\label{e:7.5}
B^\pm_{(r,s)}\subset B^\pm_{(\rho_0, \rho_0)}\quad\hbox{for}\quad
r=s\sqrt{\frac{8a_1'}{a_1}}.
\end{equation}
Then for positive constants
\begin{equation}\label{e:7.6}
\varepsilon=a_1's^2\quad\hbox{and}\quad\hbar=\frac{a_1}{8}s^2
\end{equation}
the following conclusions hold.
\begin{enumerate}
\item[{\rm (i)}] $(\nabla_2F(z,u), P^+u)\ge\hbar\quad\forall (z, u)\in
B^{(0)}_{\rho_0}\times B^\pm_{(r,s)}\;\hbox{with}\;\|P^+u\|=s$;
\item[{\rm (ii)}] $(\nabla_2F(z,u), P^-u)\le-\hbar\quad\forall\;(z, u)\in
B^{(0)}_{\rho_0}\times B^\pm_{(r,s)}\;\hbox{with}\;
F(z,u)=-\varepsilon$;
\item[{\rm (iii)}] $F(z,u)\le -\varepsilon\quad\forall (z, u)\in
B^{(0)}_{\rho_0}\times B^\pm_{(r,s)}\;\hbox{with}\;\|P^-u\|=r$.
\end{enumerate}
In particular, taking $z=0$ we get
\begin{eqnarray*}
&\bullet& (\nabla\mathcal{L}(u), P^+u)\ge\hbar\quad\forall  u\in
B^\pm_{(r,s)}\;\hbox{with}\;\|P^+u\|=s,\\
&\bullet& (\nabla\mathcal{L}(u), P^-u)\le-\hbar\quad\forall\;u\in
 B^\pm_{(r,s)}\;\hbox{with}\;
\mathcal{L}(u)=-\varepsilon;\\
&\bullet& \mathcal{L}(u)\le -\varepsilon\quad\forall u\in
 B^\pm_{(r,s)}\;\hbox{with}\;\|P^-u\|=r.
\end{eqnarray*}
\end{proposition}


\begin{proof}
For $u\in B^\pm_{(\rho,\rho)}\cap
X^\pm\setminus\{0\}$, since $H^-\oplus H^0\subset X$,
$P^+u=u-P^-u\in X^\pm$. Hence
\begin{eqnarray*}
& &(\nabla_2F(z,u), P^+u)\\
&=&d_uF(z, u)(P^+u)\\
&=&d\mathcal{L}(z+ h(z)+ u)(P^+u)\\
 &=&d({\mathcal{L}}|_X)(z+ h(z)+ u)(P^+u)\\
&=&d(\mathcal{L}|_X)(z+ h(z)+ u)(P^+u)-d({\mathcal{
L}}|_X)(z+ h(z))(P^+u)\\
&=&d^2(\mathcal{L}|_X)(z+ h(z)+ tu)(u, P^+u)\\
&=&(B(z+ h(z)+ tu)u, P^+u)\\
&=&(B(z+ h(z)+ tu)P^+u, P^+u)+ (B(z+ h(z)+ tu)P^-u, P^+u)
\end{eqnarray*}
for some $t\in (0,1)$. Here the fourth equality is because
\begin{eqnarray*}
d(\mathcal{L}|_X)(z+ h(z))(P^+u)&=&(A(z+h(z)),
P^+u)_H\\
&=&((I-P^0)A(z+h(z)), P^+u)_H=0,
\end{eqnarray*}
and the fifth equality comes from the mean value theorem. It follows
from (i)-(ii) in Lemma~\ref{lem:3.4} that
$$
(\nabla_2F(z,u), P^+u)\ge a_1\|P^+u\|^2-\omega(z+ h(z)+
tu)\|P^-u\|\cdot\|P^+u\|.
$$
Since $2pq\le p^2+q^2$ for any $p, q\in\R$, we deduce that
\begin{eqnarray*}
& &\omega(z+ h(z)+ tu)\|P^-u\|\cdot\|P^+u\|\\
&=&2\omega(z+ h(z)+ tu)\|P^-u\|
\frac{1}{2\sqrt{\eta}} \sqrt{\eta}\|P^+u\|\\
 &\le&
\frac{1}{4\eta}(\omega(z+ h(z)+ tu)\|P^-u\|)^2+ \eta\|P^+u\|^2
\end{eqnarray*}
for any $\eta>0$. Taking $\eta=3a_1/4$, we arrive at
\begin{equation}\label{e:7.7}
(\nabla_2F(z,u), P^+u)\ge
\frac{a_1}{4}\|P^+u\|^2-\frac{1}{3a_1}(\omega(z+ h(z)+
tu)\|P^-u\|)^2
\end{equation}
for all $u\in B^\pm_{(\rho,\rho)}\cap X^\pm\setminus\{0\}$, where
$t=t(u)\in (0, 1)$.

Similarly, for any $u\in B^\pm_{(\rho,\rho)}\cap
X^\pm\setminus\{0\}$ and some $t'=t'(u)\in (0,1)$, we have
\begin{eqnarray*}
&&(\nabla_2F(z,u), P^-u)\\
&=&d_uF(z,u)(P^-u)\\
&=&d\mathcal{L}(z+ h(z)+ u)(P^-u)\\
&=&d({\mathcal{
L}}|_X)(z+ h(z)+ u)(P^-u)\\
&=&d(\mathcal{L}|_X)(z+ h(z)+ u)(P^-u)-d({\mathcal{
L}}|_X)(z+ h(z))(P^-u)\\
&=&d^2(\mathcal{L}|_X)(z+ h(z)+ t'u)(u, P^-u)\\
&=&(B(z+ h(z)+ t'u)u, P^-u)\\
&=&(B(z+ h(z)+ t'u)P^-u, P^-u)+ (B(z+ h(z)+ t'u)P^+u, P^-u).
\end{eqnarray*}
Since for any $\eta>0$,
\begin{eqnarray*}
&&\omega(z+ h(z)+ t'u)\|P^+u\|\cdot\|P^-u\|\\
&=&2\omega(z+ h(z)+ t'u)\|P^+u\|
\frac{1}{2\sqrt{\eta}} \sqrt{\eta}\|P^-u\|\\
 &\le&
\frac{1}{4\eta}(\omega(z+ h(z)+ t'u)\|P^+u\|)^2+ \eta\|P^-u\|^2,
\end{eqnarray*}
taking $\eta=3a_1/4$, as above we derive from (ii)-(iii) of
Lemma~\ref{lem:3.4} that
\begin{eqnarray}\label{e:7.8}
&&(\nabla_2F(z,u), P^-u)\nonumber\\
&\le& -a_1\|P^-u\|^2
+\omega(z+ h(z)+ t'u)\|P^+u\|\cdot\|P^-u\|\nonumber\\
&\le& -\frac{a_1}{4}\|P^-u\|^2+ \frac{1}{3a_1}(\omega(z+ h(z)+
t'u)\|P^+u\|)^2.
\end{eqnarray}

Since the functional $B_{H^\pm}(\theta,\delta)^X\ni u\to F^X(z,u)$
 is twice Fr\'echet differentiable  for each
fixed $z$,
by the Taylor formula, for $u\in B^\pm_{(\rho_0,\rho_0)}\cap
X\setminus\{\theta\}$,
\begin{eqnarray}\label{e:7.9}
F(z,u)&=&F(z,\theta)+ \frac{1}{2}d^2_uF^X(z, t''u)(u,u)\nonumber\\
&=&\frac{1}{2}d^2(\mathcal{L}|_X)(z+ h(z)+ t''u)(u,u)\nonumber\\
&=&\frac{1}{2}(B(z+ h(z)+ t''u)u, u)\nonumber\\
&=&\frac{1}{2}(B(z+ h(z)+ t''u)P^-u, P^-u)\nonumber\\
& & + (B(z+ h(z)+ t''u)P^-u, P^+u)\nonumber\\
& &+ \frac{1}{2}(B(z+ h(z)+ t''u)P^+u, P^+u)
\end{eqnarray}
for some $t''=t''(u)\in (0, 1)$. As in the proof of (\ref{e:7.8}) we
have
\begin{eqnarray}\label{e:7.10}
&&\frac{1}{2}(B(z+ h(z)+ t''u)P^-u, P^-u)+ (B(z+ h(z)+ t''u)P^-u, P^+u)\nonumber\\
&&\le -\frac{a_1}{2}\|P^-u\|^2
+\omega(z+ h(z)+ t''u)\|P^+u\|\cdot\|P^-u\|\nonumber\\
&&\le -\frac{a_1}{4}\|P^-u\|^2+
\frac{1}{a_1}(\omega(z+ h(z)+ t''u)\|P^+u\|)^2\nonumber\\
&&\le -\frac{a_1}{4}\|P^-u\|^2+ \frac{1}{2}\|P^+u\|^2
\end{eqnarray}
by (\ref{e:7.3}). In addition, Since $C'_1<C'_2$, by the condition
({\bf D4**}) and (\ref{e:7.2})-(\ref{e:7.3}),
\begin{eqnarray*}
&&(B(z+ h(z)+ t''u)P^+u, P^+u)\\
&=&(P(z+ h(z)+ t''u)P^+u,P^+u)+ (Q(z+ h(z)+ t''u)P^+u,
P^+u)\\
&\le& C'_2\|P^+u\|^2+ (C'_2+ \|Q(\theta)\|)\|P^+u\|^2.
\end{eqnarray*}
  From this and (\ref{e:7.9})-(\ref{e:7.10}) it follows that for
  any $(z,u)\in B^{(0)}_{\rho_0}\times (B^\pm_{(\rho_0,
\rho_0)}\cap X)$,
\begin{equation}\label{e:7.11}
F(z, u)\le -\frac{a_1}{4}\|P^-u\|^2+  \frac{(2C'_2+
\|Q(\theta)\|+1)}{2}\|P^+u\|^2.
\end{equation}

As in the proof of (\ref{e:7.7})  we have
\begin{eqnarray}\label{e:7.12}
&&\frac{1}{2}(B(z+ h(z)+ t''u)P^+u, P^+u)+ (B(z+ h(z)+ t''u)P^-u, P^+u)\nonumber\\
&&\ge \frac{a_1}{2}\|P^+u\|^2
-\omega(z+ h(z)+ t''u)\|P^-u\|\cdot\|P^+u\|\nonumber\\
&&\ge \frac{a_1-\eta}{2}\|P^+u\|^2- \frac{1}{2\eta}(\omega(z+ h(z)+
t''u)\|P^-u\|)^2
\end{eqnarray}
for any $0<\eta<a_1$ because
$$\omega(z+ h(z)+
t''u)\|P^-u\|\cdot\|P^+u\|\le\frac{\eta}{2}\|P^+u\|^2+
\frac{1}{2\eta}(\omega(z+ h(z)+ t''u)\|P^-u\|)^2.
$$
Note that the condition ({\bf D4**}) and (\ref{e:7.2}) imply
\begin{eqnarray*}
&&(B(z+ h(z)+ t''u)P^-u, P^-u)\\
&=&(P(z+ h(z)+ t''u)P^-u,P^-u)+ (Q(z+ h(z)+ t''u)P^-u,
P^-u)\\
&\ge& C'_1\|P^-u\|^2+ (Q(z+ h(z)+ t''u)P^-u, P^-u)\\
&\ge& C'_1\|P^-u\|^2+ (-\frac{C'_1}{2}- \|Q(\theta)\|)\|P^-u\|^2\\
&=& \left(\frac{C'_1}{2}- \|Q(\theta)\|\right)\|P^-u\|^2.
\end{eqnarray*}
 From this, (\ref{e:7.9}), (\ref{e:7.12}) and (\ref{e:7.3}) we derive
\begin{equation}\label{e:7.13}
F(z,u)\ge \frac{a_1-\eta}{2}\|P^+u\|^2 -
\left[\frac{a_1}{4\eta}-\frac{C'_1}{4} +
\frac{\|Q(\theta)\|}{2}\right]\|P^-u\|^2
\end{equation}
for all $(z,u)\in B^{(0)}_{\rho_0}\times (B^\pm_{(\rho_0,
\rho_0)}\cap X)$.

Let us take $\eta$ such that
$$
\frac{a_1}{4\eta}=\frac{C'_1}{4}+ C'_2+ \frac{1}{2}
$$
Then $0<\eta<a_1/8$, and by (\ref{e:7.1})
$$
a_1'= \frac{(2C'_2+ \|Q(\theta)\|+1)}{2}+
\frac{1}{3a_1}=\left[\frac{a_1}{4\eta}-\frac{C'_1}{4} +
\frac{\|Q(\theta)\|}{2}\right]+ \frac{1}{3a_1}.
$$
It follows from (\ref{e:7.11}) and (\ref{e:7.13}) that
$$
\frac{a_1}{4}\|P^+u\|^2 -  a_1'\|P^-u\|^2\le F(z,u)\le
-\frac{a_1}{4}\|P^-u\|^2+  a_1'\|P^+u\|^2
$$
for any  $(z,u)\in B^{(0)}_{\rho_0}\times (B^\pm_{(\rho_0,
\rho_0)}\cap X)$. This implies
\begin{equation}\label{e:7.14}
\frac{a_1}{4}\|P^+u\|^2 -  a_1'\|P^-u\|^2\le F(z,u)\le
-\frac{a_1}{4}\|P^-u\|^2+  a_1'\|P^+u\|^2
\end{equation}
for all $(z,u)\in B^{(0)}_{\rho_0}\times B^\pm_{(\rho_0, \rho_0)}$
because
 $B^{(0)}_{\rho_0}\times (B^\pm_{(\rho_0,
\rho_0)}\cap X)$ is dense in $B^{(0)}_{\rho_0}\times B^\pm_{(\rho_0,
\rho_0)}$.

Moreover, since $a_1'>\frac{1}{3a_1}$,
 by (\ref{e:7.7}) and (\ref{e:7.8}),
 for any $(z,u)\in B^{(0)}_{\rho_0}\times (B^\pm_{(\rho_0,
\rho_0)}\cap X)$ with $u\ne 0$ there exist $t=t(u)\in (0,1)$ and
$t'=t'(u)\in (0, 1)$ such that
\begin{eqnarray}
\hspace{-2mm}(\nabla_2F(z,u), P^+u)\ge
\frac{a_1}{4}\|P^+u\|^2-a_1'(\omega(z+ h(z)+
tu))^2\|P^-u\|^2,\label{e:7.15}
\end{eqnarray}
\begin{eqnarray}
\hspace{-2mm} (\nabla_2F(z,u), P^-u)\le -\frac{a_1}{4}\|P^-u\|^2+
a_1'(\omega(z+ h(z)+ t'u))^2\|P^+u\|^2.\label{e:7.16}
\end{eqnarray}

Now we may prove that the positive constants $r$, $s$, $\varepsilon$
and $\hbar$ in (\ref{e:7.5})-(\ref{e:7.6}) satisfy (i)-(iii).

Firstly, for any $(z, u)\in B^{(0)}_{\rho_0}\times B^\pm_{(r,s)}$
with $\|P^-u\|=r$ it follows from (\ref{e:7.14}) that
$$
 F(z,u)\le
-\frac{a_1}{4}\|P^-u\|^2+  a_1'\|P^+u\|^2\le -\frac{a_1}{4}r^2+
a_1's^2=-a_1's^2=-\varepsilon.
$$

Next, by (\ref{e:7.15}) and (\ref{e:7.4}) we have
$$
(\nabla_2F(z,u), P^+u)\ge
\frac{a_1}{4}\|P^+u\|^2-\frac{a_1^2}{64a_1'}\|P^-u\|^2
$$
for any $(z,u)\in B^{(0)}_{\rho_0}\times (B^\pm_{(\rho_0,
\rho_0)}\cap X)$. The density of $B^{(0)}_{\rho_0}\times
(B^\pm_{(\rho_0, \rho_0)}\cap X)$  in $B^{(0)}_{\rho_0}\times
B^\pm_{(\rho_0, \rho_0)}$ implies that this inequality also holds
for any $(z, u)\in B^{(0)}_{\rho_0}\times B^\pm_{(\rho_0, \rho_0)}$.
So for any $(z, u)\in B^{(0)}_{\rho_0}\times B^\pm_{(r, s)}$ with
$\|P^+u\|=s$ we have
\begin{eqnarray*}
 (\nabla_2F(z,u), P^+u)&\ge&
\frac{a_1}{4}\|P^+u\|^2-\frac{a_1^2}{64a_1'}\|P^-u\|^2\nonumber\\
&\ge&
\frac{a_1}{4}s^2-\frac{a_1^2}{64a_1'}r^2=\frac{a_1}{8}s^2=\hbar.
\end{eqnarray*}

Finally, for any $(z, u)\in (B^{(0)}_{\rho_0}\times B^\pm_{(r,
s)})\cap\{F(z,u)\le -\varepsilon\}$, by (\ref{e:7.14}) we get
\begin{equation}\label{e:7.17}
\frac{a_1}{4}\|P^+u\|^2- a_1'\|P^-u\|^2\le -\varepsilon.
\end{equation}
This implies $a_1'\|P^-u\|^2\ge \varepsilon$, and thus $u\ne 0$. If
this $u$  also belongs to $X$, then it follows from this,
(\ref{e:7.16}) and (\ref{e:7.4}) that
\begin{eqnarray*}
(\nabla_2F(z,u), P^-u)&\le& -\frac{a_1}{4}\|P^-u\|^2+
a_1'(\omega(z+ h(z)+ t'u))^2\|P^+u\|^2\\
&\le&-\frac{a_1}{4}\|P^-u\|^2+
\frac{a_1^2}{64a_1'}\|P^+u\|^2\hspace{25mm}\hbox{by}\;(\ref{e:7.4})\\
&\le&-\frac{a_1}{4}\|P^-u\|^2+
\frac{a_1^2}{64a_1'}\frac{4}{a_1}\Bigl[a_1'\|P^-u\|^2-\varepsilon\Bigr]\hspace{6mm}\hbox{by}\;{\rm (\ref{e:7.17})}\\
&\le&-\frac{a_1}{4}\|P^-u\|^2+
\frac{a_1}{16}\|P^-u\|^2-\frac{a_1\varepsilon}{16a_1'}\\
&=&-\frac{3a_1}{16}\|P^-u\|^2-\frac{a_1\varepsilon}{16a_1'}\\
&\le
&-\frac{3a_1}{16}\frac{\varepsilon}{a_1'}-\frac{a_1\varepsilon}{16a_1'}
=-\frac{a_1\varepsilon}{4a_1'}.
\end{eqnarray*}
Since $\bigl((B^{(0)}_{\rho_0}\times (B^\pm_{(r, s)})\cap
X\bigr)\cap\{F(z,u)\le -\varepsilon\}$ is dense in $
\bigl(B^{(0)}_{\rho_0}\times (B^\pm_{(r, s)})\bigr)\cap\{F(z,u)\le
-\varepsilon\}$ we deduce that
$$
(\nabla_2F(z,u), P^-u)\le -\frac{a_1\varepsilon}{4a_1'}<-\hbar
$$
for all $(z,u)\in\bigl(B^{(0)}_{\rho_0}\times (B^\pm_{(r,
s)})\bigr)\cap\{F(z,u)\le -\varepsilon\}$.
\end{proof}

\section{Concluding remarks}\label{sec:8}

In this section we shall show that
some conclusions of Theorem~\ref{th:2.1} can still be obtained if
the strictly Fr\'{e}chet differentiability  at $\theta$ of the map $A: V^X\to X$
is replaced by a weaker condition
similar to $({\rm E_\infty})$ or $({\rm E'_\infty})$ in Theorems~4.1
and 4.3 of \cite{Lu2}. That is, the condition (F2) can be replaced by the following weaker
(${\bf F2^\prime}$) or (${\bf F2^{\prime\prime}}$).
\begin{enumerate}
\item[(${\bf F2^\prime}$)] There exists a  continuously directional
differentiable (and thus $C^{1-0}$) map $A: V^X\to X$ such that
$D\mathcal{ L}(x)(u)=(A(x), u)_H$ for all $x\in V^X$ and $u\in X$
(which actually implies that $\mathcal{ L}|_{V^X}\in C^1(V^X, \R)$),
and that
\begin{eqnarray}\label{e:N1}
&&\|(I-P^0)A(z_1+x_1)-B(\theta)x_1-(I-P^0)A(z_2+x_2)+
B(\theta)x_2\|_{X^\pm}\nonumber\\
&&\le\frac{1}{\kappa C_1}\|z_1+x_1-z_2-x_2\|_{X}
\end{eqnarray}
for some positive numbers $\kappa>1$, $r_1>0$ and all $z_i\in
B_{H^0}(\theta,r_1)$, $x_i\in B_X(\theta, r_1)\cap X^\pm$, $i=1,2$.
Here $C_1$ is given by (\ref{e:3.2}).

\item[(${\bf F2^{\prime\prime}}$)] The inequality (\ref{e:N1}) in (${\bf
F2^\prime}$) is replaced by
\begin{eqnarray}\label{e:N2}
&&\|(I-P^0)A(z+x_1)-B(\theta)x_1-(I-P^0)A(z+x_2)+
B(\theta)x_2\|_{X^\pm}\nonumber\\
&&\le\frac{1}{\kappa C_1}\|x_1-x_2\|_{X}
\end{eqnarray}
for some positive numbers $\kappa>1$, $r_1>0$ and all $z\in
B_{H^0}(\theta,r_1)$, $x_i\in B_X(\theta, r_1)\cap X^\pm$, $i=1,2$.
Here $C_1$ is given by (\ref{e:3.2}).
\end{enumerate}

Clearly,  (\ref{e:N1}) and (\ref{e:N2}) are, respectively, implied
in the following  inequalities
\begin{eqnarray}\label{e:N3}
&&\|A(z_1+x_1)-B(\theta)x_1-A(z_2+x_2)+
B(\theta)x_2\|_{X}\nonumber\\
&&\le\frac{1}{\kappa C_1C_2}\|z_1+x_1-z_2-x_2\|_{X}
\end{eqnarray}
for  all $z_i\in B_{H^0}(\theta,r_1)$, $x_i\in B_X(\theta, r_1)\cap
X^\pm$, $i=1,2$, and
\begin{eqnarray}\label{e:N4}
&&\|A(z+x_1)-B(\theta)x_1-A(z+x_2)+
B(\theta)x_2\|_{X}\nonumber\\
&&\le\frac{1}{\kappa C_1C_2}\|x_1-x_2\|_{X}
\end{eqnarray}
for  all $z\in B_{H^0}(\theta,r_1)$, $x_i\in B_X(\theta, r_1)\cap
X^\pm$, $i=1,2$. Here $C_1$ and $C_2$ are given by (\ref{e:3.2}).

\underline{We first consider the case (${\bf F2^{\prime\prime}}$)
holding}. Checking the proof of (\ref{e:3.4}) we have
\begin{eqnarray*}
&&\|S(z, x_1)-S(z, x_2)\|_{X^\pm}\nonumber\\
&&\le C_1\cdot\|(I-P^0)A(z+x_1)- B(\theta)x_1 -(I-P^0)A(z+x_2)+ B(\theta)x_2\|_{X^\pm}\nonumber\\
&&\le\frac{1}{\kappa}\|x_1-x_2\|_X
\end{eqnarray*}
for all $z\in B_{H^0}(\theta,r_1)$ and $x_i\in B_{X^\pm}(\theta,
r_1)$, $i=1,2$. Since $A(x)\to \theta$ as $x\to\theta$ we can choose
$r_0\in (0, r_1)$ such that $\|S(z,0)\|<r_1(1-1/\kappa)$ for any
$z\in B_{H^0}(\theta, r_0)$. By Theorem~10.1.1 in \cite[Chap.10]{Di}
we have a unique  map $h: B_{H^0}(\theta, r_0)\to
\bar{B}_{X^\pm}(\theta,r_0)$ with $h(\theta)=\theta$, which is also
continuous, such that $S(z, h(z))=h(z)$ or equivalently $(I-P^0)A(z+
h(z))=\theta\;\forall z\in B_{H^0}(\theta, r_0)$ as in
(\ref{e:3.5}).

\underline{Next we consider the case (${\bf F2^\prime}$) holding}.
By the proof of (\ref{e:3.4}) we easily see
\begin{eqnarray}\label{e:N5}
&&\|S(z_1, x_1)-S(z_2, x_2)\|_{X^\pm}\nonumber\\
&&\le C_1\cdot\|(I-P^0)A(z_1+x_1)- B(\theta)x_1 -(I-P^0)A(z_2+x_2)+ B(\theta)x_2\|_{X^\pm}\nonumber\\
&&\le\frac{1}{\kappa}\|z_1+ x_1-z_2-x_2\|_X
\end{eqnarray}
and thus $\|S(z, x_1)-S(z,
x_2)\|_{X^\pm}\le\frac{1}{\kappa}\|x_1-x_2\|_X$ if $z_1=z_2=z$.
Since $A(x)\to \theta$ as $x\to\theta$ we can choose $r_0\in (0,
r_1)$ such that
\begin{eqnarray*}
\|S(z, x)\|_{X^\pm}&=&\|S(z,x)- S(z,\theta)\|_{X^\pm}+ \|S(z,\theta)\|\\
&\le&\frac{1}{\kappa}\|x\|_X+ \frac{\kappa-1}{\kappa}r_0
\end{eqnarray*}
for any $z\in \bar{B}_{H^0}(\theta, r_0)$. Hence for each $z\in
\bar{B}_{H^0}(\theta, r_0)$ we may apply the Banach fixed point
theorem to the map
$$
\bar{B}_{X^\pm}(\theta,r_0)\ni x\mapsto S(z,x)\in
\bar{B}_{X^\pm}(\theta,r_0)
$$
to get a unique map $h:\bar{B}_{H^0}(\theta, r_0)\to
\bar{B}_{X^\pm}(\theta,r_0)$ such that $S(z, h(z))=h(z)$. From the
latter and (\ref{e:N5}) it easily follows that
\begin{eqnarray}\label{e:N6}
\|h(z_1)-h(z_2)\|_{X^\pm}\le\frac{1}{\kappa-1}\|z_1-z_2\|_X
\end{eqnarray}
for any $z_i\in \bar{B}_{X^\pm}(\theta,r_0)$, $i=1,2$. That is, $h$
is Lipschitz continuous. Using this we may prove as in Step 2 of the
proof of Lemma~\ref{lem:3.1} that  $\mathcal{ L}^\circ$ has a linear
bounded G\^{a}teaux derivative at each $z_0\in \bar{B}_{H^0}(\theta,
r_0)$ and
$$
D\mathcal{ L}^\circ(z_0)z=(A(z_0+ h(z_0)), z)_H=(P^0A(z_0+
h(z_0)), z)_H\;\forall z\in H^0.
$$
 Moreover, checking the proof of
(\ref{e:3.10}) we have still (\ref{e:3.10}), i.e.,
\begin{eqnarray*}
|D\mathcal{ L}^\circ(z_0)z- D\mathcal{ L}^\circ(z'_0)z|
&\le &\|A(z_0+ h(z_0))- B(\theta)(z_0+ h(z_0))\nonumber\\
&&- A(z'_0+h(z'_0))+ B(\theta)(z'_0+ h(z'_0))\|_X\cdot\|z\|_X
\end{eqnarray*}
for all $z_0\in\bar{B}_{H^0}(\theta, r_0)$ and $z\in H^0$.
Note that $A$ is continuously directional differentiable and hence $C^{1-0}$. It follows from
(\ref{e:N6}) that the map  $\bar{B}_{H^0}(\theta, r_0)\ni z_0\mapsto
D\mathcal{ L}^\circ(z_0)\in L(H^0,\R)$ is $C^{1-0}$.
As before we derive from \cite[Th.2.1.13]{Ber} that $\mathcal{ L}^\circ$ is Fr\'echet differentiable at
$z_0$ and its Fr\'echet differential $d\mathcal{
L}^\circ(z_0)=D\mathcal{ L}^\circ(z_0)$ is $C^{1-0}$ in
$z_0\in B_{H^0}(\theta, r_0)$.

Summarizing the above arguments we obtain

\begin{theorem}\label{th:8.1}
Under the above assumptions {\rm (S)}, {\rm (F1),(${\bf F2^{\prime\prime}}$), (F3)} and {\rm
(C1)-(C2)}, {\rm (D)}, if $\nu>0$ there exist a positive
$\epsilon\in\R$, a (unique)  continuous map
$h:B_{H^0}(\theta,\epsilon)=B_{H}(\theta,\epsilon)\cap H^0\to X^\pm$
satisfying $h(\theta)=\theta$ and (\ref{e:2.3}),
an open neighborhood $W$ of $\theta$ in $H$ and an origin-preserving
homeomorphism
$$
\Phi: B_{H^0}(\theta,\epsilon)\times
\left(B_{H^+}(\theta,\epsilon) +
B_{H^-}(\theta,\epsilon)\right)\to W
$$
of form $\Phi(z, u^++ u^-)=z+ h(z)+\phi_z(u^++ u^-)$ with
$\phi_z(u^++ u^-)\in H^\pm$  such that (\ref{e:2.5}) and (\ref{e:2.6}) are satisfied.
Moreover, the homeomorphism $\Phi$ has also the properties (a) and (b) in Theorem~\ref{th:2.1}.
Furthermore, if (${\bf F2^{\prime\prime}}$) is replaced by the slightly strong (${\bf F2^{\prime}}$)
then the map $h$ is Lipschitz continuous
and  the function $B_{H^0}(\theta,\epsilon)\ni z\mapsto
\mathcal{ L}^\circ(z):=\mathcal{ L}(z+ h(z))$ is $C^{2-0}$ and
 $$
d\mathcal{ L}^\circ(z_0)(z)=(A(z_0+ h(z_0)), z)_H\quad\forall z_0\in
B_{H^0}(\theta, \epsilon),\; z\in H^0.
 $$
Consequently, $\theta$ is  an isolated critical point
of $\mathcal{L}^\circ$ provided that $\theta$ is an isolated critical point of $\mathcal{
L}|_{V^X}$.
\end{theorem}

Carefully checking the arguments in Section~\ref{sec:2} and the proofs in Section~\ref{sec:4}
it is not hard to derive:

\begin{corollary}\label{cor:8.2}
If the above assumptions {\rm (S)}, {\rm (F1),(${\bf F2^{\prime\prime}}$), (F3)} and {\rm
(C1)-(C2)}, {\rm (D)} are satisfied then Corollary~\ref{cor:2.5} also holds. Moreover,
Corollaries~\ref{cor:2.6},~\ref{cor:2.7},~\ref{cor:2.8} and ~\ref{cor:2.9} are true under
 the  assumptions {\rm (S)}, {\rm (F1),(${\bf F2^{\prime}}$), (F3)} and {\rm
(C1)-(C2)}, {\rm (D)}.
\end{corollary}

By Claim~6.1, $\widehat{C}_1:=\|(\widehat B(\theta)|_{\widehat
X^\pm})^{-1}\|_{L(\widehat{X}^\pm)}\ge C_1:=\|(
B(\theta)|_{X^\pm})^{-1}\|_{L(X^\pm)}$ if
$\|Jx\|_{\widehat{X}}=\|x\|_X\;\forall x\in X$. In order to assure
that Theorem~\ref{th:6.1} also holds when Theorem~\ref{th:2.1} with
(F2) is replaced by Theorem~\ref{th:8.1} with (${\bf
F2^{\prime\prime}}$) we should require not only that
$J|_X:X\to\widehat{X}$ is a Banach isometry but also that $C_1$ in
(\ref{e:N2}) for $(A,B)$ is replaced by $\widehat{C}_1$. For
Theorem~\ref{th:6.3} being true after Theorem~\ref{th:2.1} is
replaced by Theorem~\ref{th:8.1} it is suffice to assume that
$J|_X:X\to\widehat{X}$ is a Banach isometry. Theorem~\ref{th:6.4}
also holds if we replace ``Theorem~\ref{th:2.1}''  by
``Theorem~\ref{th:8.1}'' there.

Finally, we have also a corresponding result with
Proposition~\ref{prop:7.1} provided that the sentence ``Under the
assumptions of Theorem~\ref{th:2.1} with ({\bf D4}) replaced by
({\bf D4**}),  suppose that the map $A:V^X\to X$ in the condition
({\bf F2}) is Fr\'echet differentiable.'' in
Proposition~\ref{prop:7.1} is replaced by ``Under the assumptions of
Theorem~\ref{th:8.1} with ({\bf D4}) replaced by ({\bf D4**}),
suppose that the map $A:V^X\to X$ in the condition (${\bf
F2^{\prime\prime}}$) is Fr\'echet differentiable.''

\appendix
\section{ Parameterized version of Morse-Palais
 lemma due to Duc-Hung-Khai}\label{app:A}\setcounter{equation}{0}

 Almost repeating the proof of Theorem 1.1 in
\cite{DHK} one easily gets the following parameterized version of
it (\cite{Lu2}). Actually we give more conclusions, which are key for proofs of
some results in this paper.

\begin{theorem}\label{th:A.1}
Let $(H, \|\cdot\|)$ be a normed vector space and let $\Lambda$ be a
 topological space. Let $J:\Lambda\times B_H(\theta,
2\delta)\to\R$ be continuous, and let the function $J(\lambda,
\cdot): B_H(\theta, 2\delta)\to\R$ be continuously directional
differentiable for every $\lambda\in\Lambda$.
 Assume that there exist a closed
vector subspace $H^+$ and a finite-dimensional vector subspace $H^-$
of $H$ such that $H^+\oplus H^-$ is a direct sum decomposition of
$H$ and
\begin{enumerate}
\item[\bf (i)] $J(\lambda, \theta)=0$ and $D_2J(\lambda, \theta)=0$,
\item[\bf (ii)] $[D_2J(\lambda, x+ y_2)-D_2J(\lambda, x+ y_1)](y_2-y_1)<0$ for any $(\lambda, x)\in\Lambda\times\bar
B_{H^+}(\theta,\delta)$, $y_1, y_2\in\bar
B_{H^-}(\theta,\delta)$ and $y_1\ne y_2$,

\item[\bf (iii)] $D_2J(\lambda, x+y)(x-y)>0$ for any $(\lambda, x, y)\in\Lambda\times\bar B_{H^+}(\theta,
\delta)\times\bar B_{H^-}(\theta,\delta)$ and $(x,y)\ne (\theta,
\theta)$,

\item[\bf (iv)] $D_2J(\lambda, x)x>p(\|x\|)$ for any $(\lambda, x)\in\Lambda\times\bar
B_{H^+}(\theta,\delta)\setminus\{\theta\}$, where $p:(0,
\delta]\to (0, \infty)$ is a non-decreasing function. (One may require that $p(t)\le 4t^2\;\forall
t\in (0,\delta]$.)
\end{enumerate}
Then we have:\\
$1^\circ.$ If $H^-=\{\theta\}$ {\rm (}so
  the condition (ii) is empty and (iv) implies (iii) {\rm )} then
there exists an open neighborhood $U$ of $\Lambda\times\{\theta\}$ in
$\Lambda\times H$ and a homeomorphism
$\phi:\Lambda\times B_H(\theta, \sqrt{p(\delta/2)/2})\to
U$  satisfies
$$
J(\phi(\lambda,x))=\|x\|^2\quad\forall (\lambda,
x)\in\Lambda\times B_H(\theta, \sqrt{p(\delta/2)/2}).
$$
If $H^+=\{\theta\}$ {\rm (}so the conditions (iii) and (iv)
are empty, and (ii) becomes:\\
{\bf (ii')} $[D_2J(\lambda,  y_2)-D_2J(\lambda, y_1)](y_2-y_1)<0$ for any $\lambda\in\Lambda$, $y_1, y_2\in\bar
B_{H}(\theta,\delta)$ and $y_1\ne y_2${\rm )},\\
then there exist two open neighborhoods of $\Lambda\times\{\theta\}$ in
$\Lambda\times H$, $W$ and $V$ with $V\subset\Lambda\times
B_H(\theta,\delta)$, and  a homeomorphism
 $\phi:W\to V$ with $\phi(\lambda,x)=(\lambda,\phi_\lambda(x))$, such that
 $$J(\phi(\lambda,x))=-\|x\|^2\;\forall (\lambda,
x)\in W,
$$
 moreover $W$ can be taken as $\Lambda\times B_H(\theta, \sqrt{p(\delta/2)/2})$
 provided that  {\bf (ii')} is replaced by\\
{\bf (iv')} $D_2J(\lambda, x)x<-p(\|x\|)$ for any $(\lambda, x)\in\Lambda\times\bar
B_{H}(\theta,\delta)\setminus\{\theta\}$, where $p:(0,
\delta]\to (0, \infty)$ is as in (iv).

\noindent{$2^\circ.$} If $\Lambda$ is compact, and $H^+\ne\{\theta\}$ and $H^-\ne\{\theta\}$,
then there exist a positive $\epsilon\in\R$, an open neighborhood
$U$ of $\Lambda\times\{\theta\}$ in $\Lambda\times H$ and a
homeomorphism
$$
\phi: \Lambda\times \bigl(B_{H^+}(\theta, \sqrt{p(\epsilon)/2})+
B_{H^-}(\theta, \sqrt{p(\epsilon)/2})\bigr)\to U
$$
with $\phi(\lambda,x)=(\lambda,\phi_\lambda(x))$, such that
$$
J(\phi(\lambda, x+ y))=\|x\|^2-\|y\|^2\quad\hbox{and}\quad
\phi(\lambda, x+ y)=(\lambda, \phi_\lambda(x+y))\in\Lambda\times H
$$
for all $(\lambda, x,y)\in \Lambda\times B_{H^+}(\theta,
\sqrt{p(\epsilon)/2})\times B_{H^-}(\theta,
\sqrt{p(\epsilon)/2})$. Moreover, for each $\lambda\in\Lambda$,
$\phi_\lambda(0)=0$, $\phi_\lambda(x+y)\in H^-$ if and only if
$x=0$,  and $\phi$ is a homoeomorphism from $\Lambda\times
B_{H^-}(\theta, \sqrt{p(\epsilon)/2})$ onto $U\cap (\Lambda\times
H^-)$ according to the topologies on both induced by any norms on
$H^-$.
\end{theorem}

The claim in ``Moreover'' part was not stated in \cite{DHK}, and can
be seen from the proof therein. It precisely means: for any two
norms $\|\cdot\|_1$ and $\|\cdot\|_2$  on $H^-$, if $\Lambda\times
B_{H^-}(\theta, \sqrt{p(\epsilon)/2})$ (resp.  $U\cap
(\Lambda\times H^-)$) is equipped with the topology induced by
$\Lambda\times(H^-,\|\cdot\|_1)$ (resp.
$\Lambda\times(H^-,\|\cdot\|_2)$)
 then $\phi$ is also a homoeomorphism from
$\Lambda\times B_{H^-}(\theta, \sqrt{p(\epsilon)/2})$ onto $U\cap
(\Lambda\times H^-)$. This leads to the proof of Theorem~\ref{th:2.1}(b), which
is a key for the proofs of Corollary~\ref{cor:2.5} and Theorem~\ref{th:2.10}.
 So it is helpful for readers to outline the proof of
Theorem~\ref{th:A.1}. \

\begin{proof}[Sketches of proof of Theorem~\ref{th:A.1}]
$1^\circ)$ {\bf Case} $H^-=\{\theta\}$ or  $H^+=\{\theta\}$. This is actually contained
in the proof of \cite{DHK}.

\underline{We first consider the case $H^-=\{\theta\}$}.  Define
$$
\psi(\lambda, x)=\left\{\begin{array}{ll}
 \frac{\sqrt{J(\lambda, x)}}{\|x\|}x &\;\hbox{if}\;x\in \bar B_H(\theta, \delta)\setminus\{\theta\},\\
 \theta&\;\hbox{if}\;x=\theta.
 \end{array}\right.
$$
Then it is continuous and $J(\lambda, x)=\|\psi(\lambda,x)\|^2$. It
easily follows from the condition (iv) that for each
$\lambda\in\Lambda$ the map $\psi(\lambda,\cdot)$ is one-to-one on
$\bar B_H(\theta, \delta)$. Moreover, for any $x\in\partial
B_H(\theta,\delta)$, as in \cite[(2.9)]{DHK} we have $s_x\in (1/2,
1)$ such that
\begin{eqnarray*}
&&J(\lambda, x)>J(\lambda, x)- J(\lambda, x/2)=D_2J(\lambda,
s_xx)(x/2)\\
&&=\frac{1}{2s_x}D_2J(\lambda,
s_xx)(s_xx)>\frac{1}{2}p(\|s_xx\|)\ge\frac{1}{2}p(\|x/2\|)=\frac{1}{2}p(\delta/2)
\end{eqnarray*}
by the condition (iv). Hence
$\|\psi(\lambda,x)\|>\sqrt{p(\delta/2)/2}$. For any
$0<\|y\|<\sqrt{p(\delta/2)/2}$, without loss of generality we assume
$\delta>\sqrt{p(\delta/2)/2}$. (This can be assured by replacing the
function $p(t)$ in Theorem~\ref{th:A.1}(iv) with $\min\{p(t),
4t^2\}$). Then we have a unique positive number $r>1$ such that
$x:=ry\in\partial B_H(\theta,\delta)$. Since the function
$$
[0, 1]\to\R,\; t\mapsto\sqrt{J(\lambda, tx)}
$$
is continuous there exists  a $t_0\in (0, 1)$ such that $\|y\|=\sqrt{J(\lambda, t_0x)}$
and hence
$$
\psi(\lambda, t_0x)=\sqrt{J(\lambda,
t_0x)}\frac{t_0x}{\|t_0x\|}=\|y\|\frac{y}{\|y\|}=y.
$$
This shows that $B_H\bigl(\theta,
\sqrt{p(\delta/2)/2}\bigr)\subset\psi\bigl(\{\lambda\}\times
B_H(\theta, \delta)\bigr)$. Let
$$
U=\left\{(\lambda, z)\in\Lambda\times
B_H(\theta,\delta)\,\Bigm|\,\psi(\lambda,z)\in B_H\bigl(\theta,
\sqrt{p(\delta/2)/2}\bigr)\right\}.
$$
It is an open neighborhood of $\Lambda\times\{\theta\}$ in
$\Lambda\times H$. Define
$$
\phi:\Lambda\times B_H(\theta, \sqrt{p(\delta/2)/2})\to
U,\;(\lambda, x)\mapsto (\lambda, y),
$$
where $y\in B_H(\theta,\delta)$ is the unique point such that
$\psi(\lambda, y)=x$. As in the proof of Lemma 2.7 of \cite{DHK} it is
easily showed that $\phi$ is continuous and satisfies
$$
J(\phi(\lambda,x))=\|x\|^2\quad\forall (\lambda,
x)\in\Lambda\times B_H(\theta, \sqrt{p(\delta/2)/2}).
$$

\underline{Next we assume  $H^+=\{\theta\}$}. Then the conditions (iii) and (iv)
are empty, and the condition (ii) becomes:
$[D_2J(\lambda,  y_2)-D_2J(\lambda, y_1)](y_2-y_1)<0$ for any $\lambda\in\Lambda$, $y_1, y_2\in\bar
B_{H^-}(\theta,\delta)=\bar
B_{H}(\theta,\delta)$ and $y_1\ne y_2$.
This implies that $0=J(\lambda, \theta)>J(\lambda, y)\;\forall
y\in B_{H}(\theta,\delta)\setminus\{\theta\}$
for each $\lambda\in\Lambda$.
 Define
$$
\psi(\lambda, x)=\left\{\begin{array}{ll}
 \frac{\sqrt{-J(\lambda, x)}}{\|x\|}x &\;\hbox{if}\;x\in \bar B_H(\theta, \delta)\setminus\{\theta\},\\
 \theta&\;\hbox{if}\;x=\theta.
 \end{array}\right.
$$
Clearly, it is continuous and $J(\lambda, x)=-\|\psi(\lambda,x)\|^2$. It
easily follows from the condition (ii) that for each
$\lambda\in\Lambda$ the map $\psi(\lambda,\cdot)$ is one-to-one on
$\bar B_H(\theta, \delta)$. Moreover,  for any fixed $x\in\partial
B_H(\theta,\delta)$ and $t\in (0, 1]$ we have
\begin{eqnarray*}
\frac{d}{dt}J(\lambda, tx)&=&-D_2J(\lambda,tx)(x)\\
&=&\frac{1}{t}\bigl[D_2J(\lambda, tx)(tx)-D_2J(\lambda,\theta)(\theta)\bigr)<0
\end{eqnarray*}
by the condition (ii). Hence $J(\lambda, x)=\min\{J(\lambda,tx)\,|\, 0\le t\le 1\}\;\forall
x\in\partial B_H(\theta,\delta)$. Since $\bar B_H(\theta,\delta)$ is compact we have $\rho_\lambda>0$
such that
$$
-\rho^2_\lambda=\min\{J(\lambda, x)\,|\, x\in \bar B_H(\theta,\delta)\}=
\min\{J(\lambda, x)\,|\, x\in \partial B_H(\theta,\delta)\}.
$$
It follows that $B_H\bigl(\theta,
\rho_\lambda\bigr)\subset\psi\bigl(\{\lambda\}\times
B_H(\theta, \delta)\bigr)$. Since $W=\{(\lambda,x)\,|\, x\in B_H\bigl(\theta,
\rho_\lambda\bigr)\}$ is an open neighborhood of $\Lambda\times\{\theta\}$ in
$\Lambda\times H$, so is
$$
V=\bigl\{(\lambda, z)\in\Lambda\times
B_H(\theta,\delta)\,\bigm|\,\psi(\lambda,z)\in W \bigr\}.
$$
 Define $\phi:W\to V,\;(\lambda, x)\mapsto (\lambda, y)$,
where $y\in B_H(\theta,\delta)$ is the unique point such that
$\psi(\lambda, y)=x$. Then
$J(\lambda, \phi(\lambda,x))=-\|x\|^2\;\forall (\lambda,
x)\in W$. We claim that $\phi$ is continuous. In fact, suppose that
$\{(\lambda_n, x_n)\}\subset W$ converges to $(\lambda_0,x_0)\in W$.
Let $\phi(\lambda_0, x_0)=(\lambda_0, y_0)$ and
$\phi(\lambda_n, x_n)=(\lambda_n, y_n)\;\forall n\in\N$. Then
$\psi(\lambda_0, y_0)=x_0$ and
$\psi(\lambda_n, y_n)=x_n\;\forall n\in\N$.
We can always assume $x_n\ne\theta\;\forall n$.
Then $y_n\ne\theta\;\forall n$. It follows that $J(\lambda_n, y_n)=-\|x_n\|^2\to -\|x_0\|^2=J(\lambda_0, y_0)$.
Since $\{y_n\}\subset B_H(\theta,\delta)$, we may assume $y_n\to y^\ast\in\bar{B}_H(\theta,\delta)$ (by passing a subsequence if necessary) because of the compactness of $\bar{B}_H(\theta,\delta)$.
We want to prove $y^\ast=y_0$. Since $\psi$ is continuous we get
$\psi(\lambda_0, y^\ast)=x_0=\psi(\lambda_0,y_0)$ and thus
$y^\ast=y_0$ by the fact that
the map $\psi(\lambda_0,\cdot)$ is one-to-one on
$\bar B_H(\theta, \delta)$.

If the condition (ii') is replaced by  (iv'), the arguments are obvious.

 \noindent{$2^\circ)$} {\bf Case} $H^\ast\ne\{\theta\}$ ($\ast=+,-$) and $\Lambda$ is compact.
 Since the parameter $\lambda$ appears many notations in \cite{DHK}
have corresponding changes.

\noindent{\bf Step 1}(\cite[Lemma 2.1]{DHK}). There exists a
positive real number $\epsilon_1<\delta$ having the following
property: For each $(\lambda, x)\in\Lambda\times B_{H^+}(\theta,
\epsilon_1)$ there exists a unique $\varphi_\lambda(x)\in
B_{H^-}(\theta,\delta)$ such that
$$
J(\lambda, x+\varphi_\lambda(x))=\max\{J(\lambda, x+y)\,|\, y\in
B_{H^-}(\theta,\delta)\}.
$$
See the proof of Claim~\ref{cl:A.3} in the proof of Theorem~\ref{th:A.2}. ({\it
Note}: The compactness of $\Lambda$ is necessary in proving this
claim.)

Remarks that  $J(\lambda, x+ \varphi_\lambda(x))>0$ for any $x\in
B_{H^+}(\theta,\delta)\setminus\{\theta\}$ by
Theorem~\ref{th:A.1}(iv) and the mean value theorem. Moreover, the
uniqueness of $\varphi_\lambda(x)$ implies
$$
J(\lambda, x+ \varphi_\lambda(x))> J(\lambda, x+ y)
$$
for all $x\in B_{H^+}(\theta,\epsilon_1)$ and $y\in
B_{H^-}(\theta, \delta)\setminus\{\varphi_\lambda(x)\}$.

By replacing $\delta$ by $\delta/2$ in the arguments above we can
assume $\varphi_\lambda(x)\in B_{H^-}(\theta, \delta/2)$ for any
$x\in B_{H^+}(\theta,\epsilon_1)$ below.

\noindent{\bf Step 2}(\cite[Lemma 2.2]{DHK}).  The map
$\Lambda\times B_{H^+}(\theta, \epsilon_1):(\lambda,
x)\mapsto\varphi_\lambda(x)$ is continuous.

In fact, suppose that the sequence $\{(\lambda_n,
x_n)\}\subset\Lambda\times B_{H^+}(0,\epsilon_1)$ converges to
$(\lambda_0, x_0)\in\Lambda\times B_{H^+}(0,\epsilon_1)$. Since
$\bar B_{H^-}(0,\delta/2)$ is compact, we can assume that
$\{\varphi_{\lambda_n}(x_n)\}$ converges to $y_0\in\bar
B_{H^-}(0,\delta/2)$. Then
$$
J(\lambda_n, x_n+\varphi_{\lambda_n}(x_n))\ge J(\lambda_n,
x_n+y)\quad\forall y\in B_{H^-}(0,\delta)\;\hbox{and}\;n\in\N.
$$
This implies that $J(\lambda_0, x_0+ y_0)\ge J(\lambda_0, x_0+ y)$
for any $y\in B_{H^-}(0,\delta)$. By the uniqueness of
$\varphi_{\lambda_0}(x_0)$ we get $y_0=\varphi_{\lambda_0}(x_0)$.

\noindent{\bf Step 3}(\cite[Lemma 2.3]{DHK}). Put $j(\lambda,
x)=J(\lambda, x+ \varphi_\lambda(x))$ for any $(\lambda,
x)\in\Lambda\times B_{H^+}(\theta, \epsilon_1)$. Then $j$ is
continuous and for each  $\lambda\in\Lambda$ the map $x\mapsto
j(\lambda, x)$ is continuously directional differentiable.

\noindent{\bf Step 4}(\cite[Lemma 2.4]{DHK}). Define
\begin{eqnarray*}
&&\psi_1(\lambda, x+y)=\left\{\begin{array}{ll}
 \frac{\sqrt{J(\lambda, x+ \varphi_\lambda(x))}}{\|x\|}x &\;\hbox{if}\;x\ne \theta,\\
 \theta&\;\hbox{if}\;x=\theta,
 \end{array}\right.\\
&&\psi_2(\lambda, x+y)=\left\{\begin{array}{ll}
 \frac{\sqrt{J(\lambda, x+ \varphi_\lambda(x))-J(\lambda, x+y)}}{\|y-\varphi_\lambda(x)\|}(y-\varphi_\lambda(x))
  &\;\hbox{if}\;y\ne\varphi_\lambda(x),\\
 \theta&\;\hbox{if}\;y=\varphi_\lambda(x),
 \end{array}\right.\\
&&\psi(\lambda, x+y)=\psi_1(\lambda, x+y)+\psi_2(\lambda,
x+y)\\
&&\hspace{20mm}\forall (x,y)\in B_{H^+}(\theta, \epsilon_1)\times
B_{H^-}(\theta,\delta).
\end{eqnarray*}
Then $\psi_1$, $\psi_2$ and $\psi$ are continuous on
 $\Lambda\times (B_{H^+}(\theta, \epsilon_1)+
B_{H^-}(\theta,\delta))$ and
\begin{equation}\label{e:A.1}
J(\lambda, x+y)=\|\psi_1(\lambda, x+y)\|^2-\|\psi_2(\lambda,
x+y)\|^2
\end{equation}
for any $(\lambda, x,y)\in \Lambda\times B_{H^+}(\theta,
\epsilon_1)\times B_{H^-}(\theta,\delta)$. Moreover,
$\psi(\lambda, x+y)\in {\rm Im}(\psi)\cap H^-$ if and only if
$x=\theta$.

\noindent{\bf Step 5}(\cite[Lemma 2.5]{DHK}). For each
$\lambda\in\Lambda$ the map
$$
\psi(\lambda,\cdot): B_{H^+}(\theta,\epsilon_1)+
B_{H^-}(\theta,\delta)\to H^\pm
$$
is injective.

\noindent{\bf Step 6}(\cite[Lemma 2.6]{DHK}). There is a positive
real number $\epsilon<\epsilon_1$ such that
$$
B_{H^+}\bigl(\theta, \sqrt{p(\epsilon)/2}\bigr)+
B_{H^-}\bigl(\theta,
\sqrt{p(\epsilon)/2}\bigr)\subset\psi\bigl(\lambda,
B_{H^+}(\theta, 2\epsilon)+ B_{H^-}(\theta,\delta)\bigr)
$$
for any $\lambda\in\Lambda$.

We here give a detailed proof of it because the compactness of
$\Lambda$ is very key in the following proof. They are helpful for
understanding the proof of the noncompact case in Section 4 of
\cite{Lu2}.

 For each $(\lambda, y)\in\Lambda\times \bar
B_{H^-}(0, \delta)$ with $y\ne 0$, the mean value theorem yields
$\bar t\in (0,1)$ such that
$$
J(\lambda, y)=J(\lambda, y)-J(\lambda, 0)=D_2J(\lambda, \bar t\cdot
y)y=\frac{-1}{\bar t}D_2J(\lambda, \bar t\cdot y)(-\bar t\cdot y)<0
$$
because of the condition (iii) in Theorem~\ref{th:A.1}. So the
compactness of $\Lambda\times\partial B_{H^-}(0,\delta)$ implies
that there exists a positive real number $C$ such that
\begin{equation}\label{e:A.2}
J(\lambda, y)<-C\quad\forall (\lambda, y)\in\Lambda\times\partial
B_{H^-}(0,\delta).
\end{equation}
We shall prove that there exists a positive real number
$\epsilon<\epsilon_1/4$ such that
\begin{equation}\label{e:A.3}
J(\lambda, x+ y)\le 0\quad\forall (\lambda, x, y)\in\Lambda\times
\bar B_{H^+}(0, 2\epsilon)\times\partial B_{H^-}(0,\delta).
\end{equation}
Assume by contradiction that there exists a sequence
$$
\{(\lambda_n, x_n, y_n)\}\subset\Lambda\times \bar B_{H^+}(0,
\epsilon_1)\times\partial B_{H^-}(0,\delta)
$$
such that $(\lambda_n, x_n, y_n)\to (\lambda_0, \theta,  y_0)\in\Lambda\times
\bar B_{H^+}(0, \epsilon_1)\times\partial B_{H^-}(0,\delta)$ and
$J(\lambda_n, x_n+ y_n)\ge 0\;\forall n$. Then the continuity of $J$
implies $J(\lambda_0, y_0)\ge 0$. This contradicts to (\ref{e:A.2}).
Hence (\ref{e:A.3}) holds.

Since $\varphi_\lambda(0)=0\;\forall\lambda\in\Lambda$, by Step 2 we
may shrink $\epsilon$ in (\ref{e:A.3}) such that
\begin{equation}\label{e:A.4}
\varphi_\lambda(\bar B_{H^+}(0, 2\epsilon))\subset B_{H^-}(0,
\delta/2)\quad\forall\lambda\in\Lambda.
\end{equation}
Fixing $(\lambda, x)\in\Lambda\times\bar B_{H^+}(0,
2\epsilon)\setminus\{0\}$ we can use the mean value theorem and the
condition (iv) in Theorem~\ref{th:A.1} to get $s_x\in (1/2, 1)$ such
that
\begin{eqnarray}\label{e:A.5}
J(\lambda, x+ \varphi_\lambda(x))&\ge& J(\lambda, x)>J(\lambda,
x)-J(\lambda, x/2)\nonumber\\
&=&D_2J(\lambda, s_xx)(x/2)\nonumber\\
&=&\frac{1}{2s_x}D_2J(\lambda,
s_xx)(s_xx)\nonumber\\
&>&\frac{1}{2}p(\|s_xx\|)\ge\frac{1}{2}p(\|x/2\|)
\end{eqnarray}
This and (\ref{e:A.3}) imply that for any $(\lambda, x, y)\in\Lambda\times\partial B_{H^+}(0,
2\epsilon)\times\partial B_{H^-}(0, \delta)$,
\begin{eqnarray}\label{e:A.6}
J(\lambda, x+\varphi_\lambda(x))-J(\lambda, x+y)&\ge &J(\lambda,
x+\varphi_\lambda(x))\nonumber\\
&>&\frac{1}{2}p(\|x/2\|)=\frac{p(\epsilon)}{2}.
\end{eqnarray}

Now for $x\in \partial B_{H^+}(0, 2\epsilon)$ and $0\le
t\le\sqrt{p(\epsilon)/2}$, by (\ref{e:A.5}) we have
$$
\sqrt{J(\lambda, x+ \varphi_\lambda(x))}>\sqrt{p(\epsilon)/2}\ge
t\ge 0.
$$
 Since the map $[0, 1]\to\R, s\mapsto J(\lambda, sx+
\varphi_\lambda(sx))$, is continuous we may obtain a $\bar s\in [0, 1)$ such that
$\sqrt{J(\lambda, \bar sx+ \varphi_\lambda(\bar sx))}=t$. Clearly,
$\bar s>0$ if and only $t>0$. If $t>0$, by the definition of
$\psi_1$ we get
$$
\psi_1(\lambda, \bar sx+ y)=\frac{t}{\|x\|}x=\frac{t}{\|\bar
sx\|}\bar sx\quad\forall y\in B_{H^-}(0,\delta).
$$
When $t=0$, $\psi_1(\lambda, 0)=0$. So for any $x\in\partial
B_{H^+}(0,2\epsilon)$ we have always
$$
\left\{\frac{t}{\|x\|}x\,\bigm|\, 0\le
t\le\sqrt{p(\epsilon)/2}\right\}\subset\psi_1\bigl(\lambda,
B_{H^+}(0, 2\epsilon)\bigr),
$$
 that is,
\begin{equation}\label{e:A.7}
\bar B_{H^+}(0,\sqrt{p(\epsilon)/2})\subset \psi_1\bigl(\lambda,
B_{H^+}(0, 2\epsilon)\bigr)\quad\forall\lambda\in\Lambda.
\end{equation}

For a given $(x^\ast,  y^\ast)\in \bar
B_{H^+}(0,\sqrt{p(\epsilon)/2})\times\bar
B_{H^-}(0,\sqrt{p(\epsilon)/2})$, we may assume $x^\ast\ne\theta$
and $y^\ast\ne\theta$, by (\ref{e:A.7}) we have $x_\lambda\in
B_{H^+}(0, 2\epsilon)\setminus\{\theta\}$ such that
\begin{equation}\label{e:A.8}
\psi_1(\lambda, x_\lambda+ y)=x^\ast\quad\forall y\in
B_{H^-}(\theta, \delta).
\end{equation}

Let us write $y^\ast=\bar t z/\|z\|$, where $z\in\partial B_{H^-}(0,
\delta/2)$ and $0< \bar t\le\sqrt{p(\epsilon)/2}$.
 Since
$\varphi_\lambda(x_\lambda)\in B_{H^-}(0, \delta/2)$ by
(\ref{e:A.4}),  and $\varphi_\lambda(x_\lambda)\ne\theta$, we
have always a real number $k$ with $|k|>1$ such that
$$
y:=kz+ \varphi_\lambda(x_\lambda)\in \partial B_{H^-}(0,\delta)
$$
 (because $|k\cdot
z|=|y-\varphi_\lambda(x)|\ge |y|-|\varphi_\lambda(x)|>\delta/2$). By
(\ref{e:A.6}) the continuous  map
$$
[0, 1]\mapsto\R, \;s\mapsto J(\lambda,
x_\lambda+\varphi_\lambda(x_\lambda)) -J(\lambda, x+
(1-s)\varphi_\lambda(x_\lambda)+ sy)
$$
takes a value $J(\lambda, x_\lambda+\varphi_\lambda(x_\lambda))
-J(\lambda, y)>p(\epsilon)/2$ at $s=1$, and zero at $s=0$. So we
have $\hat s\in (0, 1)$ such that
$$
\sqrt{J(\lambda, x_\lambda+\varphi_\lambda(x_\lambda)) -J(\lambda,
x+ (1-\hat s)\varphi_\lambda(x_\lambda)+ \hat sy)}=\bar t.
$$
Set
\begin{eqnarray*}
y_\lambda:=(1-\hat s)\varphi_\lambda(x_\lambda)+ \hat sy&=&(1-\hat
s)\varphi_\lambda(x_\lambda)+ \hat sk\cdot z+ \hat
s\varphi_\lambda(x_\lambda)\\
&=&\varphi_\lambda(x_\lambda)+ \hat sk\cdot z.
\end{eqnarray*}
 Then
$$
\|y_\lambda\|=\|(1-\hat s)\varphi_\lambda(x_\lambda)+ \hat sy\|\le
(1-\hat s)\|\varphi_\lambda(x_\lambda)\|+ \hat s\delta<(1-\hat
s)\delta/2 + \hat s\delta<\delta,
$$
and the definition of $\psi_2$ shows that
$$
\psi_2(\lambda, x_\lambda + y_\lambda)=\frac{\bar
t}{\|y_\lambda-\varphi_\lambda(x_\lambda)\|}(y_\lambda-\varphi_\lambda(x_\lambda))=\frac{\bar
t}{\|z\|}z=y^\ast.
$$
 This and (\ref{e:A.8}) show that $\psi(\lambda, x_\lambda+
y_\lambda)=(x^\ast, y^\ast)$. The desired result is proved. $\Box$

\noindent{\bf Step 7}(\cite[Lemma 2.7]{DHK}).  Put
$$
U=[\Lambda\times (B_{H^+}(\theta, 2\epsilon)+
B_{H^-}(\theta,\delta))]\cap \psi^{-1}\left(B_{H^+}(\theta,
\sqrt{p(\epsilon)/2})+ B_{H^-}(\theta,
\sqrt{p(\epsilon)/2})\right)
$$
and
\begin{eqnarray}\label{e:A.9}
&&\phi:\Lambda\times\left( B_{H^+}(\theta, \sqrt{p(\epsilon)/2})+
B_{H^-}(\theta, \sqrt{p(\epsilon)/2})\right)\to U,\\
&&\hspace{20mm} (\lambda,x+y)\mapsto (\lambda,
\phi_\lambda(x+y)):=(\lambda, x'+ y'),\nonumber
\end{eqnarray}
where $(x', y')\in B_{H^+}(\theta, 2\epsilon)\times
B_{H^-}(\theta,\delta)$ is a unique point satisfying
$x+y=\psi(\lambda, x'+y')$. Then $\phi$ is continuous and
$$
J(\phi(\lambda, x+ y))=\|x\|^2-\|y\|^2
$$
for any $(\lambda, x, y)\in\Lambda\times B_{H^+}(\theta,
\sqrt{p(\epsilon)/2})\times B_{H^-}(\theta,
\sqrt{p(\epsilon)/2})$. Moreover, $\phi(\lambda, x+ y)\in{\rm
Im}(\psi)\cap(\Lambda\times H^-)$ if and only if $x=\theta$.

\noindent{\bf Step 8.} We shall prove the claims in ``Moreover'' part
of Theorem~\ref{th:A.1}.
 It suffices to check Steps 4, 7.
By Step 1,  for each $(\lambda, x)\in \Lambda\times
B_{H^+}(\theta, \epsilon_1)$, $\varphi_\lambda(x)\in
B_{H^-}(\theta,\delta)$ is a unique maximum point of the function
$B_{H^-}(\theta,\delta)\to \R,\;y\mapsto J(\lambda, x+ y)$. For
any $y\in B_{H^-}(\theta,\delta)$ with $y\ne \theta$, it follows
from the condition (ii) and the mean value theorem that
$$
J(\lambda, y)=J(\lambda, y)-J(\lambda, \theta)=D_2J(\lambda,
ty)(y)=\frac{1}{t}D_2J(\lambda, ty)(ty)<0
$$
for some $t\in (0, 1)$. Hence $\varphi_\lambda(\theta)=\theta$.
For any $x\in B_{H^+}(\theta, \epsilon_1)$ with $x\ne \theta$,
by the condition (iv) and the similar reason we get a $t\in (0, 1)$
such that
$$
J(\lambda, x+\varphi(x))\ge J(\lambda, x)-J(\lambda,
\theta)=D_2J(\lambda, tx)(x)>p(\|tx\|)/t>0.
$$
This implies that $\psi_1(\lambda, x+y)\ne \theta$ if $x\ne
\theta$. When $\psi(\lambda, x+y)\in H^-$, $\psi_1(\lambda,
x+y)=\theta$ and thus $x=\theta$. Conversely, if $x=\theta$ then
$\psi_1(y)=\theta$ and
$$
\psi(\lambda, y)=\theta+ \psi_2(\lambda,
y)=\left\{\begin{array}{ll}
 \frac{\sqrt{-J(\lambda, y)}}{\|y\|}y
  &\;\hbox{if}\;y\ne \theta,\\
 \theta&\;\hbox{if}\;y=\theta.
 \end{array}\right.
 $$
Hence we get that $\psi(\lambda, x+y)\in H^-$ if and only if
$x=\theta$. By the definition of $\phi$ in (\ref{e:A.9}), it is
easy to see that
 $\phi(\lambda, x+y)$  sits in $U\cap (\Lambda\times H^-)$ if and only if $x=\theta$.

 As
to the final claim, since $\dim H^-<\infty$ implies that any norm
$\|\cdot\|^\ast$ on $H^-$ is equivalent to the original $\|\cdot\|$,
$\Lambda\times (H^-, \|\cdot\|^\ast)$ and $\Lambda\times (H^-,
\|\cdot\|)$ induce equivalent topologies on each of the sets
$\Lambda\times B_{H^-}(\theta, \sqrt{p(\epsilon)/2})$ and
$U\cap(\Lambda\times H^-)$. The claim follows.
\end{proof}

In order to give the corresponding version at critical submanifolds
we need  a more general result than Theorem~\ref{th:A.1}. For future conveniences
we here present it because many arguments and notations can be saved.
 Let
$\Lambda$ and $\mathcal{ E}$ be two topological spaces. Imitating
\cite[\S1 of Chap.III]{La} one can naturally define a  {\it
topological normed vector bundle} over $\Lambda$  to be a triple
$(\mathcal{ E}, \Lambda, p)$, where $p:\mathcal{ E}\to \Lambda$ is a
continuous surjection (projection).  In particular we have the
notions of a {\it topological Banach} (resp. {\it Hilbert}) {\it
vector bundle}. Corresponding to Definition 3.1 in Chapter 2 of
\cite{Hu}, a {\it bundle morphism} from the normed vector bundles
$p_1:\mathcal{ E}^{(1)}\to\Lambda_1$ to $p_2:\mathcal{
E}^{(2)}\to\Lambda_2$ is a pair of continuous maps $(\tilde f, f)$,
where $\tilde f:\mathcal{ E}^{(1)}\to \mathcal{ E}^{(2)}$ and
$f:\Lambda_1\to\Lambda_2$ such that $p_2\circ\tilde f=f\circ p_1$.
As on the pages 43-44 of \cite{La} we may define the notion of  a
{\it normed vector bundle morphism}. If
$\Lambda_1=\Lambda_2=\Lambda$ and $f=id_{\Lambda}$ we get the
notions of a $\Lambda$-{\it bundle morphism} and a $\Lambda$-{\it
normed vector bundle morphism}. When $f$ and $\tilde f$ are
homeomorphisms onto $\Lambda_2$ and $\mathcal{ E}^{(2)}$ the
corresponding bundle morphism and normed vector bundle morphism
$(\tilde f, f)$ are called {\it bundle isomorphism} and {\it normed
vector bundle isomorphism} from $\mathcal{ E}^{(1)}$ onto $\mathcal{
E}^{(2)}$. See \cite{La} for more notions such as  subbundles and so
on. As in \cite[Def.2.2, page 15]{Ch} we can define a {\it Finsler
structure} on the bundle $p:\mathcal{ E}\to \Lambda$, and show the
existence of such a structure on the vector bundle if $\Lambda$ is
paracompact.

Let $G$ be a topological group.  For a normed vector bundle
$p:\mathcal{ E}\to\Lambda$, let both $\mathcal{ E}$ and $\Lambda$ be
also $G$-spaces and let $p$ be a $G$-map (or $G$-equivariant map),
we call it a $G$-{\it normed vector bundle} if for all $g\in G$ the
action of $g: \mathcal{ E}_\lambda\to \mathcal{ E}_{g\lambda}$ is a
vector space isomorphism.

\begin{theorem}\label{th:A.2}
Let $\Lambda$ be a  topological space, and let $p:\mathcal{
E}\to\Lambda$ be a topological normed vector bundle with a Finsler
structure $\|\cdot\|:\mathcal{ E}\to [0, \infty)$. Suppose that
$\mathcal{ E}$ can be split into a direct sum of two topological
normed vector subbundles, $\mathcal{ E}=\mathcal{
E}^+\oplus\mathcal{ E}^-$, where $p_-:\mathcal{ E}^-\to\Lambda$ has
finite rank. For $\delta>0$ let $B_{\delta}(\mathcal{
E})=\{(\lambda, v)\in\mathcal{ E}_\lambda\,|\,\|v\|_{\lambda}:=
\|(\lambda, v)\|<\delta\}$. Assume that $J: B_{2\delta}(\mathcal{
E})\to\R$ is continuous and that the restriction of it to each fiber
$$
J_\lambda: B_{2\delta}(\mathcal{ E})_\lambda=\{v\in\mathcal{
E}_\lambda\,|\,\|v\|_{\lambda}<2\delta\},\;v\mapsto J(\lambda, v)
$$
is continuously directional differentiable. Furthermore assume:
\begin{enumerate}
\item[\bf (i)] $J_\lambda(\theta_\lambda)=0$ and $DJ_\lambda(\theta_\lambda)=0$,
\item[\bf (ii)] $[DJ_\lambda(x+ y_2)-DJ_\lambda(x+ y_1)](y_2-y_1)>0$ for any $(\lambda, x)\in\bar
B_{\delta}(\mathcal{ E}^+)$, $(\lambda,y_1), (\lambda,y_2)\in\bar B_{\delta}(\mathcal{
E}^-)$ and $y_1\ne y_2$,

\item[\bf (iii)] $DJ_\lambda(x+y)(x-y)>0$ for any $(\lambda, x)\in\bar B_{\delta}(\mathcal{ E}^+)$
and $(\lambda,y)\in\bar B_{\delta}(\mathcal{ E}^+)$ with $x+y\ne
\theta_\lambda$,

\item[\bf (iv)] $DJ_\lambda(x)x>p(\|x\|_\lambda)$ for any $(\lambda, x)\in\bar
B_{\delta}(\mathcal{ E}^+)$ with $x\ne\theta_\lambda$, where $p:(0,
\delta]\to (0, \infty)$ is a non-decreasing function independent of $\lambda\in\Lambda$.
\end{enumerate}
Then we have:\\
$1^\circ.$ If ${\rm rank}{\mathcal E}^-=0$ {\rm (}so
  the condition (ii) is empty and (iv) implies (iii) {\rm )} then there exist
  an open neighborhood $U$ of the zero section $0_\mathcal{ E}$ of $\mathcal{ E}$  and a
preserving-fiber homeomorphism
$\phi: B_{\sqrt{p(\epsilon)/2})}(\mathcal{ E})\to U$
such that for all $(\lambda, x)\in  B_{\sqrt{p(\epsilon)/2})}(\mathcal{E})$,
$$
J(\phi(\lambda, x))=\|x\|^2_{\lambda}\quad\hbox{and}\quad
\phi(\lambda, x)=(\lambda, \phi_\lambda(x))\in\mathcal{E}.
$$
If ${\rm rank}{\mathcal E}^+=0$ {\rm (}so the conditions (iii) and (iv)
are empty, and (ii) becomes:\\
{\bf (ii')} $[DJ_\lambda(y_2)-DJ_\lambda(y_1)](y_2-y_1)>0$ for any  $(\lambda,y_1), (\lambda,y_2)\in\bar B_{\delta}(\mathcal{
E}^-)$ and $y_1\ne y_2${\rm )},
then there exist open neighborhoods of the zero section $0_\mathcal{ E}\subset\mathcal{ E}$, $W$ and $V$ with $V\subset B_\delta({\mathcal E})$, and  a preserving-fiber homeomorphism
 $\phi:W\to V$ such that
 $$J(\phi(\lambda,x))=-\|x\|^2\;\forall (\lambda,
x)\in W,
$$
 moreover $W$ can be taken as $B_{\sqrt{p(\epsilon)/2})}(\mathcal{ E})$
 provided that  {\bf (ii')} is replaced by\\
{\bf (iv')} $DJ_\lambda(x)x<-p(\|x\|_\lambda)$ for any $(\lambda, x)\in\bar
B_{\delta}(\mathcal{ E}^+)$ with $x\ne\theta_\lambda$, where $p:(0,
\delta]\to (0, \infty)$ is a non-decreasing function independent of $\lambda\in\Lambda$.

\noindent{$2^\circ.$} If $\Lambda$ is compact, and ${\rm rank}{\mathcal E}^+>0$ and ${\rm rank}{\mathcal E}^->0$, then there exist a positive $\epsilon\in\R$, an open neighborhood
$U$ of the zero section $0_\mathcal{ E}$ of $\mathcal{ E}$  and a
preserving-fiber homeomorphism
$$
\phi: B_{\sqrt{p(\epsilon)/2})}(\mathcal{ E}^+) \oplus
B_{\sqrt{p(\epsilon)/2})}(\mathcal{ E}^-)\to U
$$
such that
$$
J(\phi(\lambda, x+
y))=\|x\|^2_{\lambda}-\|y\|^2_{\lambda}\quad\hbox{and}\quad
\phi(\lambda, x+ y)=(\lambda, \phi_\lambda(x+y))\in\mathcal{ E}
$$
for all $(\lambda, x+ y)\in  B_{\sqrt{p(\epsilon)/2})}(\mathcal{
E}^+) \oplus B_{\sqrt{p(\epsilon)/2})}(\mathcal{ E}^-)$. Moreover,
for each $\lambda\in\Lambda$,
$\phi_\lambda(\theta_\lambda)=\theta_\lambda$, $\phi_\lambda(x+y)\in
\mathcal{ E}^-_\lambda$ if and only if $x=\theta_\lambda$, and
$\phi$ is a preserving-fiber homoeomorphism from $B_{\sqrt{p(\epsilon)/2})}(\mathcal{
E}^-)$ onto $U\cap \mathcal{ E}^-$ according to any topology on both
induced by any Finsler structure on $\mathcal{ E}^-$.

\noindent{$3^\circ.$} In the above two cases, if
$G$ is a topological group and $p:\mathcal{ E}\to\Lambda$ is a
$G$-normed vector bundle such that the splitting $\mathcal{
E}=\mathcal{ E}^+\oplus\mathcal{ E}^-$, the functional $J$ and the
Finsler structure $\|\cdot\|$ are preserved, i.e.
\begin{equation}\label{e:A.10}
\left.\begin{array}{ll}
 J(g(\lambda,x))=J(\lambda, x),\quad
 \|gx\|_{g\lambda}=\|x\|_\lambda\\
\hbox{and}\quad gx\in\mathcal{ E}^+\;(\hbox{resp.}\; gx\in \mathcal{
E}^-)
\end{array}\right\}
\end{equation}
 for any $g\in G$
and $(\lambda, x)\in\mathcal{ E}^+$ (resp. $\mathcal{ E}^-$), then
the above homoeomorphism $\phi$ is $G$-equivariant, i.e.
$$
\phi(g(\lambda, x+y))=(g\lambda, \phi_{g\lambda}(gx+ gy))=(g\lambda,
g\phi_{\lambda}(x+ y))=g\phi(\lambda, x+y)
$$
for any $g\in G$ and $(\lambda, x+y)\in\mathcal{ E}^+\oplus\mathcal{
E}^-$.
\end{theorem}

\begin{proof}
 The case $1^\circ$ may be proved as in the proof of Theorem~\ref{th:A.1}.
 For the case $2^\circ$ the key  is the first two steps
corresponding with
 the proof of Theorem~\ref{th:A.1}. We can
slightly modify  the proof of \cite[Lemma 2.1]{DHK} to prove:
\begin{claim}\label{cl:A.3}
There exists a positive real number $\epsilon_1<\delta$ having the
following property: For each $(\lambda, x)\in
B_{\epsilon_1}({\mathcal{E}}^+)$ there exists a unique
$\varphi_\lambda(x)\in B_{\delta}({\mathcal{E}}^-)_\lambda$ such
that
\begin{equation}\label{e:A.11}
J(\lambda, x+\varphi_\lambda(x))=\max\{J(\lambda, x+y)\,|\, y\in
B_{\delta}(\mathcal{E}^-)_\lambda\}.
\end{equation}
\end{claim}

 In fact, the existence of $\epsilon_1$ can be obtained as follows.
Since $\bar B_{\delta}(\mathcal{E}^-)$ is compact, suppose by
contradiction that there exists a sequence $\{(\lambda_n, x_n)\}$ in
$B_{\delta}(\mathcal{E}^+)$ such that $(\lambda_n, x_n)\to
(\lambda_0, \theta_{\lambda_0})$ and a sequence
$\{y_n\}\subset\partial B_{\delta}(\mathcal{E}^-)_{\lambda_n}$ such
that
$$
J(\lambda_n, x_n+ y_n)> J(\lambda_n, x_n+ y)\quad\forall y\in
B_{\delta}(\mathcal{E}^-)_{\lambda_n},\;n=1,2,\cdots.
$$
We may assume $y_n\to y_0\in\partial B_{\delta}({\mathcal{
E}}^-)_{\lambda_0}$. Then
$$
\lim_{n\to\infty}J(\lambda_n, x_n+ y_n)=J(\lambda_0,
y_0)\quad\hbox{and}\quad \lim_{n\to\infty}J(\lambda_n,
x_n)=J(\lambda_0, \theta_{\lambda_0}).
$$
Hence $J(\lambda_0, y_0)\ge J(\lambda_0, \theta_{\lambda_0})$.
Moreover, by the mean value theorem and  Theorem~\ref{th:A.2}(iii)
there exists a $t\in (0, 1)$ such that
$$
J(\lambda_0, y_0)-J(\lambda_0, \theta_{\lambda_0})=DJ_{\lambda_0}(
t\cdot y_0)(y_0)=-\frac{1}{t}DJ_{\lambda_0}(t\cdot y_0)(-t\cdot
y_0)<0.
$$
This leads to a contradiction.

The uniqueness of $\varphi_\lambda(x)$ can also be proved by
contradiction.

Next, as in Step 2 of the proof of Theorem~\ref{th:A.1} above we can show that the map
$$
B_{\epsilon_1}(\mathcal{E}^+)\to
B_{\epsilon_1}(\mathcal{E}^-),\;(\lambda, x)\mapsto (\lambda,
\varphi_\lambda(x))
$$
is continuous. As in Step 4 above, for $(\lambda, x+y)\in
B_{\epsilon_1}(\mathcal{E}^+)\oplus B_{\delta}(\mathcal{E}^-)$ we
define
\begin{eqnarray*}
&&\psi_1(\lambda, x+y)=\left\{\begin{array}{ll}
 \frac{\sqrt{J(\lambda, x+ \varphi_\lambda(x))}}{\|x\|_\lambda}x &\;\hbox{if}\;x\ne \theta_\lambda,\\
 \theta_\lambda&\;\hbox{if}\;x=\theta_\lambda,
 \end{array}\right.\nonumber\\
&&\psi_2(\lambda, x+y)=\left\{\begin{array}{ll}
 \frac{\sqrt{J(\lambda, x+ \varphi_\lambda(x))-J(\lambda, x+y)}}{\|y-\varphi_\lambda(x)\|_\lambda}(y-\varphi_\lambda(x))
  &\;\hbox{if}\;y\ne\varphi_\lambda(x),\\
 \theta_\lambda&\;\hbox{if}\;y=\varphi_\lambda(x),
 \end{array}\right.
\end{eqnarray*}
and
\begin{eqnarray}\label{e:A.12}
\psi(\lambda, x+y)=\psi_1(\lambda, x+y)+\psi_2(\lambda, x+y).
\end{eqnarray}
They are continuous and $\psi(\lambda,
\theta_\lambda)=\theta_\lambda$. Let $\tilde\psi(\lambda,
x+y)=(\lambda, \psi(\lambda, x+y))$. As in Step 6 above  there is a positive real
number $\epsilon<\epsilon_1$ such that
$$
B_{\sqrt{p(\epsilon)/2})}(\mathcal{E}^+) \oplus
B_{\sqrt{p(\epsilon)/2})}(\mathcal{E}^-)\subset\tilde\psi\bigl(
B_{2\epsilon}(\mathcal{E}^+)\oplus B_{\delta}(\mathcal{E}^-)\bigr).
$$
Set
$$
U=\bigl( B_{2\epsilon}(\mathcal{E}^+)\oplus B_{\delta}({\mathcal{
E}}^-)\bigr)\cap
\tilde\psi^{-1}\left(B_{\sqrt{p(\epsilon)/2})}({\mathcal{E}}^+)
\oplus B_{\sqrt{p(\epsilon)/2})}(\mathcal{E}^-)\right)
$$
and
\begin{eqnarray}\label{e:A.13}
&&\phi:B_{\sqrt{p(\epsilon)/2})}(\mathcal{E}^+)
\oplus B_{\sqrt{p(\epsilon)/2})}(\mathcal{E}^-) \to U,\\
&&\hspace{10mm} (\lambda,x+y)\mapsto (\lambda,
\phi_\lambda(x+y)):=(\lambda, x'+ y'),\nonumber
\end{eqnarray}
where $(x', y')\in B_{2\epsilon}(\mathcal{E}^+)_\lambda\oplus
B_{\delta}({\mathcal{E}}^-)_\lambda$ is a unique point satisfying
$x+y=\psi(\lambda, x'+y')$. Except the final claim we leave the
remainder arguments to the reader.

As to the conclusion in $3^\circ$,   since $\|gx\|_{g\lambda}=\|x\|_\lambda$
for any $g\in G$ and $(\lambda,x)\in\mathcal{E}$, for any
$\varepsilon>0$ the sets $B_{\varepsilon}(\mathcal{E})$,
$B_{\varepsilon}(\mathcal{E}^+)$ and
$B_{\varepsilon}(\mathcal{E}^-)$ are $G$-invariant. For any $g\in G$
and $(\lambda, x)\in B_{\epsilon_1}(\mathcal{E}^+)$, by
Claim~\ref{cl:A.3} there exists a unique $\varphi_{g\lambda}(gx)\in
B_{\delta}(\mathcal{E}^-)_{g\lambda}$ such that
\begin{equation}\label{e:A.14}
J(g\lambda, gx+\varphi_{g\lambda}(gx))=\max\{J(g\lambda, gx+y)\,|\,
y\in B_{\delta}(\mathcal{E}^-)_{g\lambda}\}.
\end{equation}
Note that $g:B_{\delta}(\mathcal{E}^-)_{\lambda}\to
B_{\delta}({\mathcal{E}}^-)_{g\lambda},\;x\mapsto gx$ is a
homeomorphism. We conclude
\begin{eqnarray*}
\max\{J(g\lambda, gx+y)\,|\, y\in B_{\delta}({\mathcal{
E}}^-)_{g\lambda}\}&=&\max\{J(g\lambda, gx+ gy)\,|\, y\in
B_{\delta}(\mathcal{E}^-)_{\lambda}\}\\
&=&\max\{J(\lambda, x+ y)\,|\, y\in B_{\delta}({\mathcal{
E}}^-)_{\lambda}\}\\
&=&J(\lambda, x+\varphi_\lambda(x))\\
&=&J(g\lambda, gx+ g\varphi_\lambda(x)),
\end{eqnarray*}
where the third equality comes from (\ref{e:A.11}).  Since
$g\varphi_\lambda(x)\in B_{\delta}(\mathcal{E}^-)_{g\lambda}$ it
follows from this, (\ref{e:A.14}) and Claim~\ref{cl:A.3} that
$$
\varphi_{g\lambda}(gx)= g\varphi_\lambda(x)\quad\forall g\in
G\;\hbox{and}\;(\lambda,x)\in B_{\epsilon_1}(\mathcal{E}^+).
$$
Then the desired conclusion follows from this and
(\ref{e:A.12})-(\ref{e:A.13}).
\end{proof}


\section{ Several results on functional analysis}\label{app:B}\setcounter{equation}{0}

Perhaps the results in this appendix can be founded in some
references. For the readers's convenience we shall give proofs of
them. Let $E_1$ and $E_2$ be two real normed linear spaces and let
$T$ be a map from an open subset $U$ of $E_1$ to $E_2$.  For a positive integer $n$ we call $T$ {\it finite
$n$-continuous} at $x\in U$ if for any $h_1,\cdots, h_n\in E_1$ the
map
$$
\R^n\supseteq B^n(0, \epsilon)\ni t=(t_1,\cdots,t_n)\mapsto T(x+
t_1h_1+\cdots+ t_nh_n)
$$
is continuous at the origin $0\in\R^n$.

\begin{proposition}\label{prop:B.1}
\begin{enumerate}
\item[\bf (i)] If for any $u\in E_1$ the map $x\mapsto DT(x, u)$ is finite
$2$-continuous at $x_0\in U$ then $u\mapsto DT(x_0, u)$ is additive.

\item[\bf (ii)] If $T$ is continuously directional differentiable on
$U$ then it is strictly $H$-differentiable at every $x\in U$, and
restricts to a $C^1$-map on any finitely dimensional subspace. (So
the continuously directional differentiability is a notion between
the strict $H$-differentiability and $C^1$.)

\item[\bf (iii)]  If $T:U\to E_2$ is $G$-differentiable near $x_0\in U$
and also strictly $G$--differentiable at $x_0$, then $T'$ is
strongly continuous at $x_0$, i.e. for any $v\in E_1$ it holds that
$\|T'(x)v-T'(x_0)v\|\to 0$ as $\|x-x_0\|\to 0$. In particular, if
$E_2=\R$ this means that $T'$ is continuous  with respect to the
weak* topology on $E^\ast_1$.
\end{enumerate}
\end{proposition}

\begin{proof}
{\bf (i)} This directly follows from the
mean value theorem. In fact, for $u,v\in E_1$ and a small $t\ne 0$
let $ \triangle^2_{tu,tv}T(x_0)=T(x_0+ tu+ tv)-T(x_0+ tu)-T(x_0+
tv)+ T(x_0)$. Then
$$
\lim_{t\to 0}\frac{1}{t}\triangle^2_{tu,tv}T(x_0)=DT(x_0, u+v)-
DT(x_0, u)- DT(x_0, v).
$$
By the Hahn-Banach theorem there exists a functional $y^\ast\in
E_2^\ast$ such that $\|y^\ast\|=1$ and
$y^\ast(\triangle^2_{tu,tv}T(x_0))=\|\triangle^2_{tu,tv}T(x_0)\|$.
Applying twice the mean value theorem yields $\tau_1, \tau_2\in [0,
t]$ such that
\begin{eqnarray*}
&&\!\!\!y^\ast(T(x_0+ t u+ tv)-T(x_0+ t u)-T(x_0+ tv)+ T(x_0))\\
&=&\!\!\!\! y^\ast(DT(x_0+ tv+ \tau_1u, u))t-y^\ast(DT(x_0+ \tau_2u, u))t\\
 &\le&\!\!\!\! \|DT(x_0+ tv+ \tau_1u, u)-DT(x_0, u)\|\cdot |t|\\
  &+&\|DT(x_0+ \tau_2u, u)-DT(x_0,
 u)\|\cdot |t|.
\end{eqnarray*}
Since the map $x\mapsto DT(x, u)$ is finite $2$-continuous at
$x_0\in U$ it follows that
$$
\lim_{t\to
0}y^\ast\Bigl(\frac{1}{t}\triangle^2_{tu,tv}T(x_0)\Bigr)=0.
$$
Hence $DT(x_0, u+v)= DT(x_0, u)+ DT(x_0, v)$.

{\bf (ii)} Firstly, it follows from (i) that  $T$ is G\^{a}teaux
differentiable at every $x\in U$ if  $T$ is continuously directional
differentiable on $U$.

Next we prove that $T$ is strictly $G$-differentiable at every $x\in
U$. Otherwise, there exist $x_0\in U$, $v\in E_1$, $\varepsilon_0>0$
and sequences $\{x_n\}\subset U$ with $x_n\to x_0$,
$\{t_n\}\subset\R\setminus\{0\}$ with $t_n\to 0$, such that
$$
\left\|\frac{T(x_n+ t_nv)-T(x_n)}{t_n}-T'(x_0)v\right\|\ge
\varepsilon_0\quad\forall n=1,2,\cdots.
$$
As above we may use the Hahn-Banach theorem to get a sequence of
functionals $y^\ast_n\in E_2^\ast$ such that $\|y^\ast_n\|=1$ and
$$
y^\ast_n\left(\frac{T(x_n+ t_nv)-T(x_n)}{t_n}-T'(x_0)v
\right)=\left\|\frac{T(x_n+ t_nv)-T(x_n)}{t_n}-T'(x_0)v\right\|
$$
for any $n\in\N$. Then the mean value theorem yields a
sequence $\{\tau_n\}\subset (0, 1)$  such that
$$
y^\ast_n\left(\frac{T(x_n+ t_nv)-T(x_n)}{t_n}-T'(x_0)v
\right)=y^\ast_n\left(T'(x_n+ \tau_nt_nv)v-T'(x_0)v \right)
$$
$\forall n\in\N$. It follows that
$$
\|T'(x_n+ \tau_nt_nv)v-T'(x_0)v\|\ge\varepsilon_0\;\forall
n=1,2,\cdots.
$$
This contradicts to the continuously directional differentiability
of $T$.

Finally, suppose that $T$ is not strictly $H$-differentiable at some
$x_0\in U$. Then there exist a compact subset $K\subset E_1$,
$\varepsilon_0>0$, and and sequences $\{x_n\}\subset U$ with $x_n\to
x_0$, $\{t_n\}\subset\R\setminus\{0\}$ with $t_n\to 0$, such that for
some sequence $\{v_n\}\subset K$,
$$
\left\|\frac{T(x_n+ t_nv_n)-T(x_n)}{t_n}-T'(x_0)v_n\right\|\ge
\varepsilon_0\quad\forall n=1,2,\cdots.
$$
Since $K$ is compact we may assume $v_n\to v_0\in K$. As just we
have a sequence $\{s_n\}\subset (0, 1)$  such that $\|T'(x_n+
s_nt_nv_n)v-T'(x_0)v_n\|\ge\varepsilon_0$ for all $n\in\N$,
which leads to a contradiction.

The second claim can be derived from the fact that the strong
convergence and weak one are equivalent on finitely dimensional
spaces.

{\bf (iii)} Since $T$ is strictly $G$--differentiable at $x_0$, for
any $v\in E_1$ and $\varepsilon>0$ there exists a $\delta>0$ such
that
$$
\left\|\frac{T(x+ tv)-T(x)}{t}-T'(x_0)v\right\|< \varepsilon
$$
for any $t\in (-\delta,\delta)\setminus\{0\}$ and $x\in
B_X(x_0,\delta)$. Setting $t\to 0$ we get
$\|T'(x)v-T'(x_0)v\|\le\varepsilon\;\forall x\in B_X(x_0,\delta)$.
 \end{proof}

\begin{proposition}\label{prop:B.2}
Suppose that a bounded linear self-adjoint operator $B$ on a Hilbert
space $H$ has a decomposition $B=P+ Q$, where $Q\in L_s(H)$ is
compact and $P\in L_s(H)$ is positive, i.e., $\exists\; C_0>0$ such
that $(Pu, u)_H\ge C_0\|u\|^2\;\forall u\in H$. Then
 every $\lambda\in (-\infty, C_0)$ is either a regular value of $B$ or an
isolated point of $\sigma(B)$, which is also an eigenvalue of finite
multiplicity.
\end{proposition}

\begin{proof}
 Since  $(Pu-\lambda u, u)_H=(Pu,
u)_H-\lambda\|u\|^2\ge (C_0-\lambda)\|u\|^2$ for any  $\lambda\in
(-\infty, C_0)$ and $u\in H$, it follows from Theorem~9.1-2 in
\cite{Kr} that every $\lambda\in (-\infty, C_0)$ belongs to
$\rho(P)$. For such a $\lambda\in (-\infty, C_0)$,  observe that
$$
\lambda I_H- B=(\lambda I_H-P)[I_H- (\lambda I_H-P)^{-1}Q].
$$
So $\lambda I_H- B$ is Fredholm, and hence $\dim{\rm Ker}(\lambda
I_H- B)<\infty$, ${\rm codim}{\rm Ker}(\lambda I_H- B)<\infty$,
 and $R(\lambda I_H-B)\subset H$ is closed. By Theorem~4.5 on the
 page 150 of \cite{Sch}, either  $\lambda\notin
\sigma(B)$ or $\lambda$ is  an isolated point of $\sigma(B)$.
Clearly, in the latter case $\lambda$ is also an eigenvalue of $B$
with finite multiplicity.
\end{proof}

Actually, this result may also follow from
Proposition~\ref{prop:B.3} below.

 By Proposition~4.5 of
\cite{Con}, if $A$ is a continuous linear normal operator (i.e.
$A^\ast A=AA^\ast$) on a Hilbert space $H$, then for
$\lambda\in\sigma(A)$ the range
 $R(A-\lambda I)$ is closed if and only if $\lambda$ is not a limit
 point of $\sigma(A)$. As a consequence we deduce that  (i) and (ii) of the following
 proposition are equivalent.

\begin{proposition}\label{prop:B.3}
Let $H$ be a Hilbert space and let $A\in L(H)$ be a normal operator
(i.e. $A^\ast A=AA^\ast$). Then the following three claims are
equivalent.
\begin{enumerate}
\item[\bf (i)] $0$ is at most an isolated point of $\sigma(A)$;
\item[\bf (ii)] The range $R(A)$ is closed in $H$;
\item[\bf (iii)] The operator $A|_W: W\to W$ is invertible and
its inverse operator $(A|_W)^{-1}:W\to W$ is bounded, where $W=({\rm
Ker}(A))^\bot$.
 \end{enumerate}
\end{proposition}

 By the Banach inverse operator
theorem we arrive at (ii)$\Rightarrow$ (iii). Conversely,
$R(A)=A(W)=W$ is closed.

\section*{Acknowledgments}
 The author is deeply grateful to
 the anonymous referee for some interesting questions,  numerous comments
 and improved suggestions.

\end{document}